\documentclass[11pt]{article}
\usepackage{amssymb}
\usepackage{amsbsy}
\usepackage[latin1]{inputenc}
\usepackage{amsthm}
\usepackage{graphicx} 
\usepackage{subfigure}
\usepackage{pst-eucl}
\usepackage[latin1]{inputenc}
\usepackage[english]{babel}
\usepackage{amsmath,amssymb,graphics,mathrsfs}
\usepackage{amsmath,amssymb,latexsym}
\usepackage{float}
\usepackage{graphicx,color}
\usepackage[T1]{fontenc}
\usepackage[active]{srcltx}
\usepackage{multicol}
\usepackage[latin1]{inputenc}
\usepackage{pst-all}
\usepackage{enumerate}
\usepackage{pstricks}
\usepackage{pstricks-add}
\usepackage{setspace}
\usepackage{soul}
\usepackage{cancel}
\usepackage{epsfig}
\usepackage{pifont}
\usepackage{hyperref}
\usepackage{nonfloat}
\usepackage[margin=10pt,font=footnotesize,labelfont=bf,labelsep=endash]{caption}
\usepackage[left=4cm,top=3cm,right=2.4cm,bottom=3.2cm]{geometry}
\newcommand{\R}{\hbox{\rm I \kern -5pt R}}     
\newcommand{\p} {\hbox{\rm I \kern -5pt P}}
\def\x  {\boldsymbol x}

\def\H        {{\boldsymbol H}}
\def\L       {{\boldsymbol L}}

\def\Om       {\Omega}

\def\nn       {{\boldsymbol n}}
\def\bsigma       {{\boldsymbol \sigma}}

\def \hueco{\noalign{\medskip}}
\def \pato{\forall\,}
\def \beq{\begin{equation}}
\def \eeq{\end{equation}}
\def \ba{\begin{array}}
\def \ea{\end{array}}
\def \dis{\displaystyle}

\newtheorem{prop}{Proposition}[section]
\newtheorem{tma}[prop]{Theorem}
\newtheorem{cor}[prop]{Corollary}
\newtheorem{obs}[prop]{Remark}

\allowdisplaybreaks

\begin{document}

\title{Finite Element numerical schemes for a chemo-attraction and consumption model\thanks{This research work
has been supported by \textit{Proyecto PGC2018-098308-B-I00, financed by  FEDER / Ministerio de Ciencia e Innovaci\'on - Agencia Estatal de Investigaci\'on, Spain.}} } 
\author{F.~Guill\'en-Gonz\'alez
\thanks{Dpto. Ecuaciones Diferenciales y An\'alisis Num\'erico and IMUS, Universidad de Sevilla, Facultad de Matem\'aticas, C/ Tarfia, S/N, 41012 Sevilla (SPAIN). Email: guillen@us.es},
\and
 G.~Tierra
 \thanks{Department of Mathematics, University of North Texas. Email:  gtierra@unt.edu}
 }

\maketitle

\begin{abstract}
This work is devoted to design and study efficient and accurate numerical schemes to approximate a chemo-attraction model with consumption effects, which is a nonlinear parabolic system for two variables; the cell  density 
and the concentration of the chemical signal that the cell feel attracted to. 

We present several finite element   schemes to approximate the system, detailing the main properties of each of them, such as conservation of cells, energy-stability and approximated positivity. Moreover, we carry out several numerical simulations to study the efficiency of each of the schemes and to compare them with others classical schemes.
\end{abstract}




\section{Introduction}

Chemotaxis can be defined as the orientation or movement of an  organism or cell in relation to chemical agents. This movement can be towards a higher (attractive) or lower (repulsive) concentration of the chemical substance. 
At the same time, the presence of living organisms can produce or consume the chemical substance, producing nontrivial dynamics of the living organisms and  chemical substances.
In 1971, Keller and Segel \cite{KellerSegel} introduced the model (\ref{eq:KellerSegel}) that can be considered the first realistic attempt to capture the chemotactic response of bacteria towards chemical agents in a bounded spatial domain $\Omega\subset \mathbb{R}^d$ ($d=1,2,3$) during a time interval $[0,T]$, where the cell population density $u(\x,t)$ moves towards the concentration of the chemical substance $v(\x,t)$, which is produced by the cell population with a rate $\mu>0$:

\begin{equation}\label{eq:KellerSegel}
\left\{\begin{array}{rcl}
u_t - \Delta u + \chi\nabla\cdot(u\nabla v) &=&0\,,
\\ \hueco
\alpha\,v_t - \Delta v + v- \mu \, u&=& 0 \,,
\end{array}\right.
\end{equation}
where $\chi>0$ denotes the chemo-sensitivity parameter, $\alpha\in\{0,1\}$ determines the character of the chemical equation, being parabolic when $\alpha=1$ and elliptic when $\alpha=0$.
Since then, many models following the same spirit have been proposed and studied mathematically (check \cite{Bellomo-et-al,HorstmannI,HorstmannII} for reviews on the development of this topic during the last years).

In this work we focus on developing numerical schemes for a system where the
cell population  is attracted by the chemical substance, which is consumed by the cell population with a rate proportional to the amount of living organisms:

\begin{equation}\label{eq:atractioncomsuption}
\left\{\begin{array}{rcl}
u_t -  \Delta u + \chi \,\nabla\cdot (u \nabla v)&=&0\,
 \\ \hueco
 v_t -  \Delta v + \mu\, u\,v &=& 0\,.
\end{array}\right.
\end{equation}

There are several works that have focused on studying the analytical properties of model \eqref{eq:atractioncomsuption}.
In \cite{Tao} the corresponding $d$-dimensional problem was studied for $d\ge 2$, proving global (in time) regularity and uniqueness whether the following criterium holds
$$
0<\chi\le \frac1{6(d+1)\|v_0\|_{L^\infty(\Om)}}.
$$   
  For $3$-dimensional domains and  arbitrarily large initial data, in  \cite{TaoWinkler} it is showed that this type of system admits at least one global weak solution. Moreover, it is asserted that such solutions at least \textit{eventually} (i.e., for large enough times) become bounded and smooth, and that they approach the unique relevant constant steady state in the large time limit.
  
  \

In recent times several groups have focused on numerical analysis for this type of attractive chemotaxis models with consumption/production of chemical substance and many relevant works have been produced. It is not our intention to provide a detailed review of all the works that have been produced in recent years, but rather to provide the reader with some interesting works that are related with the work that we present in this text. 
In \cite{Zuhr} the authors investigate nonnegativity of exact and Finite Element (FE) numerical solutions to a generalized Keller-Segel model, under certain standard assumptions. 
Marrocco presented in \cite{Marrocco} a new formulation of the Keller-Segel system, based on the introduction of a new variable and he approximated this new system via a mixed FE technique. 
Saito in \cite{Saito} focused on presenting an error analysis of an approximation for the Keller-Segel system using a semi-implicit time discretization with a time-increment control and Baba-Tabata's conservative upwind FE approximation \cite{BabaTabata}, that allows to  show the positivity and mass conservation properties of the scheme. 

\

Several works have also focused on numerical schemes to approximate the simplified Keller-Segel system, where the parabolic equation for the concentration of the chemical substance is replaced by an elliptic one (taking $\alpha=0$ in \eqref{eq:KellerSegel}), arriving at a parabolic-elliptic system.
Filbet presented in \cite{Filbet} a well-posed Finite Volume scheme to discretize the simplified Keller-Segel, providing a priori estimates and convergence.
In \cite{ZhouSaito} a Finite Volume approximation of the simplified Keller-Segel system  is also considered, presenting a linear scheme that satisfies both positivity and mass conservation, deriving some inequalities on the discrete free energy and under some assumptions they establish error estimates in $L^p$ norm with a suitable $p > 2$ for the $2$-dimensional case. 
 A simplified Keller-Segel system with additional cross diffusion is presented in  \cite{BessemoulinJungel}. The main feature of this model is that there exists a new entropy functional yielding gradient estimates for the cell density and chemical concentration. The authors also present in  \cite{BessemoulinJungel} a Finite Volume scheme that satisfy positivity preservation, mass conservation, entropy stability, and (under additional assumptions) entropy dissipation. Moreover, the existence of a discrete solution and its numerical convergence to the continuous solution is proved. 
 
 \

There are more related models that have been studied recently. For instance, one related type of models are the ones that focus on repulsive chemotaxis systems with the cell population producing chemical substance, that is, a system like \eqref{eq:KellerSegel} with $\chi<0$.  We refer the reader to \cite{FMD2,FMD4} (and the references therein) where the authors focused on studying unconditionally energy stable and mass-conservative FE numerical schemes, by  introducing  the gradient of the chemical concentration variable,  for chemo-repulsive systems with quadratic ($-\mu u^2$) and linear production terms ($-\mu u$) in $\eqref{eq:KellerSegel}_2$, respectively.

\

On the other hand, it has been experimentally observed \cite{Tuval} that the chemotactic motion in liquid environments affects substantially the migration of cells and this fact has increased the interest of studying the coupling of chemotaxis systems with the Navier-Stokes equations. For instance,
Winkler in \cite{Winkler} has studied the $d$-dimensional problem ($d=2,3$) of attractive chemotaxis models with consumption of chemical substance, showing that under suitable regularity assumptions on the initial data, the chemotaxis-Navier-Stokes system admits a unique global classical solution ($d=2$) and the simplified chemotaxis-Stokes system possesses at least one global weak solution ($d=3$). Moreover, in \cite{Duarteetal} the authors construct numerical approximations for the same type of system. The presented approximations are based on using the Finite Element method, obtaining optimal error estimates and convergence towards regular solutions. Finally, we would like to mention another related work \cite{GTCH}, where Finite Elements together with singular potentials have been used in the context of the Cahn-Hilliard equation to achieve energy-stable numerical schemes that satisfy approximated positivity properties.

\

This work is organized as follows. In Section~\ref{sec:model} we present the attractive chemotaxis with consumption model that we have considered, its main properties and a reformulation that will allow us to design numerical schemes satisfying some energy laws, getting in particular an energy stable scheme in one-dimensional domains. The numerical schemes are developed and studied in Section~\ref{sec:schemes}.  In Section~\ref{sec:simulations} we report some numerical experiments that we have performed to study the efficiency and the accuracy of the schemes and to compare them with others classical schemes. Finally, the conclusions of our work are presented in Section~\ref{sec:conclusions}.
\section{The model}\label{sec:model}

In this work, we consider an  attractive-consumption chemotaxis model in a bounded domain $\Omega\subset \mathbb{R}^{d}$ $(d=1,2,3)$ given by the following parabolic  system of PDEs:
\begin{equation}  \label{modelf00}
\left\{
\begin{array}{rll}
 u_t -  \Delta u + \chi \,\nabla\cdot (u \nabla v)=0
 & \quad \mbox{in}\ \Omega\,,& t>0\,,
 \\ \hueco
 v_t -  \Delta v +\mu\, u\,v = 0
& \quad \mbox{in}\  \Omega\,,& t>0\,,
  \\ \hueco
 \displaystyle
 \nabla u \cdot \nn=  \nabla v \cdot \nn=0
 & \quad \mbox{on}\ \partial\Omega\,,& t>0\,,
  \\ \hueco
u(0,\x)=u_0(\x )\geq 0,\ v(0,\x)=v_0(\x )> 0
& \quad \mbox{in}\ \Omega\,,&
\end{array}
\right. \end{equation}

where  $u(t,\x) \geq 0$ denotes the cell population density and $v(t,\x) > 0$ denotes the concentration of the chemical substance that the cell feel attracted to and $\nn$ denotes the outward normal vector to $\partial\Omega$. Moreover, $(u_0,v_0)$ represents the initial density and concentration, $\chi>0$ is the chemo-sensitivity coefficient and $\mu>0$ is the consumption one.

\subsection{Properties of the model}
It is known that any  regular enough solution $(u,v)$ of problem \eqref{modelf00} satisfies the following properties:
\begin{enumerate}
\item
\textbf{Positivity of $u$ and strictly positivity of $v$} (\cite{Tao})
\\ If $u_0\ge 0$ in $\Om$ then $u(t,\cdot)\ge 0$ in $\Om$ for any $t>0$. Assuming $u\in L^\infty((0,\infty)\times \Om)$ and $v_0\geq v_0^{min}> 0$ in $\Om$ then $v(t,\cdot)> 0$ in $\Om$ for any $t>0$. In fact, one has the lower bound 
$$v(t,\cdot)\geq v_0^{min}\exp(-\mu\,t\,\|u\|_{L^\infty((0,t)\times \Omega)}).$$
 
\item \textbf{Maximum principle for $v$} (\cite{Evans}) 
\\One has $0 \leq  v(t,\cdot) \le \|v_0\|_{L^\infty(\Om)}$ in $\Om$ for any $t>0$. In fact, $ \|v(t,\cdot)\|_{L^\infty(\Om)}$ is a non-increasing function.

\item \textbf{Cell density conservation.} Integrating equation (\ref{modelf00})$_1$,
\begin{equation*}
\frac{d}{dt}\left(\int_\Om u(t,\cdot)\right)=0\,, 
\quad 
\mbox{that is,} 
\quad 
\int_\Om u(t,\cdot)=\int_\Om u_0\,, 
\quad\pato t>0\,.
\end{equation*}

\item
\textbf{Weak regularity for $ v$}. 
Testing equation (\ref{modelf00})$_2$ by $ v$, 
one has 
\beq\label{v-L2-reg}
\frac{d}{dt} \|v(t,\cdot)\|_{L^2(\Om)}^2 + \|\nabla v(t,\cdot)\|_{L^2(\Om)}^2 \le 0\,.
\eeq

%
%
%
%
%
%
%
%
\item \textbf{Estimate of a  singular functional.}  
Assuming $u(t,x)>0$ for all $(t,x)$, one has the time differential inequality
\beq\label{eq:regresult}
\frac{d}{dt}\left(\int_\Omega G(u)d\x\right) 
+ \frac12\int_\Omega \frac1{u^2} |\nabla u|^2d\x
\, \le \,
\frac{\chi^2}2 \int_\Omega  |\nabla v|^2 d\x\,, 
\eeq
 where 
 $$G(u)=-\log(u) + C$$ 
 is a convex function and the right hand side of 
 \eqref{eq:regresult} 
 belongs to $L^1(0,+\infty)$  owing to \eqref{v-L2-reg}. 
 Relation \eqref{eq:regresult} is derived testing equation \eqref{modelf00}$_1$ by $G'(u)=-1/u$,
$$
\frac{d}{dt}\left(\int_\Omega G(u)d\x\right)
 + \int_\Omega \frac1{u^2} |\nabla u|^2 d\x
\,=\, 
-\chi \int_\Omega \frac1u  \nabla u \cdot\nabla v \,d\x\,,
$$
hence \eqref{eq:regresult} holds by using H\"older inequality. This differential inequality  \eqref{eq:regresult}  is analogous to the one derived in \cite{Winkler16} for a Keller-Segel system with singular sensitivity.  
\item
\textbf{Energy law} \cite{Winkler}.  Assuming $u(t,x)>0$ for all $(t,x)$, 
one has  the energy law:
\beq\label{eq:EL-uv}
\dis \frac{d}{dt} E(u,v)
\,+\mu\, D_1(u) \,+\chi\, D_2(v)  +\mu\,\chi\, D_3(u,v) \,=\, \chi R(v)\,,
\eeq
where
$$
\ba{rcl}
E(u,v)&:=&\dis\frac\mu  4 \int_\Om F(u)d\x  + \frac{\chi}2 \int_\Om | \nabla( \sqrt{v}) |^2 d\x\,,
\quad
\hbox{with $F'(u)=\log (u)$,}
\\ \hueco
D_1(u)&:=&\dis\frac14\int_\Omega \frac1 u |\nabla u|^2 d\x\,=\, \int_\Omega |\nabla(\sqrt{u})|^2 d\x\geq 0\,,
\\ \hueco
D_2(v) &:=&\dis \int_\Omega (\Delta(\sqrt{v}))^2 d\x
+\ \frac13 \int_\Omega\frac1{v}|\nabla(\sqrt{v})|^4d\x\geq 0\,,
\\ \hueco
D_3(u,v)&:=&\dis \frac12 \int_\Omega  u\, |\nabla(\sqrt{v})|^2d\x\geq 0\,,
\\ \hueco
R(v)&:=& -\dis \frac{2\chi}{3}\int_\Omega \frac{1 }{\sqrt{v}} \Big[\nabla\cdot \Big(\big|\nabla(\sqrt{v})\big|^2\nabla(\sqrt{v})\Big) - 3\big(\nabla(\sqrt{v})\big)^t\Big(\nabla \big(\nabla(\sqrt{v})\big)\Big)^t \nabla(\sqrt{v})\Big] d\x\,,
\ea
$$
Energy law \eqref{eq:EL-uv} can be proved following the same ideas presented in Theorem~\ref{th:stability} below. Moreover, in the particular case of $1$-dimensional domains, the energy $E(u,v)$ is dissipative due to the right-hand term of  \eqref{eq:EL-uv} vanishes. In fact, 
$$
R(v)\,=\,
 -\frac{2\chi}{3}\int_\Omega \frac{1 }{\sqrt{v}} \Big[\Big( \big((\sqrt{v})_x\big)^3\Big)_x 
 - 3 \big((\sqrt{v})_x\big)^2 (\sqrt{v})_{xx}\Big] dx
\,=\,0\, .
$$
%
%
%
On the contrary, in higher dimensions it is not clear the sign of $R(v)$, preventing the possibility of obtaining a dissipative energy law without introducing constraints on the physical parameters.
\end{enumerate}
\begin{obs}
Notice that functional $G(u)=-\log(u) + C$ in the inequality \eqref{eq:regresult} is more singular (for $u=0$) than the energy potential $F(u)=u \log(u)-u +1$ in \eqref{eq:EL-uv}. These singularities will be crucial to prove the approximate positivity of some of the numerical schemes presented in this work. A similar idea has been considered in \cite{GTCH} in the context of the Cahn-Hilliard equation.
\end{obs}

\subsection{Reformulation of the problem. The $(u,v,\bsigma)$ problem}


%
%
%
%

In order to develop a numerical scheme satisfying a discrete version of the energy estimate \eqref{eq:EL-uv}, we need to reformulate problem \eqref{modelf00}. The idea is to rewrite the $v$-equation \eqref{modelf00}$_2$ multiplying it by $w'(v)=1/(2\sqrt{v})$ (hence $w(v)=\sqrt{v})$, as follows
$$
(\sqrt{v})_t 
\,-\,  \frac{\Delta((\sqrt{v})^2)}{2\sqrt{v}} 
\,+\, \frac\mu 2 u\sqrt{v} 
\,=\, 0\,.
$$
Then, by introducing the notation $w:=\sqrt{v}>0$ (due to the positivity of $v$) we obtain the PDE 
\beq\label{eq:w_equation}
w_t 
\,-\,
\frac1w |\nabla w|^2 
\,-\,
\Delta w 
\,+\,
\frac\mu 2 uw 
\,=\,
0\,.
\eeq
Now we can take the gradient with respect to $\x$ of \eqref{eq:w_equation} to obtain:
\beq\label{eq:der_x_weq}
(\nabla w)_t 
\,+\,
\frac{1}{w^2} |\nabla w|^2\nabla w 
\,-\,
\frac{1}{w} \nabla \big(|\nabla w|^2\big) 
\,-\,
\nabla\Delta w 
\,+\,
\frac\mu 2  u \nabla w
\,+\,
\frac\mu 2 w\,\nabla u
\,=\,
0\,.
\eeq
Introducing a new unknown $\bsigma:=\nabla w=\nabla (\sqrt v)$ we can reformulate three of the terms in \eqref{eq:der_x_weq}:
\begin{itemize}
\item 
Use the definition of $\bsigma$ to rewrite the term $\dis\frac{1}{w^2} |\nabla w|^2\nabla w$ as 
$$
\frac{1}{w^2} |\nabla w|^2\nabla w 
\,=\,
\frac{1}{3v} |\bsigma|^2\bsigma
+ \frac{2}{3v} |\bsigma|^2 \nabla (\sqrt{v})\,.
$$
\item
Use $\sigma$ and a truncation of $u$ to rewrite term $\dis\frac\mu 2 \nabla w\, u$ as $\dis\frac\mu 2 \bsigma\, u_+$, where $u_+$ denotes the positive part of $u$ ($u_+(\x)=\max\{u(\x),0\}$).
\item
Rewrite the term $\dis\frac\mu 2 w\,\nabla u$ as $\dis\frac\mu 2 \sqrt{v}\, \, \nabla u$. 
\end{itemize}

Using these considerations, we formally arrive at the following $(u,v, \bsigma)$  reformulation of the problem \eqref{modelf00}:
\begin{equation}\label{eq:systemreftrun2}
\left\{
\begin{array}{rll}
u_t -  \Delta u + 2\chi \nabla\cdot(u\, \sqrt{v} \, \bsigma) =  0
& \quad \mbox{in}\ \Omega\,, & t>0\,,
\\ \hueco\dis 
v_t - \Delta v + \mu \,u\,v
 =  0 
 &  \quad\mbox{in}\ \Omega\,, & t>0\,,
\\ \hueco\dis 
\bsigma_t 
-\frac2{\sqrt{v}} (\nabla\bsigma)^t\bsigma
+ \frac{1}{3v} |\bsigma|^2\bsigma
+ \frac{2}{3v}  |\bsigma|^2 \nabla (\sqrt{v})
&&
\\ \hueco\dis 
- \nabla(\nabla\cdot\bsigma)
+ \mbox{rot}(\mbox{rot}\, \bsigma)
+ \frac\mu 2 \bsigma\, u_+ 
+ \frac\mu 2 \sqrt{v}\, \, \nabla u
 =  0 
 & \quad \mbox{in}\ \Omega\,, & t>0\,,
\\ \hueco
\displaystyle
\nabla u\cdot \nn=\nabla v\cdot \nn=\bsigma\cdot \nn=(\mbox{rot} \,\bsigma\times \nn)\big|_{tang}=0
& \quad \mbox{on}\ \partial\Omega\,, &  t>0\,,
\\ \hueco
\displaystyle
u(\x ,0)=u_0(\x )\geq 0\,, 
v(0,\x)=v_0(\x )> 0\,, 
\bsigma(\x ,0)=\bsigma_0(\x )
&\quad \mbox{in}\ \Omega\,, & 
\end{array}
\right.
\end{equation}  
where $\mbox{rot}\,\bsigma$ denotes the well-known rotational operator (also called curl) which is scalar for 2D domains and vectorial for 3D ones. We have introduced the term $\mbox{rot}(\mbox{rot}\, \bsigma)(=\mbox{rot}(\mbox{rot}\, \nabla v)=0)$ as in \cite{FMD4} to complete the $\H^1_{\bsigma}$-norm of $\bsigma$, where: 
$$
\H^1_{\bsigma}(\Omega)\,:=\,\{\bsigma\in \H^1(\Omega): \bsigma\cdot\nn = 0 \,\mbox{ on } \partial\Omega \}\,,
$$
and
$$
\|\bsigma\|_{\H_\bsigma^1}^2 
\,:=\, 
\|\bsigma\|_{\L^2}^2 
 + \|\mbox{rot}\,\bsigma\|_{\L^2}^2 
 + \|\nabla\cdot\bsigma\|_{L^2}^2\,,
$$
with $\|\cdot\|_{\H_\bsigma^1}$ being equivalent to $\|\cdot\|_{\H^1}$.
%

\

\begin{tma}\label{th:stability}
Any  regular enough solution  $(u,v, \bsigma)$ of \eqref{eq:systemreftrun2} satisfy the following energy law:
\beq\label{eq:ELsigma}
\dis \frac{d}{dt} {E}(u,\bsigma)
\,+ \mu\,{D}_1(u) 
\,+ \chi\,{D}_2(v,\bsigma) 
\,+\chi\mu \,{D}_3(u,\bsigma)
\,=\, R(v,\bsigma)\,,
\eeq
where
\beq\label{eq:Energysigma}
\ba{rclrcl}
{E}(u,\bsigma)&:=&\dis \frac{\mu}{4}\int_\Omega F(u)d\x\,+\, \frac{\chi}{2}\int_\Omega |\bsigma|^2 d\x\,,
&
{D}_1(u)&:=&\dis \frac{1}{4}\int_\Omega \frac1 u |\nabla u|^2 \, d\x\,,
\\ \hueco
{D}_2(v,\bsigma)&:=&\dis \int_\Omega \Big(|\nabla\cdot\bsigma|^2 + |\mbox{rot}\,\bsigma|^2\Big) d\x 
+  \frac{1}{3}\int_\Omega \frac{1 }{v} |\bsigma|^4 d\x\, ,
&
{D}_3(u,\bsigma)&:=&\dis \frac{1} 2 \int_\Omega u_+\, |\bsigma|^2 d\x\,.
\\ \hueco
R(v,\bsigma)&:=&-\dis \frac{2\chi}{3}\int_\Omega \frac{1 }{\sqrt{v}} \Big(\nabla\cdot (|\bsigma|^2\bsigma) - 3\bsigma^t(\nabla \bsigma)^t \bsigma\Big) d\x\,,
&&
\ea
\eeq
with ${D}_1(u), \,{D}_2(v,\sigma), \,{D}_3(u,\sigma)\geq0$.
\end{tma}

\begin{proof}
The key argument of this proof is to test by functions that allow us to relate the chemotaxis term in \eqref{eq:systemreftrun2}$_1$ with one of the consumption terms in \eqref{eq:systemreftrun2}$_2$. With this idea in mind, we test \eqref{eq:systemreftrun2}$_1$ by $\dis\frac{\mu}{4}F'(u)$ 
to obtain
\beq\label{eq:pruestasig1}
\frac{d}{dt}\left(\frac{\mu}{4}\int_\Omega F(u)d\x\right) 
\,+\, \mu\,{D}_1(u)
\,-\, \frac{\chi\mu}2 \int_\Omega  u\,\sqrt{v} \,\bsigma \cdot\nabla(F'(u)) \,d\x
\,=\, 0\,.
\eeq
Testing \eqref{eq:systemreftrun2}$_3$ by $\chi \bsigma$ we obtain
\beq\label{eq:pruestasig2}
\ba{c}\dis
\frac{d}{dt}\left(\frac{\chi}{2}\int_\Omega |\bsigma|^2 d\x\right) 
\,-\, 2\chi\int_\Omega\frac{1 }{\sqrt{v}} \bsigma^t(\nabla \bsigma)^t \bsigma \,d\x
\,+\, \frac{\chi}3\int_\Omega \frac{1}{v}|\bsigma|^4 d\x
\,+\, \frac{2\chi}3\int_\Omega \frac{1 }{v}|\bsigma|^2 \nabla (\sqrt{v}) \cdot\bsigma d\x
\\ \hueco\dis
\,+\, \chi\int_\Omega (\nabla\cdot\bsigma)^2 d\x
\,+\, \chi\int_\Omega |\mbox{rot}\,\bsigma|^2 d\x
\,+\, \frac{\chi\mu} 2 \int_\Omega u_+ \, |\bsigma|^2  d\x
\,+\, \frac{\chi\mu} 2 \int_\Omega \sqrt{v} \,\nabla u\, \cdot \,\bsigma \,d\x
\,=\, 0\,.
\ea
\eeq
Using integration by parts we rewrite the fourth term of \eqref{eq:pruestasig2} as
$$
 \frac{2\chi}3\int_\Omega \frac{1 }{v}|\bsigma|^2 \nabla (\sqrt{v}) \cdot\bsigma d\x
\,=\,
-\frac{2\chi}{3}\int_\Omega |\bsigma|^2 \nabla\left(\frac{1 }{\sqrt{v}}\right) \cdot\bsigma d\x
\,=\,
\frac{2\chi}{3}\int_\Omega \frac{1 }{\sqrt{v}} \Big(\nabla\cdot (|\bsigma|^2\bsigma)\Big) d\x\,,
$$
Finally, 
using $\nabla u=u\nabla (\ln (u))=u\nabla (F'(u))$ and
adding equations \eqref{eq:pruestasig1} and \eqref{eq:pruestasig2}, 
the terms $\int_\Omega  u\,\sqrt{v} \,\bsigma \cdot\nabla(F'(u)) \,d\x$ and 
$\int_\Omega \sqrt{v} \,\nabla u\, \cdot \,\bsigma \,d\x$ cancel, and 
the desired relation \eqref{eq:ELsigma} holds.
\end{proof}

\

\begin{cor} 
In the particular case of considering one-dimensional domains ($1D$), the right side of relation \eqref{eq:ELsigma} vanishes, implying the following energy dissipative law of the system:
\beq\label{eq:ELsigma1d}
\dis \frac{d}{dt} {E}(u,\sigma)
\,+ \mu\,{D}_1(u) 
\,+ \chi\,{D}_2(v,\sigma) 
\,+\chi\mu \,{D}_3(u,\sigma)
\,=\, 0\,\,.
\eeq
\end{cor}
\begin{proof}
Since, in one-dimensional domains variable $\sigma$ is a scalar quantity, then  the term $R(v,\sigma)$ reads:
\beq
R(v,\sigma)
\,=\,
\dis -\frac{2\chi}{3}\int_\Omega \frac{1 }{\sqrt{v}} \Big(\partial_x(\sigma^3) - 3\sigma(\partial_x \sigma)\sigma\Big) dx
\,=\,
0\,.
\eeq

\end{proof}

\section{Numerical Schemes}\label{sec:schemes}

We discretize the time interval $[0,T]$ using Finite Differences and the spatial domain $\Omega\subset\mathbb{R}^d$ $d=1,2,3$ using Finite Elements with a shape-regular and quasi-uniform family of triangulations of $\Omega$, denoted by $\{\mathcal{T}_h\}_{h>0}$. 
For the sake of simplicity we consider a constant time step $\Delta t:=T/N$, where $N$ represents the total number of time intervals considered and we denote by $\delta_t$ the (backward) discrete time derivative
$$
\delta_t u^{n+1}\,:=\, \frac{u^{n+1} - u^n}{\Delta t}\,.
$$
 In our analysis we will consider $C^0$-FE spaces of order $1$ (denoted by $\mathbb{P}_1$) for the approximation of $(u,v,\bsigma)$ via the discrete spaces $U_h$, $V_h$ and $\Sigma_h$. Additionally, we consider in our schemes \textit{mass-lumping} ideas \cite{CiarletRaviart} to help us achieve positivity of the unknowns in some of the proposed schemes. In order to do that, we introduce the discrete semi-inner product on $C^0(\overline\Omega)$ and its induced discrete seminorm:
\beq
(\phi,\psi)_h\,:=\,\int_\Omega I_h(\phi\psi)\,\qquad\mbox{ and }\qquad |\phi|_h\,:=\,\sqrt{(\phi,\phi)_h}\,.
\eeq
with $I_h(f(x))$ denoting the nodal $\mathbb{P}_1$-interpolation of the function $f(x)$.

%
%
\subsection{UV-schemes}\label{sec:esquemaC0uv}

In this section we present three different schemes to approximate  the weak formulation that appears when we test \eqref{modelf00}$_1$ and \eqref{modelf00}$_2$ by regular enough test functions $\bar{u}
$ and  $\bar{v} 
$, that is, find $(u,v):[0,T]\times \overline\Omega \mapsto \mathbb{R}^2
$ such that for all $(\bar{u},\bar{v}): \overline\Omega \mapsto \mathbb{R}^2
$:
\beq\label{eq:weakformulationC0UV}
\left\{\ba{rcl}
(u_t,\bar{u})
+ (\nabla u, \nabla \bar{u})
-  \chi\,( u \nabla v,\nabla\bar{u})
&=&0\,,
\\ \hueco
(v_t , \bar{v})
+ (\nabla v , \nabla \bar{v})
+ \mu\,(uv, \bar{v}) 
&= &0\,.
\ea\right.
\eeq
\subsubsection{Scheme UV}
\begin{itemize}
\item~[\textbf{Step 1}] 
Given $(u^n,v^n)\in U_h\times V_h$, find $u^{n+1}\in U_h=\mathbb{P}_1$ solving the linear  problem:
\begin{equation}\label{eq:schemeC0-u}
\dis\left(\delta_t u^{n+1} ,\bar{u}\right)_h
 +  \Big(\nabla u^{n+1},\nabla\bar{u}\Big) 
 - \chi \Big( u^{n+1} \nabla v^n , \nabla\bar{u}\Big)
 =  0
 \quad\pato \bar{u} \in U_h\,.
\end{equation}
\item~[\textbf{Step 2}]
Given $(u^{n+1},v^n)\in U_h\times V_h$, find $v^{n+1}\in V_h=\mathbb{P}_1$ solving the linear problem:
\begin{equation}\label{eq:schemeC0-v}
\dis\left(\delta_t v^{n+1} ,\bar{v}\right)_h
 +  \Big(\nabla v^{n+1},\nabla \bar{v}\Big) 
 +\mu \Big( (u^{n+1})_+ v^{n+1} , \bar{v}\Big)_h
 =  0
 \quad\pato \bar{v} \in V_h\,.
\end{equation}
 \end{itemize} 
Note that scheme \textbf{UV}  is linear, decoupled  and conservative, because taking $\bar{u}=1$ in \eqref{eq:schemeC0-u},  
$$
\int_\Omega u^{n+1} d\x
\,=\,\int_\Omega u^{n} d\x
\,=\,\cdots
\,=\,\int_\Omega u^{0} d\x
\,=:\,m_0\,.
$$

\


\begin{tma}[$u^{n+1}$-problem]\label{tma_UVestimatesu}
Given $v^n\in V_h$, assuming $\Delta t$ small enough satisfying 
\begin{equation} \label{hyp-dt<<}
\Delta t< \frac{2}{\chi^2\|\nabla v^n\|_\infty},
\end{equation}
then there exist a unique solution $u^{n+1}\in U_h$ solving \eqref{eq:schemeC0-u}. 
\end{tma}
\begin{proof}
Since problem \eqref{eq:schemeC0-u} is a squared algebraic  linear system, it suffices to prove uniqueness (that implies existence). In fact, it suffices to prove that if $u\in U_h$ is a solution of the corresponding homogeneous problem 
$$
\frac1{\Delta t}(u, \bar u)_h 
+ (\nabla u,\nabla \bar u )
- \chi\Big( u \nabla v^n,\nabla \bar u \Big)
=0,
$$
then $u=0$. In fact, testing by $\bar u=u$ and applying Holder and Young inequalities, one has 
$$
\dis\frac1{\Delta t}\int_\Omega I_h\Big(u^2 \Big)d\x
+ \frac12\int_\Omega |\nabla u|^2d\x
\le
\dis\frac{\chi^2}2\int_\Omega u^2|\nabla v^n|^2d\x
\le
\dis\frac{\chi^2}2\|\nabla v^n\|_\infty\int_\Omega u^2d\x\, .
$$

Using the relation 
$
\int_\Omega u^2 d\x 
\leq 
\int_\Omega I_h\Big(u^2 \Big)d\x
$
(see \cite{KiskoJV}),
we obtain
$$
\dis\left(\frac1{\Delta t} - \frac{\chi^2}2\|\nabla v^n\|_\infty\right)\int_\Omega u^2 d\x
+ \frac12\int_\Omega |\nabla u|^2d\x
\leq
0\,.
$$
Therefore, 
assuming hypothesis \eqref{hyp-dt<<}, 
we have $u=0$.
\end{proof}

\

\begin{tma}[$v^{n+1}$-problem]\label{tma_UVestimates}
There exist a unique solution  $v^{n+1}\in V_h$ solving \eqref{eq:schemeC0-v}. 
Moreover, if the triangulation $\{\mathcal{T}_h\}$ is acute, that is, all angles of the simplices are less or equal than ${\pi}/2$,
then the discrete maximum principle holds, that is, 
\beq\label{eq:maxprinu}
\mbox{if } 
v^n> 0
\quad\mbox{then}\quad
v^{n+1}> 0,
\quad\mbox{ and if } 
v^n\leq M
\quad\mbox{then}\quad
v^{n+1}\leq M\,.
\eeq
Finally, the following weak estimate holds:
\beq\label{eq_estimatevUV}
\Delta t \sum_{n=0}^{N-1}\int_\Omega |\nabla v^{n+1}|^2dx
\,\leq\,
\int_\Omega I_h\big((v^0)^2\big)dx\,.
\eeq
\end{tma}
\begin{proof}
\textbf{Existence and uniqueness.}
Since problem \eqref{eq:schemeC0-v} is linear, it suffices to prove that if $v\in V_h$ is a solution of the homogeneous problem
$$
\frac{1}{\Delta t} (v,\bar{v})_h
 +  (\nabla v,\nabla\bar{v}) 
 +\mu \Big( (u^{n+1})_+ v , \bar{v}\Big)_h
 =0\, ,
$$
then $v=0$. Indeed, testing by $\bar{v}=v \in V_h$ we obtain
$$
\frac1{\Delta t} \int_\Omega I_h\Big( v^2 \Big) d\x
+ \int_\Omega |\nabla v|^2 d\x
+ \mu \int_\Omega I_h\Big( (u^{n+1})_+ v^2\Big) d\x
=0\,.
$$
Thus previous relation implies that $v=0$.

\

 \textbf{Discrete Maximum Principle.}  Assume that $v^n> 0$. We can define the following problem: Taking $z^n:=\dis\min_{x\in\Omega} v^n > 0$, find $z^{n+1}\in \mathbb{R}$ such that 
$$
\frac{z^{n+1} - z^n}{\Delta t}
+\mu\|u^{n+1}_+\|_{\infty} z^{n+1}
=0\,,
$$
that is, 
$$
z^{n+1}=\frac{z^n}{1 + \mu\Delta t\|u^{n+1}_+\|_{\infty}}>0\,.
$$
Looking at $z^{n+1}$ as a constant function,   $\nabla z^{n+1}=0$, then
\beq\label{eq:pbauxz}
\left(\frac{z^{n+1} - z^n}{\Delta t},\bar v\right)_h
+ (\nabla z^{n+1},\nabla\bar v)
+\mu\big(\|u^{n+1}_+\|_{\infty} z^{n+1},\bar v\big)_h
=0\,,\quad \forall\, \bar v\in V_h .
\eeq

Let see that $z^{n+1}$ solving \eqref{eq:pbauxz} satisfy $z^{n+1}\leq v^{n+1}$. For this, we define $w^{n+1}=v^{n+1} - z^{n+1}\in V_h$ ($w^{n}=v^{n} - z^{n}\geq 0$). Subtracting \eqref{eq:schemeC0-v} with \eqref{eq:pbauxz} we obtain:
$$
\left(\frac{w^{n+1} - w^n}{\Delta t},\bar v\right)_h
+ (\nabla w^{n+1},\nabla\bar v)
+\mu\big(u^{n+1}_+w^{n+1} + (u^{n+1}_+ - \|u^{n+1}_+\|_{\infty}) z^{n+1},\bar v\big)_h
=0\,,\quad \forall \, \bar v\in V_h.
$$
Testing 
by $\bar v=I_h(w^{n+1}_-)\in V_h$ (with $w_-:=\min\{w,0\}$) and using the relation $I_h(v \,v_-)\,=\,I_h((v_-)^2)$:
$$
\ba{c}\dis
\frac1{\Delta t}\int_\Omega I_h\left( (w^{n+1}_-)^2\right) d\x
- \frac1{\Delta t}\int_\Omega I_h (w^{n} w^{n+1}_-) d\x
+ (\nabla I_h (w^{n+1}_+)+ \nabla I_h (w^{n+1}_-),\nabla I_h (w_-^{n+1}) )
\\ \hueco\dis
+ \mu \int_\Omega I_h\Big(z^{n+1} (u_+^{n+1} - \|u^{n+1}_+\|_{\infty}) w^{n+1}_-\Big) d\x
+\mu \int_\Omega I_h\Big(u^{n+1}_+ (w_-^{n+1})^2\Big) d\x
\leq0\,.
\ea
$$
\\Using that we are considering an acute triangulation, that is, the interior angles of the triangles or tetrahedra are less or
equal than ${\pi}/2$, we can deduce \cite{CiarletRaviart}:
$$
(\nabla I_h (w^{n+1}_+),\nabla I_h (w_-^{n+1}) )\geq 0\,.
$$
Hence, due to the positivity of all the integrands, we can deduce that $I_h (w^{n+1}_-)=0$ so $w^{n+1}\geq 0$ and therefore $v^{n+1}=w^{n+1} + z^{n+1}\geq z^{n+1} >0$.

\


On the other hand, we can rewrite \eqref{eq:schemeC0-v} as
\beq\label{eq:schemeC0-vB}
\frac1{\Delta t}\dis( v^{n+1} - M ,\bar{v})_h
 +  \Big(\nabla(v^{n+1} - M),\nabla\bar{v}\Big) 
 + \mu \Big( (u^{n+1})_+ (v^{n+1} - M+M), \bar{v}\Big)_h
 =  
\frac1{\Delta t}\dis( v^{n} - M,\bar{v})_h
\eeq
for all $\bar{v}\in V_h$. 
Then, testing \eqref{eq:schemeC0-vB} by $\bar{v}=I_h((v^{n+1} - M)_+)$ (and using that $v^n\leq M$) we obtain
$$
\ba{c}
\dis\frac1{\Delta t}\int_\Omega I_h\Big[\big((v^{n+1} - M)_+\big)^2\Big]d\x
+\dis\int_\Omega \Big[\nabla\big(I_h(v^{n+1} - M)_+\big)\Big]^2d\x
\\ \hueco\dis
+\dis\int_\Omega \nabla\big(I_h(v^{n+1} - M)_-\big)\cdot\nabla\big(I_h(v^{n+1} - M)_+\big)d\x
\\ \hueco\dis
+ \mu \int_\Omega I_h\Big((u^{n+1})_+\big( [(v^{n+1} - M)_+]^2 + M(v^{n+1} - M)_+\big)\Big) d\x
\\ \hueco\dis
 \le \frac1{\Delta t} \int_\Omega I_h((v^{n} - M)(v^{n+1} - M)_+) d\x
 \leq 0\,.
\ea
$$
Using that the interior angles of the triangles or tetrahedra are less or
equal than $\pi/2$ we can deduce 
$$
\dis\int_\Omega \nabla\big(I_h(v^{n+1} - M)_-\big)\cdot\nabla\big(I_h(v^{n+1} - M)_+\big)d\x
\geq 0\,.
$$
that implies $I_h\big[\big((v^{n+1} - M)_+\big)^2\big]=0$ and therefore $v^{n+1} \leq M$.

\

\textbf{Weak Estimate.} Testing \eqref{eq:schemeC0-v} by $v^{n+1}$ we obtain
$$\ba{rcl}
\dis\frac1{2\Delta t}\int_\Omega I_h\big((v^{n+1})^2\big)d\x 
- \frac1{2\Delta t}\int_\Omega I_h\big((v^{n})^2\big)d\x 
+ \frac1{2\Delta t}\int_\Omega I_h\big((v^{n+1} - v^n)^2\big)d\x &&
\\ \hueco
\dis +\int_\Omega |\nabla v^{n+1}|^2 d\x
+ \int_\Omega I_h\big((u^{n+1})_+ (v^{n+1})^2\big)d\x &=&0\,.
\ea
$$
Adding the previous relation for $n$ from $0$ to $N-1$, and multiplying by $2\Delta t$
$$
\int_\Omega I_h\big(|v^{N}|^2\big)d\x 
\dis + 2 \, \Delta t\sum_{n=0}^{N-1} \int_\Omega |\nabla v^{n+1}|^2 d\x
\le \int_\Omega I_h\big(|v^{0}|^2\big)d\x\,,
$$
and estimate \eqref{eq_estimatevUV} is derived.  

\end{proof}

\subsubsection{Scheme UV-ND (Nonlinear Diffusion)}
The main idea of this scheme is to rewrite the diffusion term in a way that we can use it for obtaining a discrete version of inequality \eqref{eq:regresult}. In order to do that, for any $\varepsilon>0$ small enough, we introduce a new functional $G_\varepsilon(u)$ such that is a $C^2$-approximation of $G(u) $.  
This can be achieved by defining
\begin{equation}
G_\varepsilon''(u)\,:=\, \left\{
\ba{ccrcl}
\frac1{\varepsilon^2} &\mbox{ if }& &u&<\varepsilon\,,
\\ \hueco
\frac1{u^2} &\mbox{ if }&  \dis\varepsilon \leq &u& 
\leq \frac1\varepsilon\,,
\\ \hueco
\varepsilon^2 &\mbox{ if }&  \frac1\varepsilon \leq &u& \,,
\ea
\right.
\end{equation}
with $G_\varepsilon'(u)$, $G_\varepsilon(u)$ being the corresponding integral functions such that $G_\varepsilon'(u)=-1/u$ and $G_\varepsilon(u)=-\log(u) +\frac1\varepsilon$ when $\varepsilon \leq u \leq 1/\varepsilon$,
assuring that $G_\varepsilon(u)\geq0$ for all $u\in \mathbb{R}$. 


\

Then the scheme \textbf{UV-ND} reads:
\begin{itemize}
\item~[\textbf{Step 1}] 
Find $u^{n+1}\in U_h=\mathbb{P}_1$ solving the nonlinear problem:
\begin{equation}\label{eq:schemeC0-uA}
\dis\left(\delta_t u^{n+1} ,\bar{u}\right)_h
 +  \Big((u^{n+1})^2\nabla(I_hG_\varepsilon'(u^{n+1})),\nabla\bar{u}\Big) 
 - \chi \Big( u^{n+1} \nabla v^n , \nabla\bar{u}\Big)
 =  0
 \quad\pato \bar{u} \in U_h\,.
\end{equation}
Note that \eqref{eq:schemeC0-uA} is nonlinear and conservative (taking $\bar{u}=1$ one has $
\int_\Omega u^{n+1}d\x =\int_\Omega u^{n}d\x$).

\item~[\textbf{Step 2}]
Find $v^{n+1}\in V_h=\mathbb{P}_1$ solving the same linear problem  \eqref{eq:schemeC0-v} presented in scheme~\textbf{UV}. Consequently, Theorem~\ref{tma_UVestimates} also holds for this scheme.
 \end{itemize} 
 Notice that parameter $\varepsilon>0$ has been introduced for the spatial approximation, and then it will be taken as $\varepsilon=\varepsilon(h)\to 0$ as $h\to 0$.   
\begin{tma}\label{th:stabilityND}
If $(u^{n+1}, v^{n+1})$ solves scheme \textbf{UV-ND}, then the following discrete inequality holds:
\beq\label{eq:sigschenelawG}
\dis \delta_t\left(\int_\Omega I_h G_\varepsilon(u^{n+1})d\x \right)
\, + \,\frac{1}2\int_\Omega (u^{n+1})^2\big|\nabla\big(I_h G_\varepsilon'(u^{n+1})\big)\big|^2d\x
\,\leq\,
\frac{\chi^2}2\int_\Omega |\nabla v^{n+1}|^2d\x\,.
\eeq
\end{tma}
\begin{proof}
By using Taylor expansion and the convexity  of $G_\varepsilon$, due to $G_\varepsilon''(u)\ge 0$ for all $u$, we have
\beq\label{eq:proofdiscstaG}
\dis\int_\Omega I_h\left(({u^{n+1} - u^n}) G_\varepsilon'(u^{n+1})\right)d\x
\geq \dis (G_\varepsilon(u^{n+1}),1)_h - (G_\varepsilon(u^{n}),1)_h\,,
\eeq
which is nonnegative because $G_\varepsilon'(u)$ is an increasing function ($G_\varepsilon'(u)\sim -1/u$).
Therefore, testing \eqref{eq:schemeC0-uA} by $I_h G'_\varepsilon(u^{n+1})$ we have
$$
\delta_t\left(\int_\Omega I_h G_\varepsilon(u^{n+1})d\x\right) 
\,+\,\int_\Omega (u^{n+1})^2\big|\nabla \big(I_h G_\varepsilon'(u^{n+1})\big)\big|^2d\x
\,\leq\,\chi
\int_\Omega u^{n+1}\nabla\big(I_h G_\varepsilon'(u^{n+1})\big)\cdot\nabla v^{n+1}\,d\x\,.
$$
Finally, by using H\"older inequality we arrive at expression \eqref{eq:sigschenelawG}.
\end{proof}

\begin{cor}\label{corolario}
If $(u^{n+1}, v^{n+1})$ solves scheme \textbf{UV-ND}, the following estimates hold
\beq\label{eq_boundG0}
\int_\Omega I_h G_\varepsilon(u^{n+1})d\x\leq C\,, \quad \forall\, n\,,
\eeq
\beq\label{eqsumG0}
\Delta t \sum_{n=0}^{N-1}\int_\Omega (u^{n+1})^2\big|\nabla\big(I_h G_\varepsilon'(u^{n+1})\big)\big|^2d\x
\,\leq\,C\,,
\eeq
where $C$ is a constant that bounds both $\int_\Omega I_h G_\varepsilon(u^{0})dx$ and $\int_\Omega I_h\big((v^0)^2\big)dx$. Moreover, the following estimates hold 
\beq\label{eq:positivity0}
\int_\Omega \big(I_h(u_-^n)\big)^2dx\leq C \,\varepsilon^2 
\quad\mbox{ and }\quad
\|u^n\|_{L^1} \leq m_0 + C\varepsilon 
\quad \pato n\geq 1\,.
\eeq
\end{cor}
\begin{proof}
Adding up \eqref{eq:sigschenelawG}, using estimate \eqref{eq_estimatevUV} and $G_\varepsilon(u)\geq 0$ for all $u\in \mathbb{R}^d$, we can derive the estimates \eqref{eq_boundG0} and \eqref{eqsumG0}. On the other hand, estimate \eqref{eq:positivity0} is deduced from the inequality 
$$
\frac1{\varepsilon^2}(u_-)^2 \,\leq\,G_\varepsilon(u)\,,
\quad \pato u\in \mathbb{R}\,,
$$
following the same arguments presented in \cite{FMD4}.
\end{proof}
\begin{obs}\label{RemarkPosUV-ND}
From estimate \eqref{eq:positivity0}, we can say that 
scheme \textbf{UV-ND} satisfies an approximate positivity property for $u^n$, because 
$I_h(u_-^n)\rightarrow0$ in $L^2(\Omega)$ as $\varepsilon\rightarrow 0$,  with $\mathcal{O}(\varepsilon)$ accuracy rate.
\end{obs}
\subsubsection{Scheme UV-NS (Nonlinear Sensitivity)}
In this section we present another scheme developed to satisfy a discrete version of inequality \eqref{eq:regresult}. The key point now is to rewrite the sensitivity term in a nonlinear way introducing $d$ new functionals $\Lambda^i_\varepsilon(u):U_h\rightarrow \mathbb{P}_0$ ($i=1,\dots,d$) such that they satisfy
\beq\label{eq_def_Lambda}
\Big(
\Lambda^i_\varepsilon(u)\partial_{x_i}\big(I_h G_\varepsilon'(u)\big)
\Big)^2
\,=\,
\partial_{x_i} u	\,
\partial_{x_i}\big(I_h G_\varepsilon'(u)\big)
\qquad
\pato
i=1,\dots,d\,,
\eeq
that is, $\Lambda^i_\varepsilon(u)$ are constant by elements functions such that \eqref{eq_def_Lambda} holds in each element of the triangulation. In fact, \eqref{eq_def_Lambda} holds in any dimension by imposing the constraint of considering a structured mesh and the choice of $U_h=\mathbb{P}_1$, as it has been done with related expressions in \cite{BB,FMD4,GTCH}.

\begin{obs}[How to construct $\Lambda^i_\varepsilon(u)$]
In the one-dimensional case, the domain $\Omega=[a,b]$ can be splitted into $N$ subintervals named $I_j$ with $I_j=[x_j, x_{j+1}]$ ($1\leq j\leq N$), being $x_j$ the nodes of the partition. Moreover, the discrete derivative with respect to $x$ can be defined as the vector of length $N$ with components:
$$
\delta_x u\big|_{I_j}=\frac{u_{j+1} - u_j}{|I_j|}\,,
$$
with $u_j\sim u(x_j)$ and $|I_j|$ denoting the length of the interval $I_j$. Then, in the one-dimensional case, 
$\Lambda^i_\varepsilon(u)$ can be constructed in the following way:
\beq \label{def-1D}
\Lambda^1_\varepsilon(u)\big|_{I_j}
\,=\,
\left\{
\ba{cl}
\dis\frac{\pm\sqrt{\delta_x(u)\big|_{I_j}\delta_x\big(I_h G_\varepsilon'(u)\big)\big|_{I_j} }}{\delta_x\big(I_h G_\varepsilon'(u)\big)\big|_{I_j}}
& \quad\mbox{ if } u_j\neq u_{j+1}\,,
\\ \hueco
u_j
& \quad\mbox{ if } u_j=u_{j+1}\,,
\ea
\right.
\eeq
and choosing the sign $\pm$ such that $sign(\Lambda^1_\varepsilon(u)\big|_{I_j})=sign((u_{j+1} + u_j)/2)$\,.
\\The definition \eqref{def-1D} can be extended to higher dimensional domains just by using the same construction for each functional $\Lambda^i_\varepsilon(u)$, where $I_j$ represents now the intervals in the corresponding $i$-direction.
\end{obs}

%
%

\

Now, we state the new scheme~\textbf{UV-NS} as:
\begin{itemize}
\item~[\textbf{Step 1}] 
Find $u^{n+1}\in U_h=\mathbb{P}_1$ solving the nonlinear problem:
\begin{equation}\label{eq:schemeC0-uB}
\dis\left(\delta_t u^{n+1} ,\bar{u}\right)_h
 +  \Big(\nabla u^{n+1},\nabla \bar{u}\Big) 
 - \chi \Big( \Lambda_\varepsilon(u^{n+1}) \nabla v^n , \nabla\bar{u}\Big)
 =  0
 \quad\pato \bar{u} \in U_h\,,
\end{equation}
with $\Lambda_\varepsilon(u^{n+1})=\mbox{diag}(\Lambda^i_\varepsilon(u^{n+1}))_{i=1,\dots,d}$. Note that \eqref{eq:schemeC0-uB} is nonlinear and conservative.

 \item~[\textbf{Step 2}]
Find $v^{n+1}\in V_h=\mathbb{P}_1$ solving the same linear problem  \eqref{eq:schemeC0-v} presented in scheme~\textbf{UV}. 
\end{itemize} 

\begin{tma}\label{th:stabilityNS}
If $(u^{n+1}, v^{n+1})$ solves scheme \textbf{UV-NS}, then the following discrete inequality holds:
\beq\label{eq:sigschenelawG2}
\dis \delta_t\left(\int_\Omega I_h G_\varepsilon(u^{n+1})d\x \right)
\,+\,  \frac12\int_\Omega \left|
\Lambda_\varepsilon(u^{n+1})\nabla\big(I_h G_\varepsilon'(u^{n+1})\big)
\right|^2 d\x
\,\leq\,
\frac{\chi^2}2\int_\Omega |\nabla v^{n+1}|^2d\x\,.
\eeq
\end{tma}
\begin{proof}
Testing \eqref{eq:schemeC0-uB} by $I_h G'_\varepsilon(u^{n+1})$ we have
$$
\ba{c}\dis
\delta_t\left(\int_\Omega I_h G_\varepsilon(u^{n+1})d\x\right) 
 \,+\, \int_\Omega \nabla u^{n+1}\cdot\nabla \big(I_h G'_\varepsilon(u^{n+1})\big) d\x
 \\ \hueco\dis
 \,\leq\,
\chi\int_\Omega \Lambda_{\varepsilon}(u^{n+1})\nabla\big(I_h G_\varepsilon'(u^{n+1})\big)\cdot \nabla v^{n+1} d\x
\\ \hueco\dis
\leq 
\frac12 \int_\Omega 
\left|
\Lambda_\varepsilon(u^{n+1})\nabla\big(I_h G_\varepsilon'(u^{n+1})\big)
\right|^2 d\x
\,+\, \frac{\chi^2}2\int_\Omega |\nabla v^{n+1}|^2d\x\,,
\ea
$$
and taking into account \eqref{eq_def_Lambda}  we can deduce expression \eqref{eq:sigschenelawG2}.
\end{proof}
\begin{obs}\label{RemarkPosUV-NS}
The estimates presented in Corollary~\ref{corolario} also holds in this scheme, substituting estimate \eqref{eqsumG0} by the corresponding one
\beq\label{eqsumG2}
\Delta t \sum_{n=0}^{N-1}
\int_\Omega
\left|
\Lambda_\varepsilon(u^{n+1})\,\nabla\big(I_h G_\varepsilon'(u^{n+1})\big)
\right|^2d\x
\,\leq\,C\,.
\eeq
In particular, estimate \eqref{eq:positivity0} holds, giving approximate positivity for scheme \textbf{UV-NS}  in the sense of Remark~\ref{RemarkPosUV-NS}.
\end{obs}
\subsection{Scheme UVS}

In this section we present a scheme that approximates the weak formulation obtained when we test \eqref{eq:systemreftrun2} by regular enough functions 
$(\bar{u},\bar{v},\bar{\bsigma}):\overline\Omega\mapsto \mathbb{R}^{1\times 1\times d}$
by using a Galerkin approximation and \textit{mass lumping} ideas.
%
In order to do that, we first introduce a function $a_\varepsilon$ to define a truncation function of $F(u)$ (and its derivatives), namely $F_\varepsilon$, by considering $\dis F_\varepsilon''(u):=1/{a_\varepsilon(u)} $ as an approximation of $1/u$, where $a_\varepsilon(u)$ is a $C^1$-truncation function of $a(u)=u$ defined as 

\beq\label{eq:def_a}
a_\varepsilon(u)\,:=\,\left\{\ba{cl}
\varepsilon & \quad\mbox{ if }\quad u\leq\varepsilon\,,
\\ \hueco
 u &  \quad\mbox{ if }\quad u\in\left(\varepsilon,
 \frac1\varepsilon\right)\,,
\\ \hueco
\frac1\varepsilon & \quad\mbox{ if }\quad  u\geq \frac1\varepsilon\,.
\ea\right.
\eeq


In fact, the corresponding integration constants that arise in computations from $F_\varepsilon''(u)$ to $F_\varepsilon(u)$  are fixed considering that $F_\varepsilon'(u)=\ln(u)$ and $F_\varepsilon(u)=u\ln(u)-u +1$ when $u\in (\varepsilon,+\infty)$.

The proposed numerical scheme \textbf{UVS} is: 

\begin{itemize}
 \item~[\textbf{Step 1}]  Find $(u^{n+1},\bsigma^{n+1})\in U_h \times \boldsymbol{\Sigma}_h=\mathbb{P}_1\times\mathbb{P}_1  $ with $\bsigma^{n+1}|_{\partial\Omega}=0$ and  solving the coupled and nonlinear system
\begin{equation}\label{eq:systemreftrun2scheme}
\left\{
\begin{array}{rcl}
\dis\left(\delta_t u^{n+1} , \bar{u}\right)_h
+ ( \nabla u^{n+1} , \nabla \bar{u})
- 2\chi \big(u^{n+1}\, \sqrt{v^{n}} \bsigma^{n+1}, \nabla\bar{u}\big) 
& = & 0\,,
\\ \hueco\dis 
\left(\delta_t \bsigma^{n+1} , \bar{\bsigma} \right)
- 2\left(\frac1{\sqrt{v^n}} (\nabla\bsigma^{n+1})^t\bsigma^{n+1}, \bar{\bsigma}\right) 
+\frac13\left(\frac{1}{v^n} |\bsigma^{n+1}|^2\bsigma^{n+1}, \bar{\bsigma}\right)
&&
\\ \hueco\dis 
+\frac23\left(\frac{1}{v^n}  \nabla (\sqrt{v^n}) |\bsigma^{n+1}|^2, \bar{\bsigma}\right)
+ (\nabla\cdot\bsigma^{n+1}, \nabla\cdot\bar{\bsigma})
+ (\mbox{rot}\, \bsigma,^{n+1} \mbox{rot}\, \bar{\bsigma}) 
&&
\\ \hueco\dis 
+ \frac\mu 2 (\bsigma^{n+1}\, u^{n+1}_+ , \bar{\bsigma})
+ \frac\mu 2 \Big(\sqrt{v^n}\, u^{n+1}\, \nabla(I_h F_\varepsilon'(u^{n+1})) , \bar{\bsigma}\Big)
& = & 0\,,
\end{array}
\right.
\end{equation} 
for all $(\bar{u},\bar{\bsigma})\in U_h \times \boldsymbol{\Sigma}_h$ with $\bar{\bsigma}|_{\partial\Omega}=0$. 
In the last term of the $\bsigma$-system, $\nabla u$ has been approximated by $u\nabla I_h F_\varepsilon(u)$.
Note that \eqref{eq:systemreftrun2scheme} is a nonlinear and conservative problem.  


 \item~[\textbf{Step 2}]
Find $v^{n+1}\in V_h=\mathbb{P}_1\subset H^1(\Omega)$ solving the same linear problem  \eqref{eq:schemeC0-v} presented in scheme \textbf{UV}. 
\end{itemize}

Now we present a result that provides a discrete version of the estimate \eqref{eq:ELsigma} derived in Theorem~\ref{th:stability}.
\begin{tma}\label{th:discstability}
If $(u^{n+1},\bsigma^{n+1})$ solves \textbf{Step 1} \eqref{eq:systemreftrun2scheme}, then the following discrete version of \eqref{eq:ELsigma}  holds:
\beq\label{eq:sigschenelaw}
\dis \delta_t{E}_h^\varepsilon(u^{n+1},\bsigma^{n+1})
\,+ \mu\,{D}_{1,h}^\varepsilon(u^{n+1}) 
\,+ \chi\,{D}_2(v^n,\bsigma^{n+1}) 
\,+\chi\mu \,{D}_3(u^{n+1},\bsigma^{n+1})
\,\leq\, R(v^n,\bsigma^{n+1})\,,
\eeq
where  ${D}_2(v,\bsigma)$, ${D}_3(u,\bsigma)$  and $R(v,\bsigma)$ are defined in \eqref{eq:Energysigma} and
$$
{E}_h^\varepsilon(u,\bsigma)\,:=\,\dis \frac{\mu}{4}\int_\Omega I_hF_\varepsilon(u)d\x
\,+\, 
\frac{\chi}{2}\int_\Omega |\bsigma|^2d\x
\quad
\mbox{ and }
\quad
{D}_{1,h}^\varepsilon(u)
\,:=\,
\frac14\int_\Omega \nabla u\cdot\nabla\big(I_h F_\epsilon'(u)\big)\,d\x
\,.
$$

\end{tma}
\begin{proof}
Testing \eqref{eq:systemreftrun2scheme}$_1$ by $\bar u=\dis\frac{\mu}{4}I_h F'_\varepsilon(u^{n+1})$ and using Taylor expansion and convexity of $F_\varepsilon$ (as we have done for $G_\varepsilon$ in \eqref{eq:proofdiscstaG}) we have
\beq\label{estsigsch1}
\delta_t\left(\frac{\mu}{4}\int_\Omega I_h F_\varepsilon(u^{n+1})d\x \right) 
+ \mu {D}_{1,h}^\varepsilon(u^{n+1})
- \frac{\chi\mu}{2} \int_\Omega u^{n+1} \, \sqrt{v^{n}} \bsigma^{n+1}\cdot \nabla(I_h F'_\varepsilon(u^{n+1}))\,d\x
\,\leq\,  0\,.
\eeq
Testing \eqref{eq:systemreftrun2scheme}$_2$ by $\bar{\bsigma}=\chi\bsigma^{n+1}$ we have
\beq\label{estsigsch2}
\ba{l}
\dis\delta_t\left(\frac\chi2\int_\Omega (\bsigma^{n+1})^2 dx\right) 
-\dis2\chi\int_\Omega\frac1{\sqrt{v^n}} (\bsigma^{n+1})^t(\nabla (\bsigma^{n+1}))^t\bsigma^{n+1} \,d\x
\\ \hueco\dis
+\frac{2\chi}3\int_\Omega\frac{1 }{v^n} |\bsigma^{n+1}|^2\nabla (\sqrt{v^n})\cdot \bsigma^{n+1} d\x 
+\dis\frac{\chi}3\int_\Omega\frac{1}{v^n} |\bsigma^{n+1}|^4d\x
\\ \hueco\dis
+ \chi\int_\Omega \Big((\nabla\cdot\bsigma^{n+1})^2 + |\mbox{rot}\,\bsigma^{n+1}|^2\Big)d\x
\\ \hueco\dis 
+\dis 
\frac{\chi\mu}2 \int_\Omega |\bsigma^{n+1}|^2\,u^{n+1}_+d\x
+ \frac{\chi\mu}2 \int_\Omega \sqrt{v^n}\, u^{n+1}\, \nabla(I_h F_\varepsilon'(u^{n+1}))\cdot \bsigma^{n+1}d\x =  0\,.
\ea
\eeq
Using integration by parts:
$$
\int_\Omega \frac{1 }{v^n} |\bsigma^{n+1}|^2 \nabla (\sqrt{v^n})\cdot\bsigma^{n+1} d\x
=
-\int_\Omega |\bsigma^{n+1}|^2 \nabla\left(\frac{1 }{\sqrt{v^n}}\right) \cdot\bsigma^{n+1} d\x
=
\int_\Omega \frac{1 }{\sqrt{v^n}} \nabla\cdot (|\bsigma^{n+1}|^2\bsigma^{n+1}) d\x,
$$

then the desired relation \eqref{eq:sigschenelaw} holds by adding equations \eqref{estsigsch1} and \eqref{estsigsch2} .

\end{proof}

\begin{cor}\label{cor:energy}
In the particular case of considering one-dimensional domains ($1D$), relation \eqref{eq:sigschenelaw} implies: 
\beq\label{eq:sigschenelaw1d}
\dis \delta_t{E}_h^\varepsilon(u^{n+1},\sigma^{n+1})
\,+\mu \,{D}_{1,h}^\varepsilon(u^{n+1}) 
\,+\chi \,{D}_2(v^n,\sigma^{n+1}) 
\,+\mu\chi \,{D}_3(u^{n+1},\sigma^{n+1}) 
\,\leq\, 0\,,
\eeq
with ${D}_{1,h}^\varepsilon(u^{n+1}), \,{D}_2(v^n,\sigma^{n+1}), \,{D}_3(u^{n+1},\sigma^{n+1})\geq0$. In particular, scheme {\bf UVS} is unconditional energy-stable with respect to the energy ${E}_h^\varepsilon(u,\sigma)$ defined in \eqref{eq:Energysigma}$_1$.
\end{cor}
\begin{proof}
From Theorem~\ref{th:discstability}, we only need to prove that ${D}_{1,h}^\varepsilon(u^{n+1})\geq 0$ and $R(v^n,\sigma^{n+1})=0$. Since $u^{n+1}\in U_h=\mathbb{P}_1$  and we are working in $1D$, we can write:
$$
{D}_{1,h}^\varepsilon(u^{n+1})
\,=\,
\frac14\int_\Omega (u_x^{n+1}) (I_h\,F_\varepsilon'(u^{n+1}))_x\,dx
\,=\,
\frac14 \sum_{j=1}^J\left(\frac{u^{n+1}_{j+1} - u^{n+1}_j}{h}\right)
\left(\frac{F_\varepsilon'(u_{j+1}^{n+1}) - F_\varepsilon'(u_{j}^{n+1})}{h}\right)\,,
$$
which is nonnegative because $F_\varepsilon'(u)$ is an increasing function ($F_\varepsilon'(u)\sim \ln u$).
Finally, in one-dimensional domains variable $\sigma$ is a scalar quantity, so the term $R(v^n,\sigma^{n+1})$ reads:
$$
R(v^n,\sigma^{n+1})
\,=\,
-\dis \frac{2\chi}{3}\int_\Omega \frac{1 }{\sqrt{v^n}} \Big(\partial_x((\sigma^{n+1})^3) - 3(\partial_x \sigma^{n+1})(\sigma^{n+1})^2\Big) dx
\,=\,
0\,.
$$
\end{proof}

%

\begin{cor}
In the particular case of considering one-dimensional domains ($1D$), the following estimates hold
\beq\label{eq_boundG}
\int_\Omega I_h F_\varepsilon(u^{n+1})dx
\,\leq\,
 C \quad \forall\, n\geq 1\,,
\eeq
\beq\label{eqsumG}
\Delta t \sum_{n=0}^{N-1}\Big(\mu \,{D}_{1,h}^\varepsilon(u^{n+1}) 
\,+\chi \,{D}_2^\varepsilon(v^n,\sigma^{n+1}) 
\,+\mu\chi \,{D}_3(u^{n+1},\sigma^{n+1})\Big)\,\leq\,C\,,
\eeq
where $C>0$ is a constant that bounds $\int_\Omega I_h F_\varepsilon(u^{0})dx$. Moreover, the following estimates also hold 
\beq\label{eq:positivity}
\int_\Omega \big(I_h(u_-^n)\big)^2dx
\,\leq\,
 C \,\varepsilon
\quad\mbox{ and }\quad
\|u^n\|_{L^1} \leq m_0 + C\sqrt{\varepsilon }
\quad \pato n\geq 1\,.
\eeq
\end{cor}
\begin{proof}
The proof follows the same arguments presented in Corollary~\ref{corolario} considering now  the estimate
$$
\frac1{\varepsilon}(u_-)^2 
\,\leq\, F_\varepsilon(u), 
\quad \forall\, u\in \mathbb{R}.
$$
\end{proof}
 \begin{obs}\label{RemarkPosUS}
From estimate \eqref{eq:positivity}, we can say that 
the one-dimensional version of scheme \textbf{UVS} satisfies an approximate positivity property for $u^n$, because 
$I_h(u_-^n)\rightarrow0$ in $L^2(\Omega)$ as $\varepsilon\rightarrow 0$,  with $\mathcal{O}(\sqrt\varepsilon)$ order.
\end{obs}

\section{Numerical simulations in $1D$ domains}\label{sec:simulations}

The aim of this section is to report the numerical results obtained carrying out simulations using  the schemes presented through the paper in one-dimensional domains. The idea is to illustrate the type of dynamics exhibit by chemo-attraction and consumption models and to compare the effectiveness of each of the presented schemes. All the simulations have been performed using MATLAB software \cite{matlab}. 

\

We consider a regular partition of the spatial domain $\overline\Omega=[a,b]$ denoted by $ T_h:=\bigcup_{i=1}^{J-1} I_i$, where $J$ denotes the number of nodes, $\{x_j\}_{j=1}^J$ the coordinates of these nodes and $h$ the size of the mesh, that for simplicity we assume that is constant in the whole domain. We will compare the schemes presented in Section~\ref{sec:schemes} together with a conservative and positive Finite Volume scheme that can be viewed as a Finite Element scheme with artificial diffusion, and it has been derived following the ideas presented in \cite{ZhouSaito}.

\

The physical and discrete parameters for each example will be detailed in each subsection except the value of the truncation parameter and the tolerance used for the iterative methods for the nonlinear schemes, that are going to be always chosen as:
\beq\label{eq:paramepstol}
\varepsilon\,=\,h^2
\quad
\mbox{ and }
\quad
C_{tol}\,=\,10^{-8} \,.
\eeq

The section is organized as follows: Firstly we introduce the iterative algorithms to approximate the nonlinear problems and the Finite Element scheme with artificial diffusion equivalent to a conservative and positive Finite Volume scheme. Then we present one example with the complete dynamics, that is, until it reaches the equilibrium configuration. The purpose of this example is to show that the system tends to a flat/constant equilibrium configuration of $u$ (with $u_{eq}=m_0=\frac1{|\Omega|}\int_\Omega u^0 dx$) and $v=0$, while the energy $E(u,v)$ decreases and the volume of $u$ remains constant. After that, we compare the ability of each scheme to maintain the positivity of the $u$ unknown using a choice of the initial condition and physical parameters designed in such a way that in some regions $u$ tends to be very close to zero while in other parts the value of $u$ is far away form zero. In the third part we focus on studying how each scheme capture the evolution of the energy in time using different values of the spatial and time discretization parameters. Finally, we perform a numerical study of the experimental order of convergence in space for each scheme.

\subsection{Iterative Methods}
Some of the presented schemes are nonlinear. In the following we detail the iterative algorithms that we have considered for approximating each of the nonlinear systems.
\subsubsection{Scheme UV-ND (Nonlinear Diffusion)}
Iterative algorithm to approximate the nonlinear problem \eqref{eq:schemeC0-uA}:
Find $u_{\ell+1}\in U_h$ such that
\beq\label{eq:iterative1}
\ba{rcl}
\dis\frac1{\Delta t}(u_{\ell+1} , \bar{u})_h 
+ ( \nabla u_{\ell+1} , \nabla\bar{u})
&=&\dis
\frac1{\Delta t}(u^{n} , \bar{u})_h 
+ ( \nabla(u_{\ell}) , \nabla\bar{u})
+ \chi \Big( u_{\ell} \nabla v^n , \nabla\bar{u}\Big)
\\ \hueco\dis
&&-  \Big((u_{\ell})^2\nabla(I_hG_\varepsilon'(u^{\ell})),\nabla\bar{u}\Big) 
\ea
\eeq
Stopping criterium: Iterates until 
$
\frac{\|u_{\ell+1} - u_\ell\|_{H^1(\Omega)}}{\|u_{\ell}\|_{H^1(\Omega)}} \,\leq\, C_{tol}\,,
$
with $C_{tol}$  a given tolerance.
%

\subsubsection{Scheme UV-NS (Nonlinear Sensitivity)}
Iterative algorithm to approximate the nonlinear problem \eqref{eq:schemeC0-uB}:
Find $u_{\ell+1}\in U_h$ such that
\beq\label{eq:iterative2}
\dis\frac1{\Delta t}(u_{\ell+1} , \bar{u})_h 
+ ( \nabla u_{\ell+1} , \nabla\bar{u})
=\dis
\frac1{\Delta t}(u^{n} , \bar{u})_h 
+ \chi \Big( \Lambda_\varepsilon(u_{\ell}) \nabla v^n , \nabla \bar{u}\Big)
\eeq
Stopping criterium: Iterate until
$
\frac{\|u_{\ell+1} - u_\ell\|_{H^1(\Omega)}}{\|u_{\ell}\|_{H^1(\Omega)}} \,\leq\, C_{tol}\,
$.
%

\subsubsection{Scheme UVS}

We consider the following iterative algorithm to approximate this  nonlinear problem \eqref{eq:systemreftrun2scheme}:
\begin{itemize}
\item[\textbf{Substep.1}] Find $u_{\ell+1}\in U_h$ such that
\beq\label{eq:iterative1}
\dis\frac1{\Delta t}(u_{\ell+1} , \bar{u})_h 
+ ( \nabla(u_{\ell+1}) , \nabla\bar{u})
\,=\,
\frac1{\Delta t}(u^{n} , \bar{u})_h 
+ 
 2\chi \big(u_\ell\, \sqrt{v^{n}} \bsigma_{\ell}, \nabla\bar{u}\big)\,,
 \quad \pato \bar{u}\in U_h.
\eeq
\item[\textbf{Substep.2}] Find $\bsigma_{\ell+1}\in \boldsymbol{\Sigma}_h$ such that, for all $ \bar{\bsigma}\in \boldsymbol{\Sigma}_h$,

\begin{equation}\label{eq:iterative2}
\begin{array}{rcl}
\dis\frac1{\Delta t}(\bsigma_{\ell+1} , \bar{\bsigma} )
&+& (\nabla\cdot\bsigma^{\ell+1}, \nabla\cdot\bar{\bsigma})
+ (\mbox{rot}\, \bsigma,^{\ell+1} \mbox{rot}\, \bar{\bsigma}) 
+ \frac\mu 2 (\bsigma_{\ell+1}\, (u_{\ell + 1})_+ , \bar{\bsigma})
\\ \hueco\dis
&= &
\dis\frac1{\Delta t}(\bsigma^{n} , \bar{\bsigma} )
- \frac13\left(\frac{1}{v^n} |\bsigma^{\ell}|^2\bsigma^{\ell}, \bar{\bsigma}\right)
\dis+ 2\left(\frac1{\sqrt{v^n}} (\nabla\bsigma^{\ell})^t\bsigma^{\ell}, \bar{\bsigma}\right) 
\\ \hueco \dis
&-&
\dis\frac23\left(\frac{1}{v^n}  \nabla (\sqrt{v^n}) |\bsigma^{\ell}|^2, \bar{\bsigma}\right)
- \dis \frac\mu 2 \Big(\sqrt(v^n)\, u_{\ell}\, \nabla (I_h F_\varepsilon'(u_{\ell})) , \bar{\bsigma}\Big)\,.
\end{array}
\end{equation} 
\item[\textbf{Substep 3:}] Stopping criterium. Iterate until
$
\frac{\|(u_{\ell+1},\bsigma_{\ell+1}) - (u_\ell,\bsigma_\ell)\|_{H^1(\Omega)\times \textbf{H}^1(\Omega)}}{\|(u_{\ell},\bsigma_{\ell})\|_{H^1(\Omega)\times \textbf{H}^1(\Omega)}} \,\leq\, C_{tol}\,.
$
\end{itemize}

\subsection{Scheme UV-AD (Artificial Diffusion) in $1D$}

We introduce the upwind Finite Volume scheme presented in \cite{ZhouSaito} that can be viewed as a Finite Element scheme with artificial diffusion. For this, we define the interior control volume as $K_j=[x_{j-\frac12},x_{j+\frac12}]$  for $j=2,\dots,J-1$, while the boundary control volumes are defined as $K_1=[x_1,x_{1+\frac12}]$ and $K_J=[x_{J-\frac12},x_J]$, with $x_{j+\frac12}=(x_j + x_{j+1})/2$. Moreover we consider the notation for discrete spatial derivatives $\delta_x u_{j+\frac12}:= (u_{j+1} - u_j)/h$ for $j=1,\dots,J-1$. 
Using this notation the boundary conditions are discretized as
$$
\delta_x u_{1-\frac12}\,=\,\delta_x u_{J+\frac12}\,=\,0
\quad\quad\mbox{ and }\quad\quad
\delta_x v_{1-\frac12}\,=\,\delta_x v_{J+\frac12}\,=\,0\,.
$$

The proposed upwind finite volume scheme is the following:
\begin{itemize}
\item~[\textbf{Step 1}] 
For all $j=1,\dots,J$, find $u_j^{n+1}$ solving the linear problem:
\begin{equation}\label{eq:schemeFV-u}
\ba{c}
\displaystyle
|K_j|\delta_t u^{n+1}_j 
\,-\, \Big(\delta_x u^{n+1}_{j+\frac12} - \delta_x u^{n+1}_{j-\frac12} \Big)
\\ \hueco
\,+\, \chi \Big(
(\delta_x v^{n}_{j+\frac12})_+ u_j^{n+1}
\, + \,(\delta_x v^{n}_{j+\frac12})_- u_{j+1}^{n+1}
\,- \,(\delta_x v^{n}_{j-\frac12})_+ u_{j-1}^{n+1}
\,- \,(\delta_x v^{n}_{j-\frac12})_- u_j^{n+1}
\Big)
=0.
\ea
\end{equation}

\item~[\textbf{Step 2}]
For all $j=1,\dots,J$, find $v_j^{n+1}$ solving the linear problem:
\begin{equation}\label{eq:schemeFV-v}
|K_j|\delta_t v^{n+1}_j 
\,-\, \Big(\delta_x v^{n+1}_{j+\frac12} - \delta_x v^{n+1}_{j-\frac12} \Big)
 +\mu\, u_j^{n+1} v_j^{n+1}
 =  0\,.
 \end{equation}
 \end{itemize}

In fact, it is easy to check that scheme~\eqref{eq:schemeFV-u}-\eqref{eq:schemeFV-v} is equivalent to the following linear Finite Element scheme with artificial diffusion, denoted by \textbf{UV-AD}:
\begin{itemize}
\item~[\textbf{Step 1}] 
Find $u^{n+1}\in U_h=\mathbb{P}_1\subset H^1(\Omega)$ solving the linear problem:
\begin{equation}\label{eq:schemeC0-u-b}
\dis\left(\delta_t u^{n+1} ,\bar{u}\right)_h
 +  \Big(u^{n+1}_x,\bar{u}_x\Big) 
 +  h \frac{\chi}2 \Big(|v_x^n|u^{n+1}_x,\bar{u}_x\Big) 
 - \chi \Big( u^{n+1} (v^n)_x , \bar{u}_x\Big)
 =  0
 \quad\pato \bar{u} \in U_h\,.
\end{equation}
Note that \eqref{eq:schemeC0-u-b} is conservative, because taking $\bar{u}=1$ one has 
$
\int_\Omega u^{n+1}dx
=\int_\Omega u^{n}dx
\,.
$
\item~[\textbf{Step 2}] Find $v^{n+1}\in V_h=\mathbb{P}_1\subset H^1(\Omega)$ solving the same linear problem  \eqref{eq:schemeC0-v} presented in scheme \textbf{UV}. 
\end{itemize}

\subsection{Example I: Dynamic towards constants}
In this case we consider as initial configuration two smooth functions with the same amplitude and the physical parameters shown below:
\beq\label{eq:A_initialdynconst}
\left\{\ba{rcl}
u_0&=&1.0001 + \cos(5\pi x)\,,
\\ \hueco
v_0&=&1.0001 + \cos(2\pi x)\,,
\ea\right.
\quad
[0,T]\,=\,[0,0.3]\,,
\quad
\chi\,=\,100
\quad
\mbox{ and }
\quad \mu=1000\,.
\eeq
In Figure~\ref{fig:Dynamic_UV_A} we present the dynamics of the system using scheme \textbf{UV} with discretization parameters $h=10^{-3}$ and $\Delta t=10^{-6}$, where we can observe that initially the system tends to accumulate the cell population density close to the boundaries (due to the attraction of the cell population towards the chemical substance), then the consumption effects dominates the dynamics producing that the amount of $v$ decreases and finally the system start to minimize the gradients of $u$ and the system 
reaches the expected equilibrium configuration, that is, a flat configuration of $u$ and $v=0$. Moreover, we can see in Figure \ref{fig:A_energy} that the energy is decreasing until the system reaches the equilibrium configuration and the amount of each unknown is presented in Figure \ref{fig:A_volume}, showing that the amount of chemical substance decreases to zero (due to the consumption effects) and the volume cell population density remains constant in time.

%


\begin{figure}
\begin{center}
\includegraphics[width=0.32\textwidth]{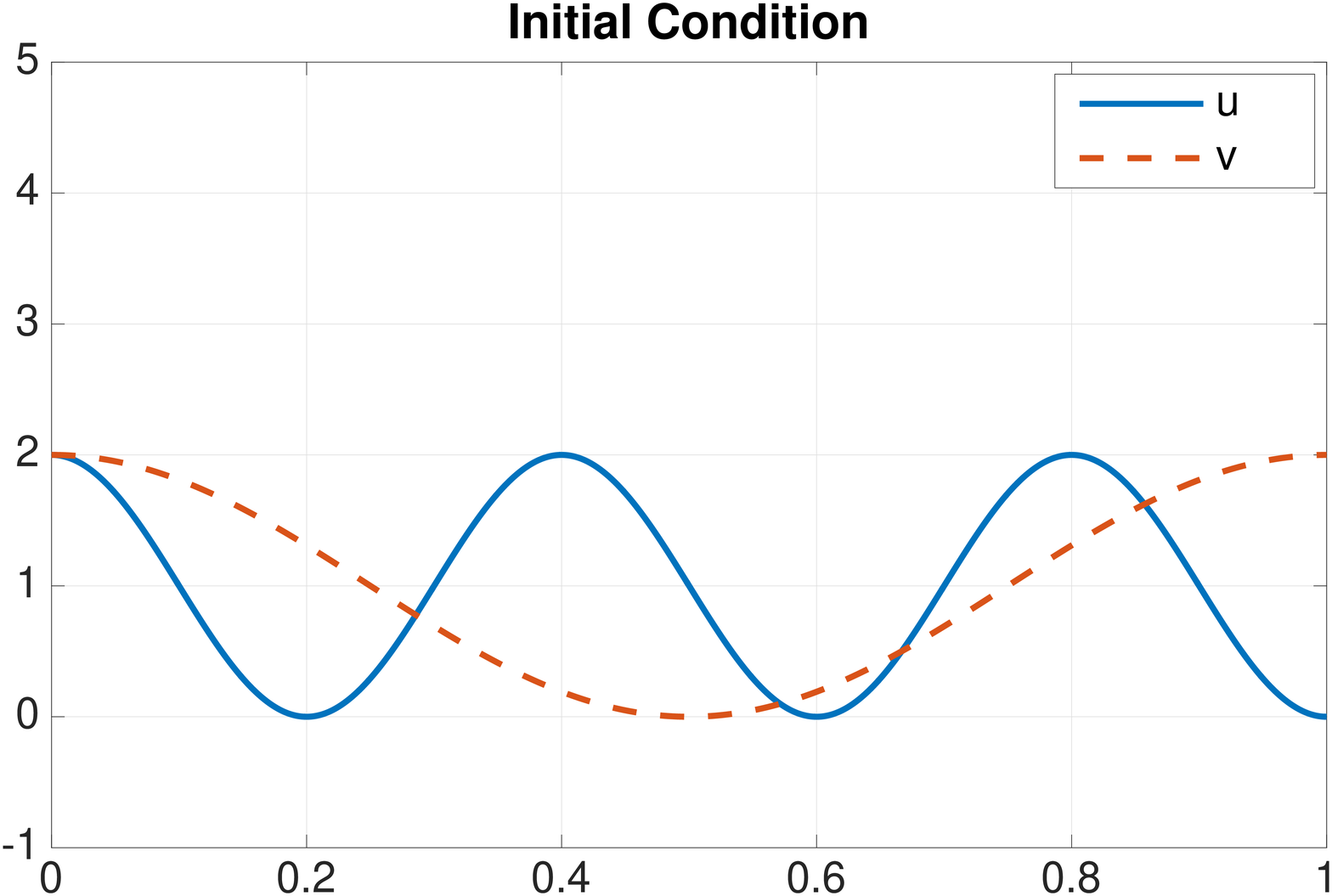}
\includegraphics[width=0.32\textwidth]{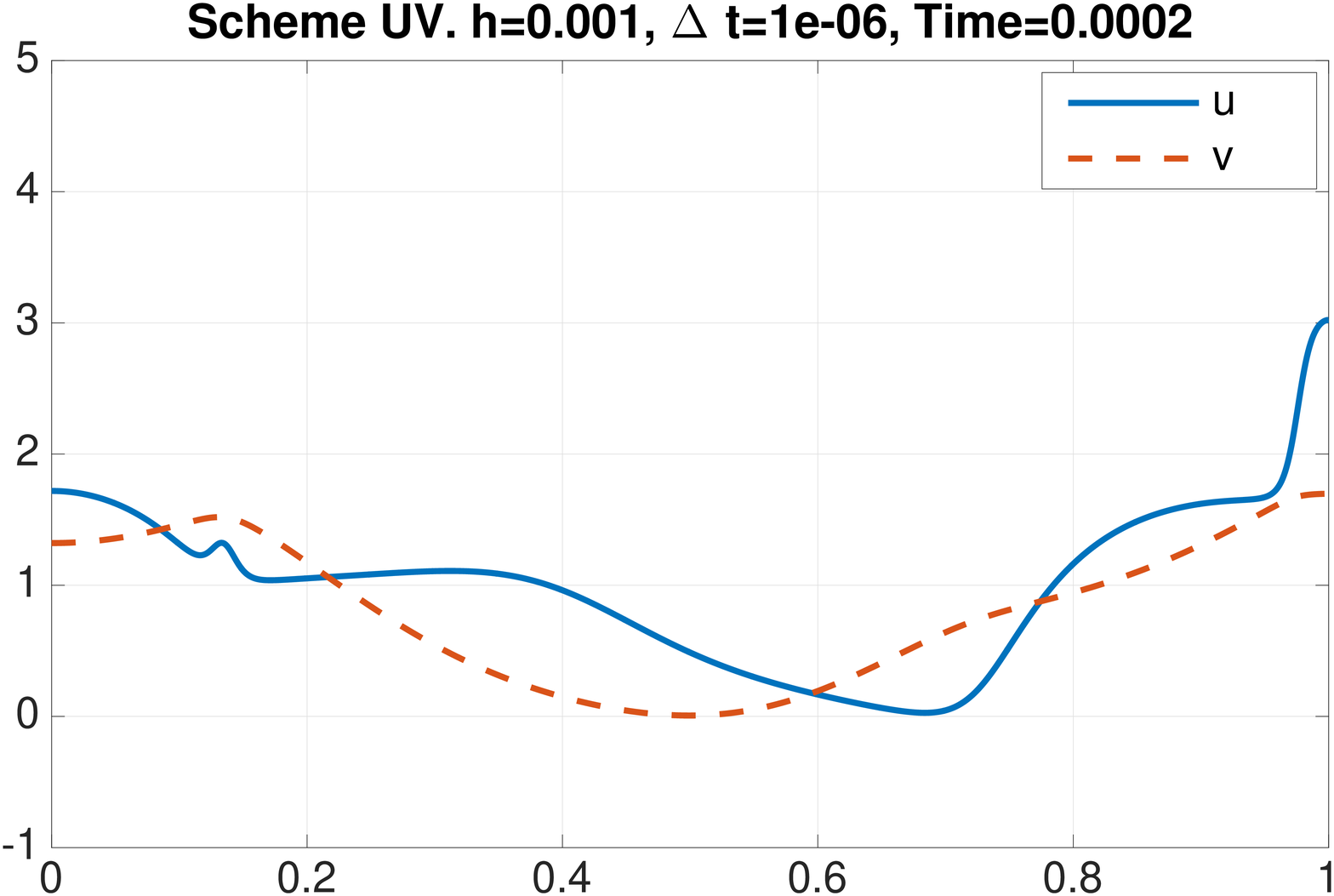}
\includegraphics[width=0.32\textwidth]{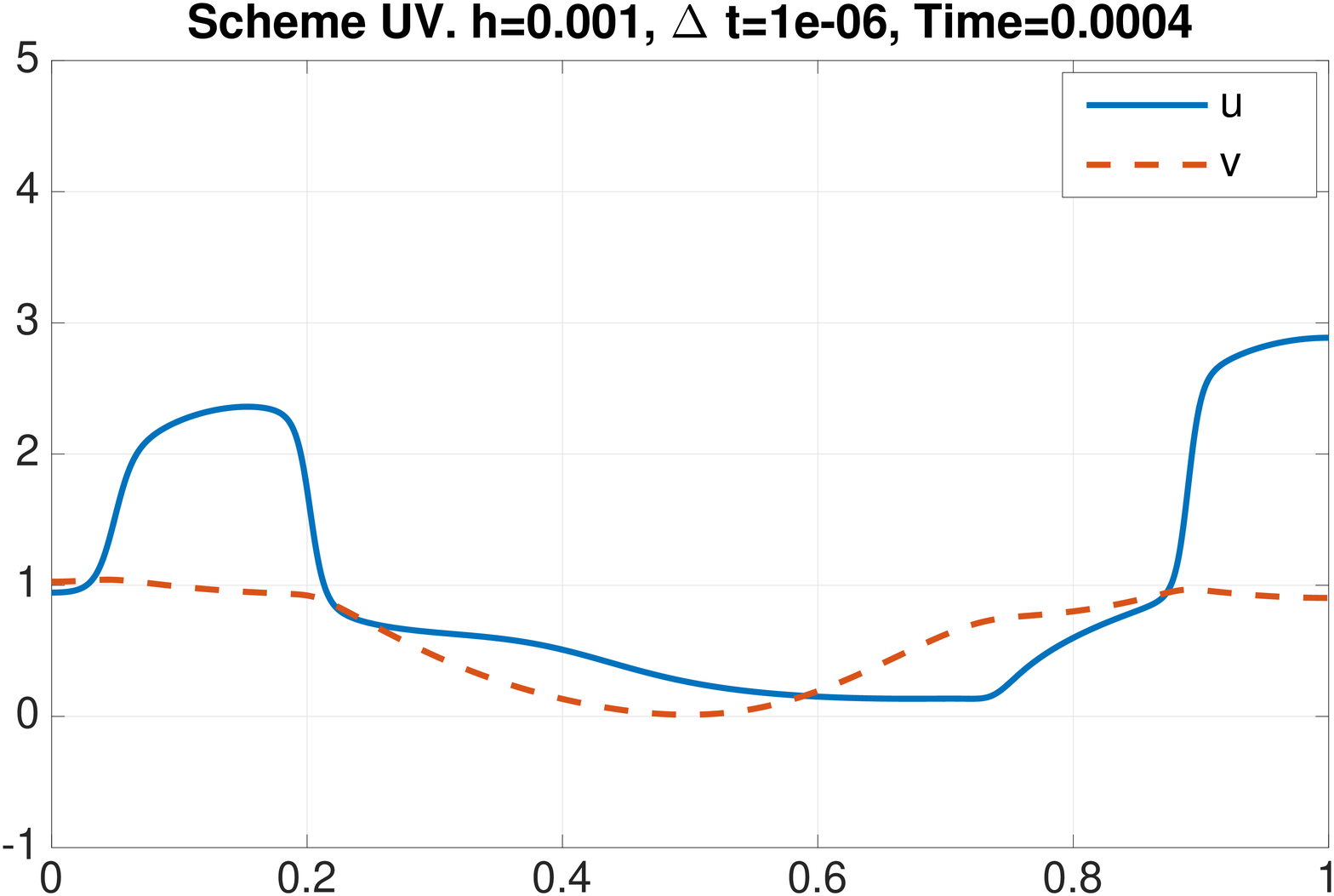}
\\
\includegraphics[width=0.32\textwidth]{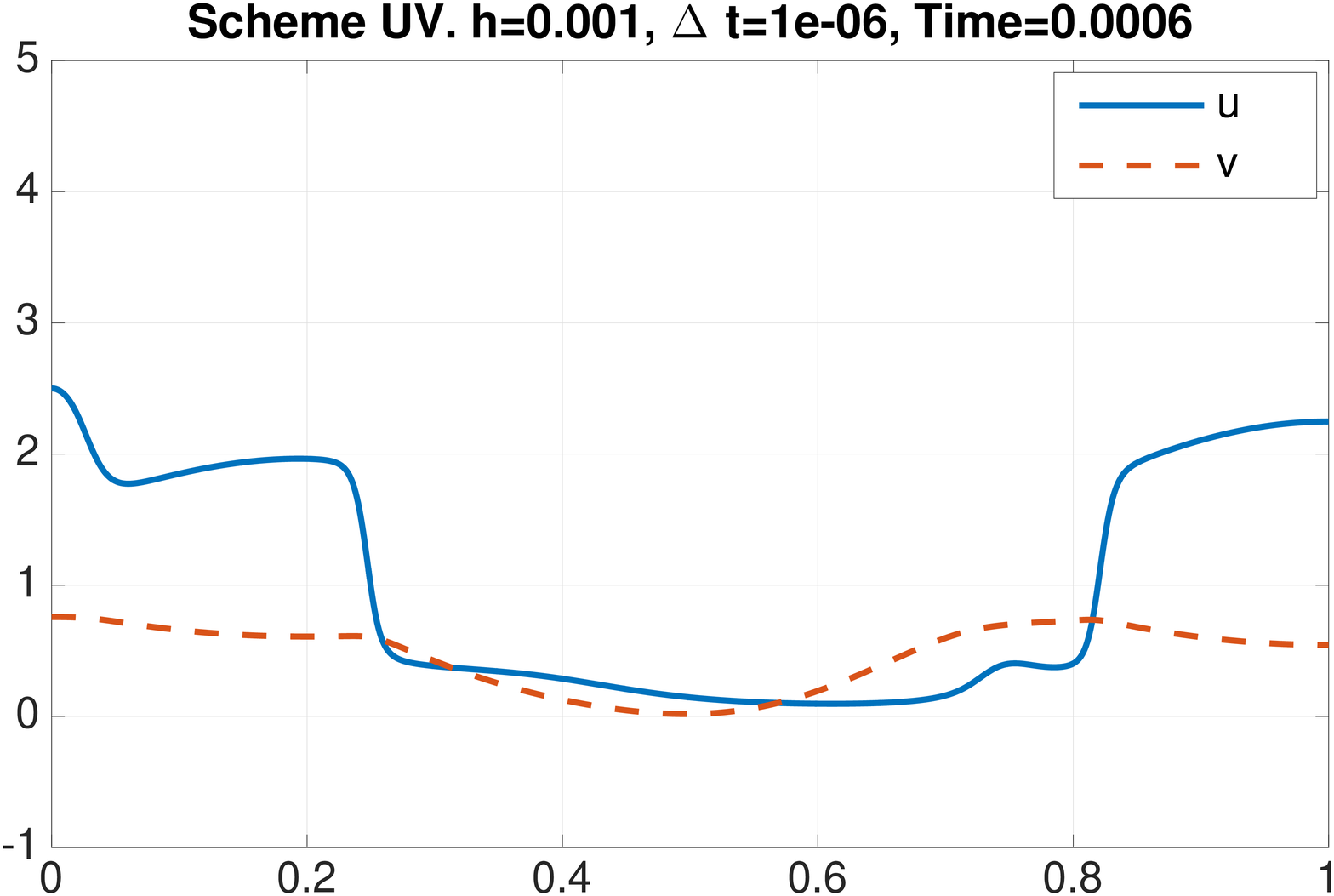}
\includegraphics[width=0.32\textwidth]{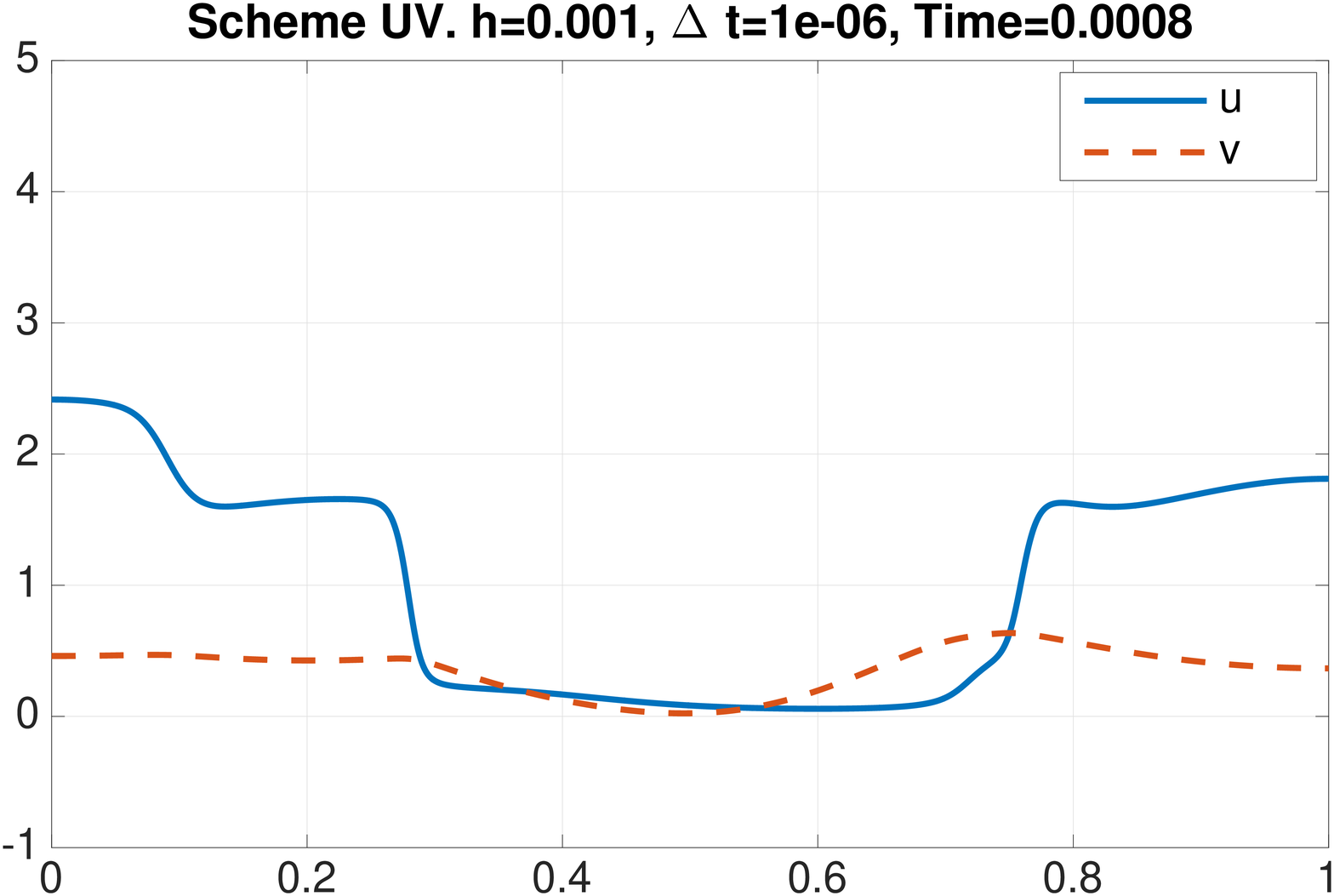}
\includegraphics[width=0.32\textwidth]{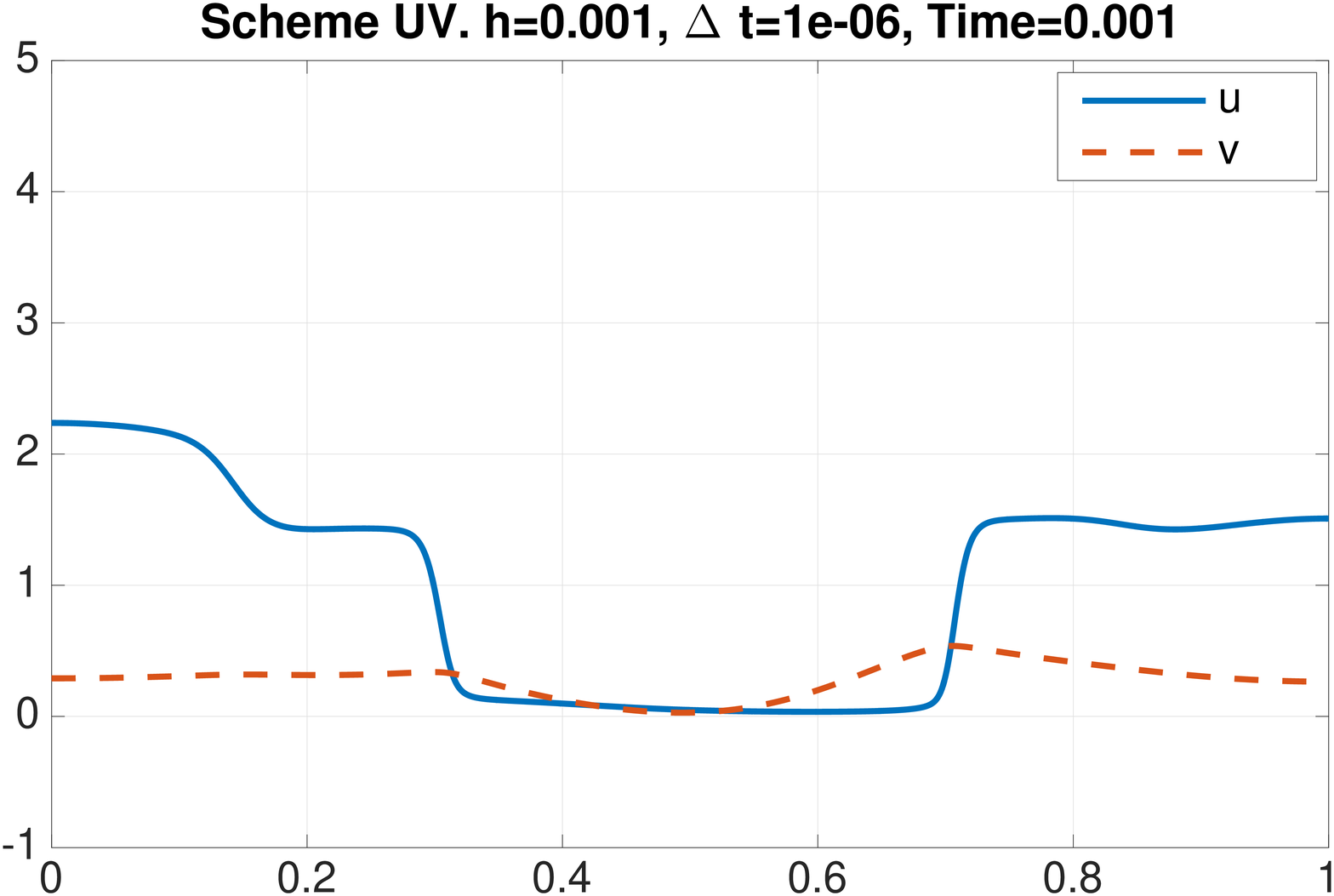}
\\
\includegraphics[width=0.32\textwidth]{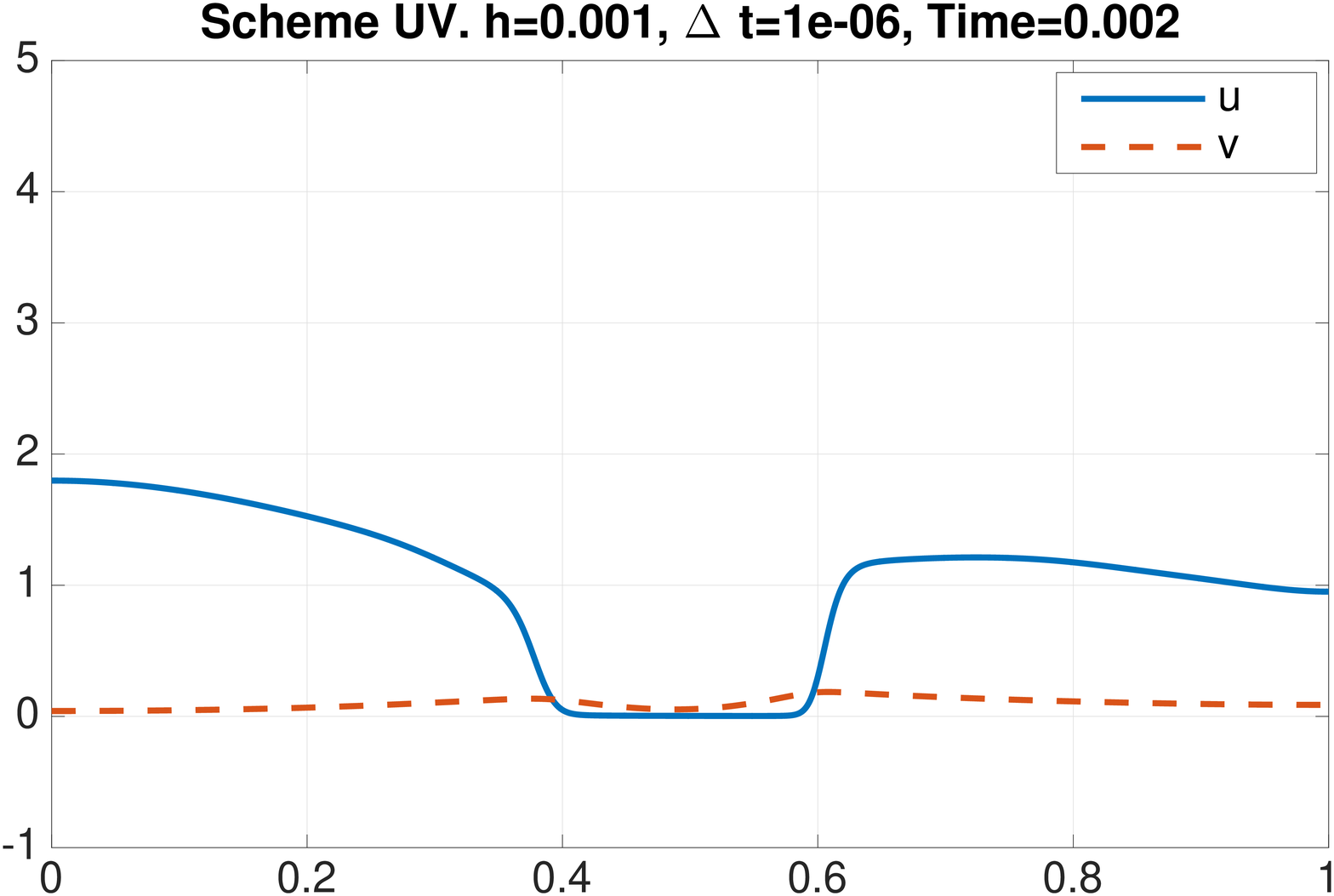}
\includegraphics[width=0.32\textwidth]{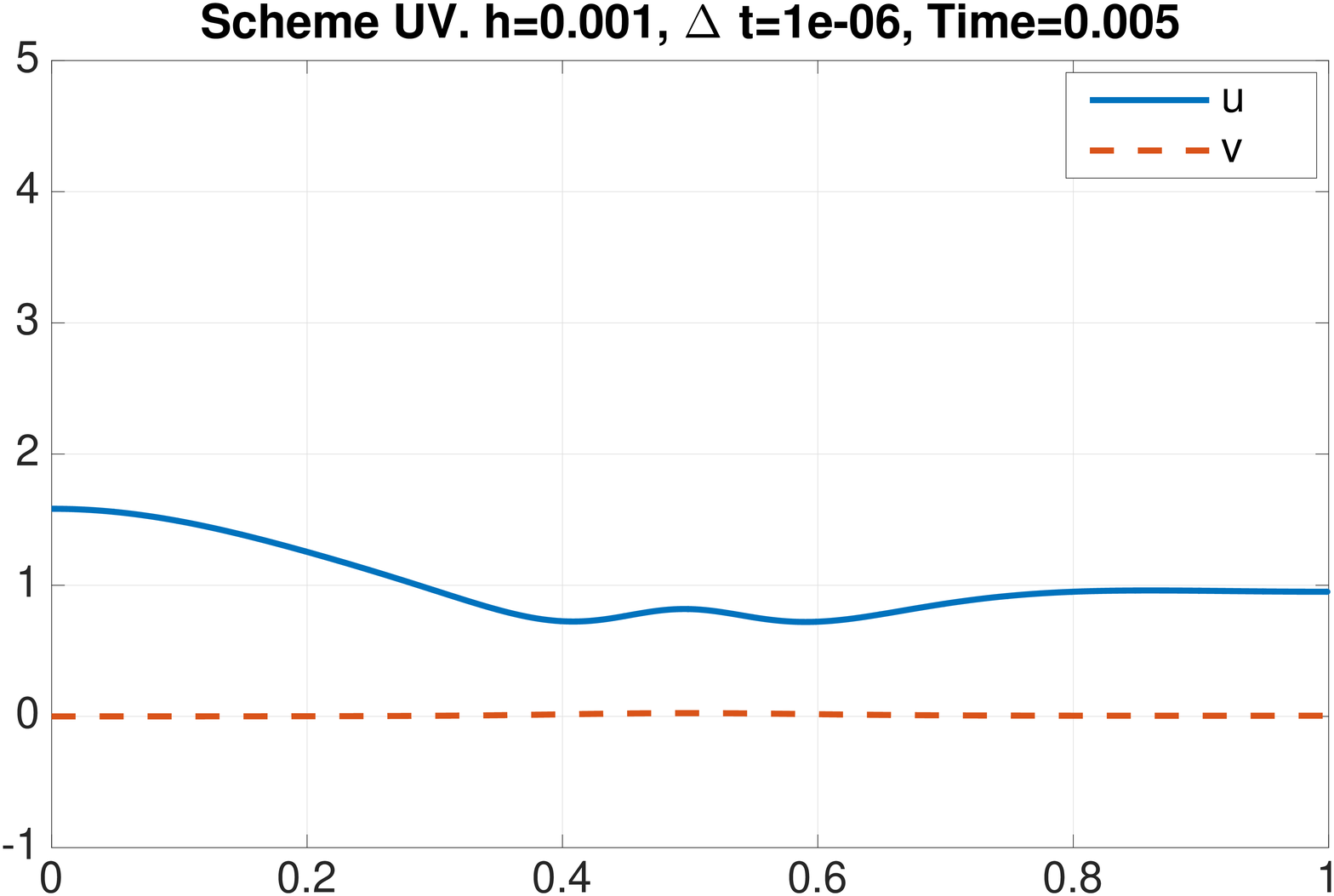}
\includegraphics[width=0.32\textwidth]{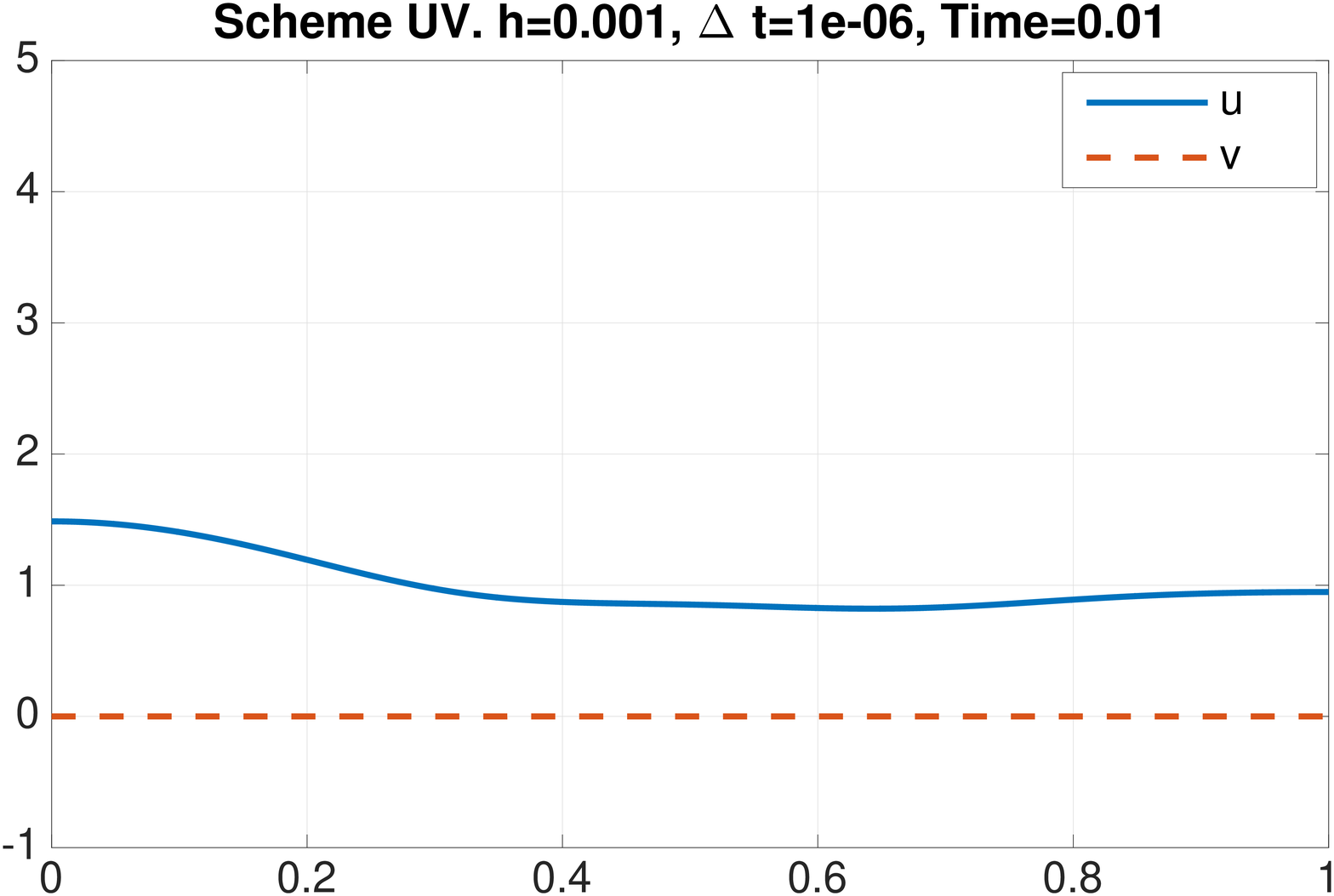}
\\
\includegraphics[width=0.32\textwidth]{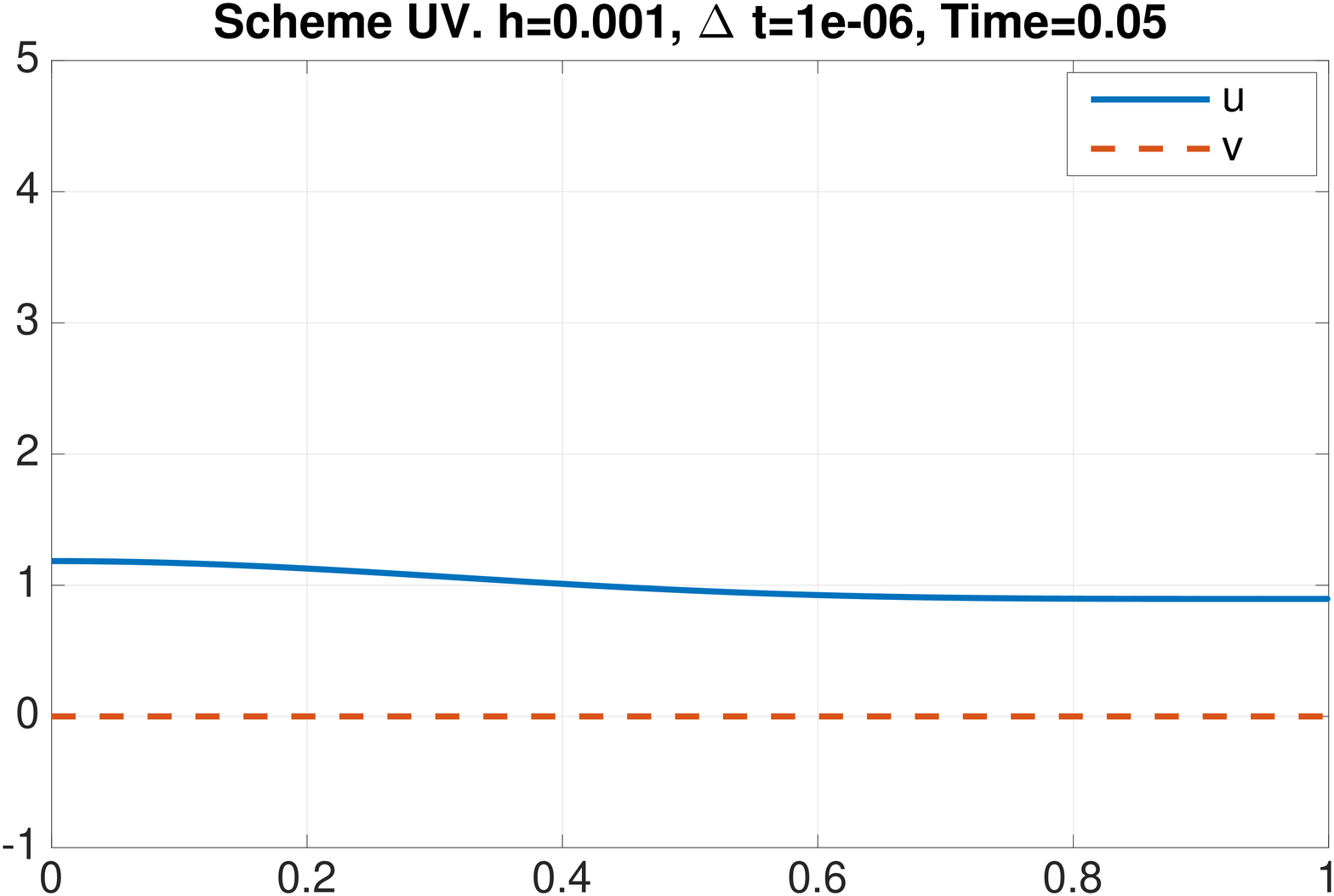}
\includegraphics[width=0.32\textwidth]{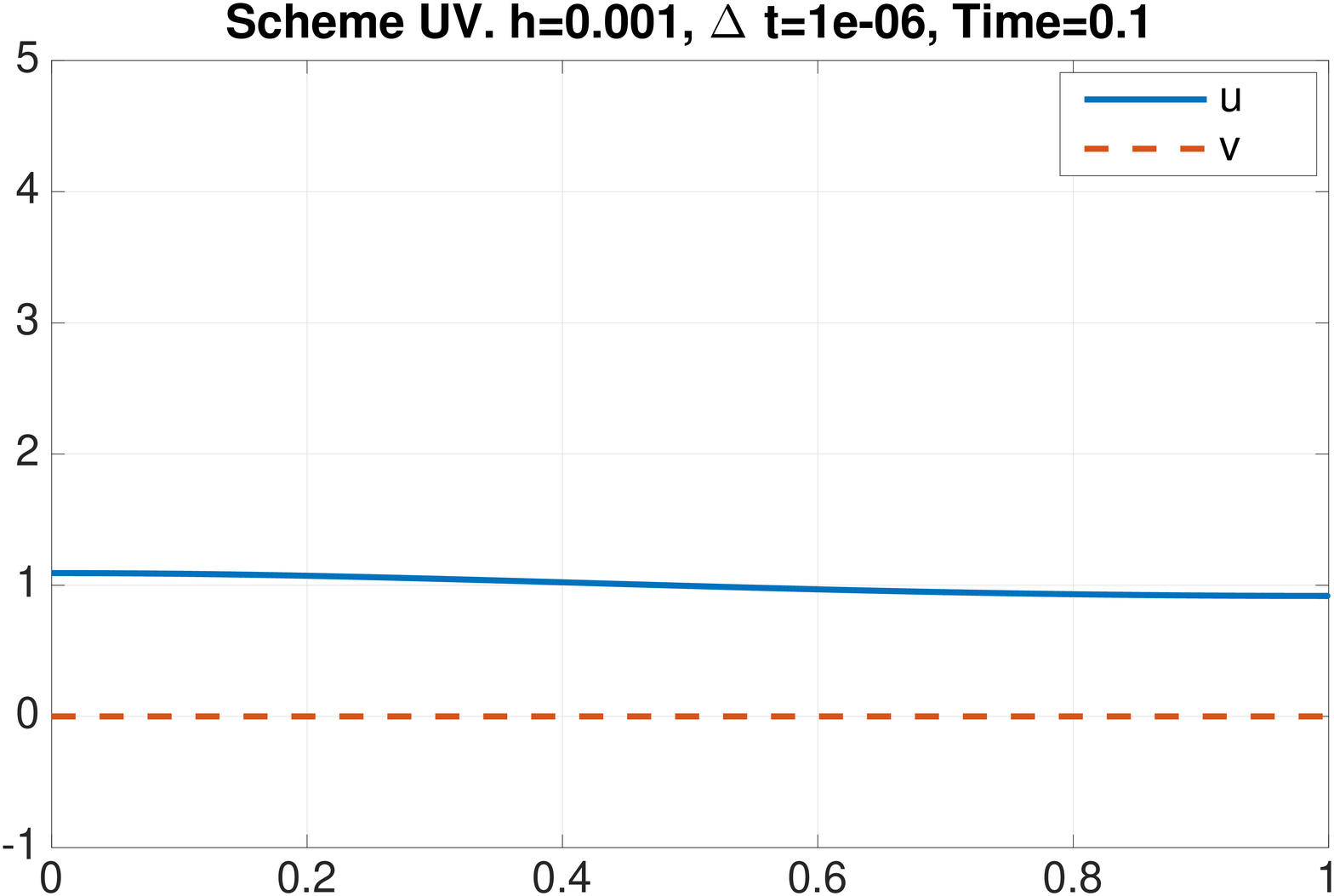}
\includegraphics[width=0.32\textwidth]{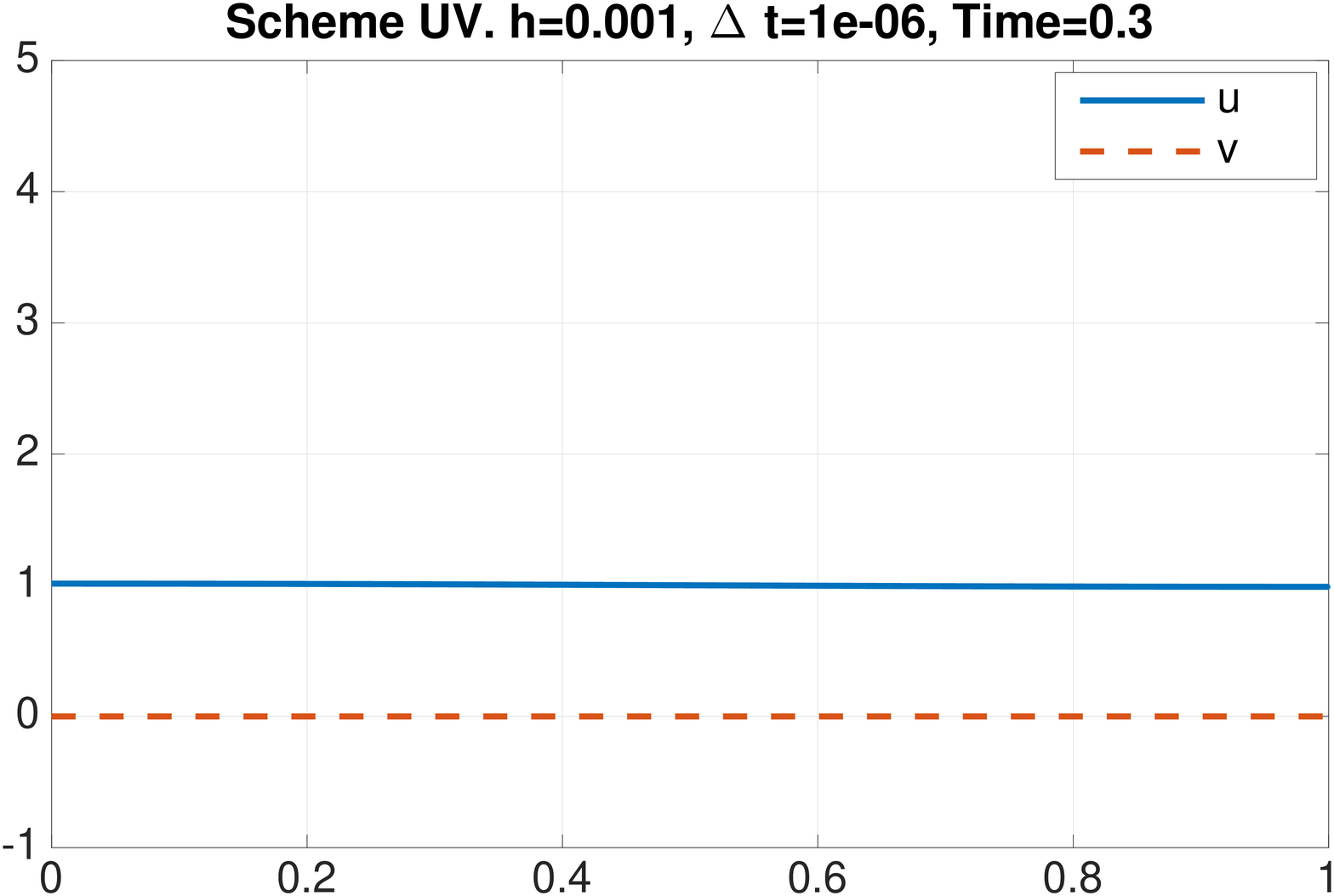}
\caption{Example I: Dynamic of scheme \textbf{UV} using data in \eqref{eq:A_initialdynconst}. The evolution in time of $u$ and $v$ is presented from Left to Right and from Top to Bottom.} \label{fig:Dynamic_UV_A}
\end{center}
\end{figure}

\begin{figure}
\begin{center}
\includegraphics[width=0.43\textwidth]{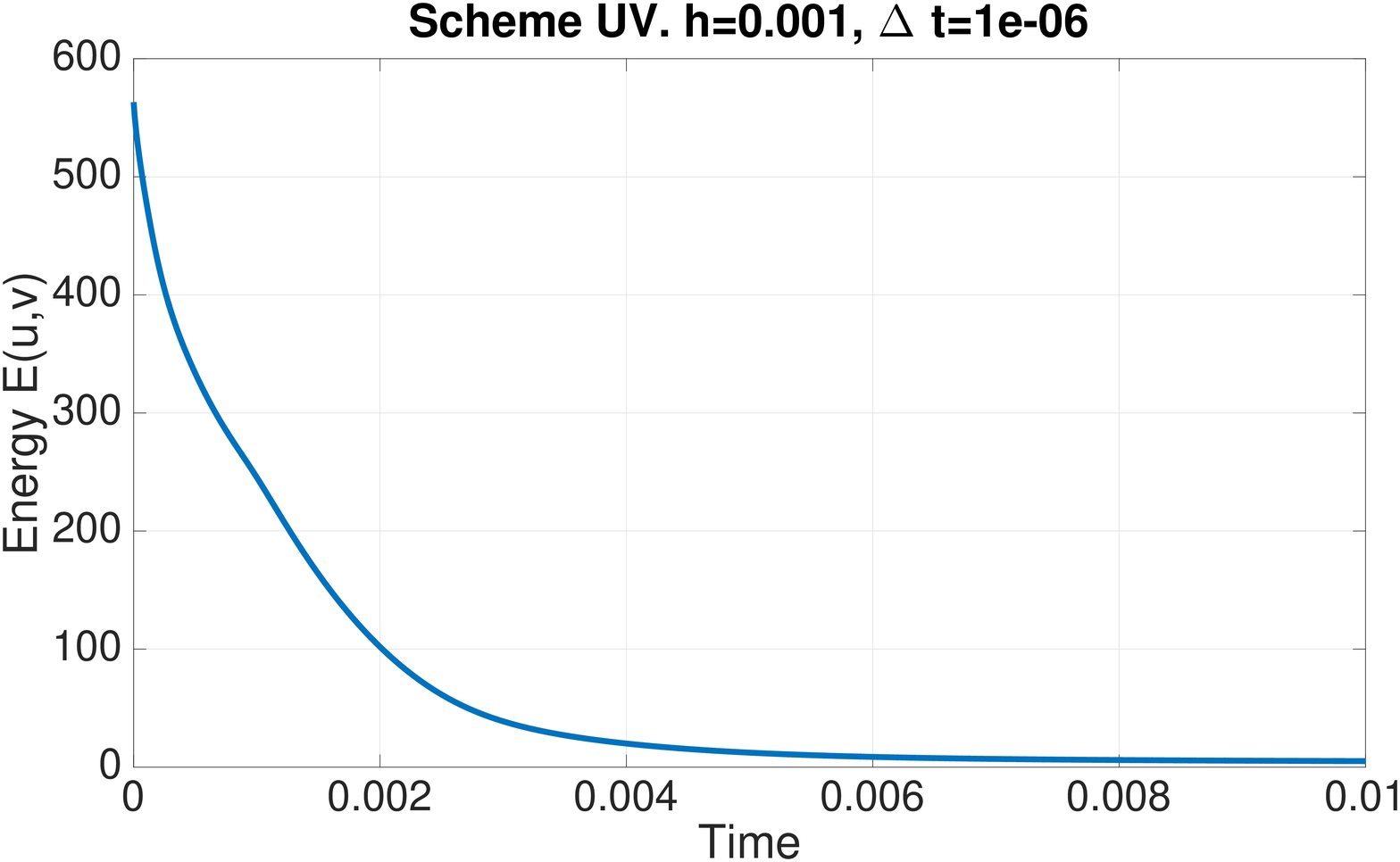}
\includegraphics[width=0.43\textwidth]{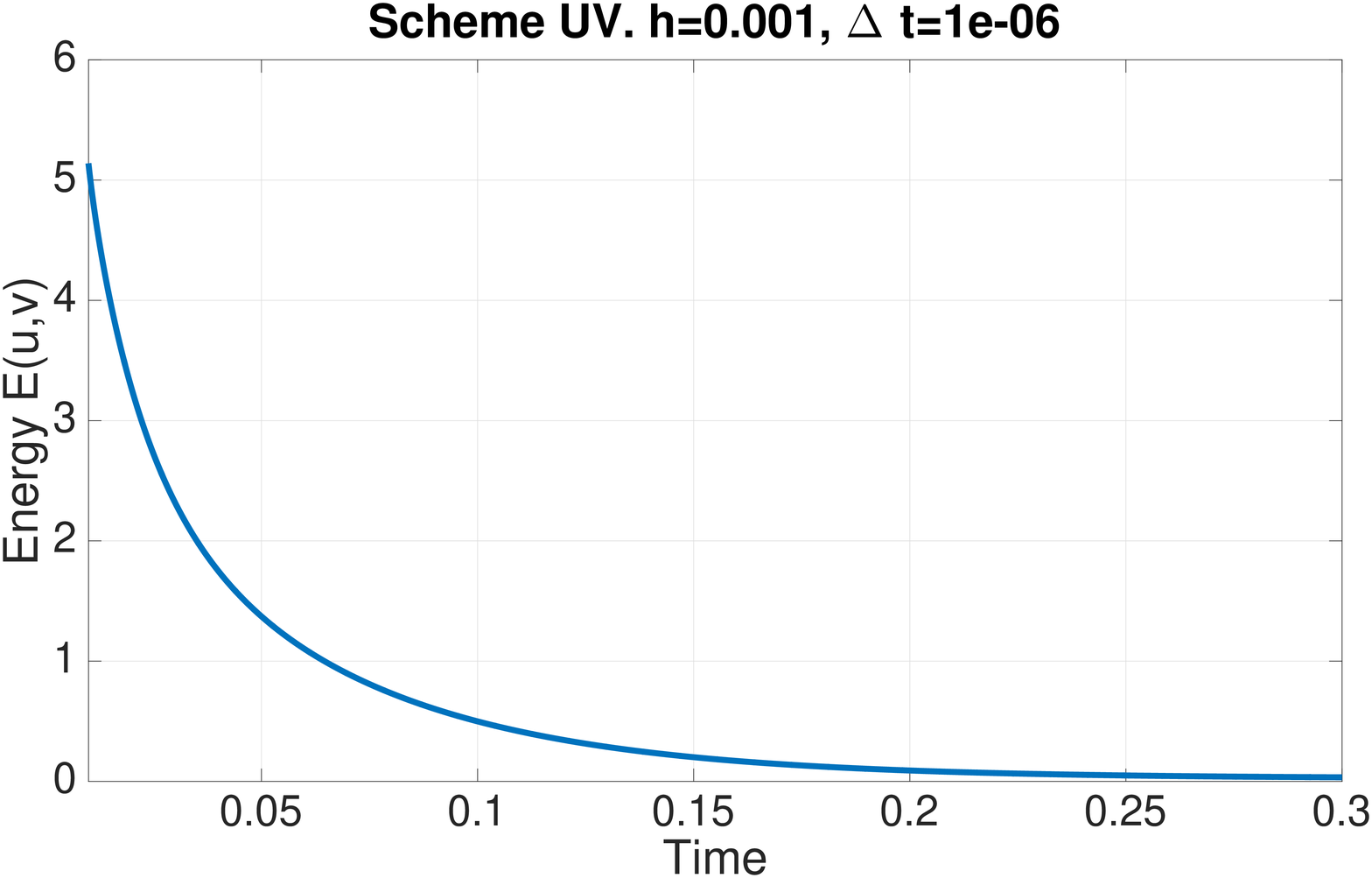}
\caption{Example I: Energy evolution for scheme \textbf{UV} using data in \eqref{eq:A_initialdynconst}. Left: Time interval $[0,0.01]$. Right: Time interval $[0.01,0.3]$ .} \label{fig:A_energy}
\end{center}
\end{figure}

\begin{figure}
\begin{center}
\includegraphics[width=0.43\textwidth]{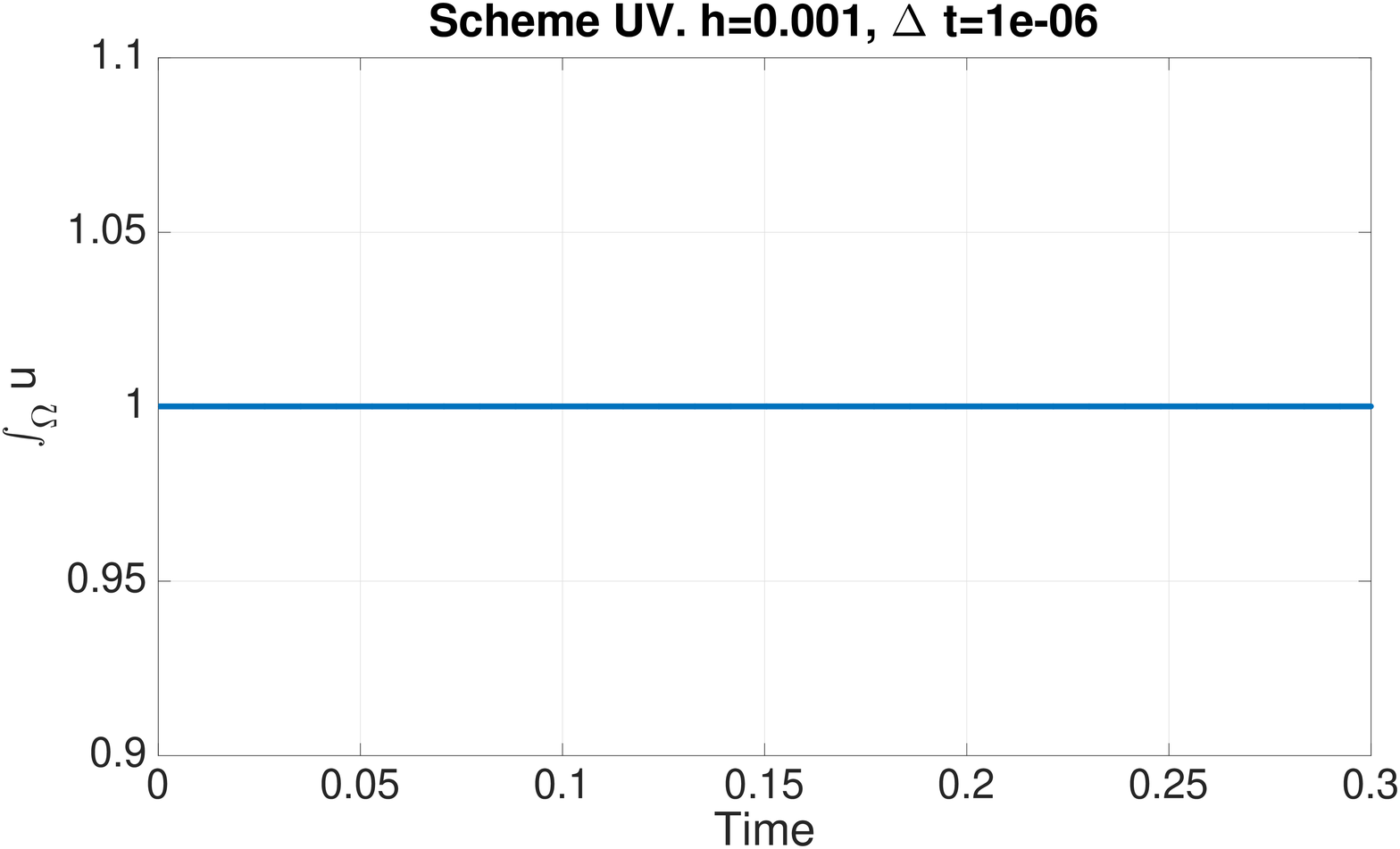}
\includegraphics[width=0.41\textwidth]{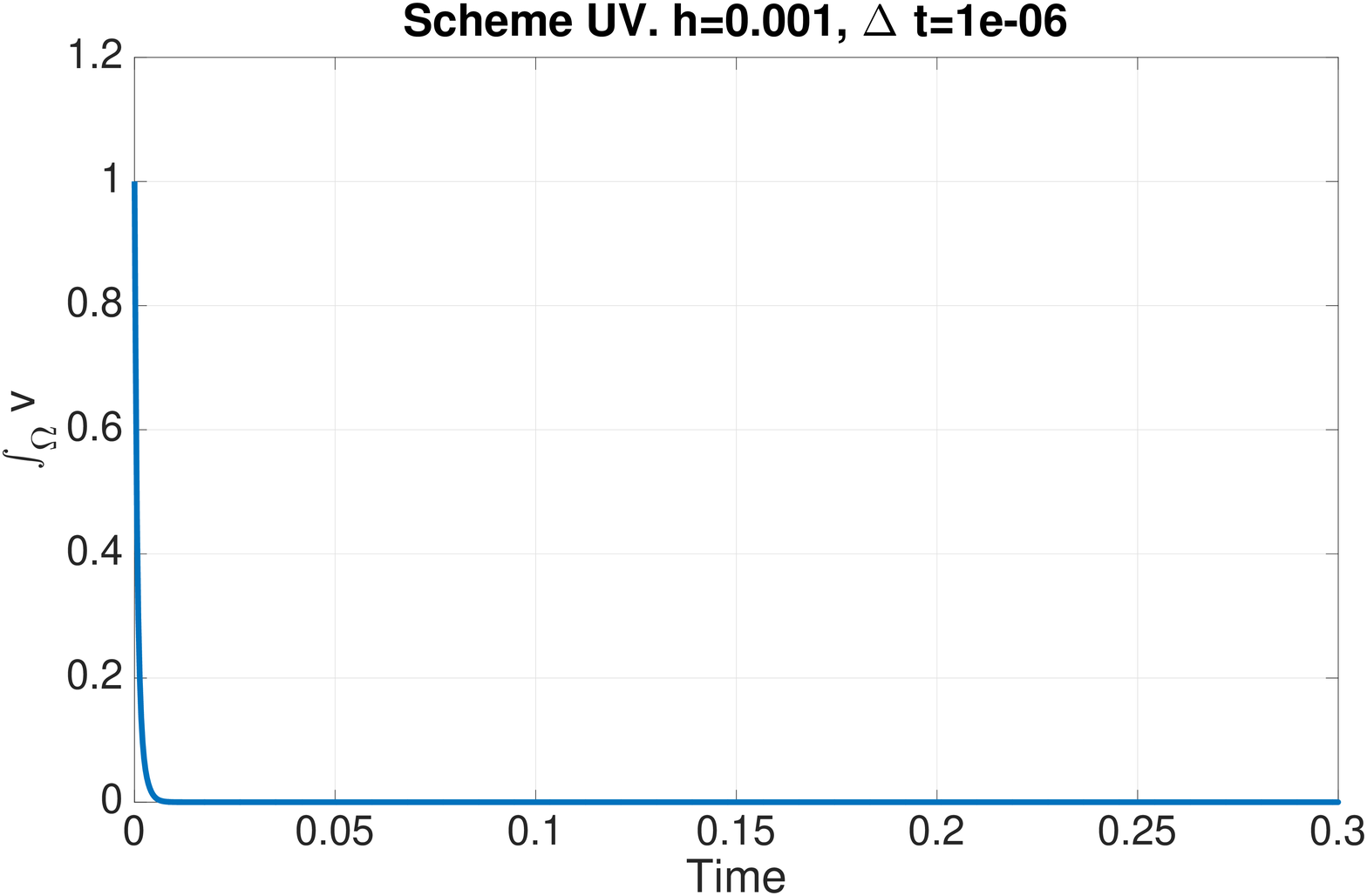}
\caption{Example I: Evolution in time of volume cell population density (Left) and amount of chemical substance (Right). } \label{fig:A_volume}
\end{center}
\end{figure}

\subsection{Example II: Positivity test}

The initial configuration and the physical parameters for this example are:
\beq\label{eq:B_initialdynconst}
\left\{\ba{rcl}
u_0&=&1.1 - e^{-(\frac{x-0.5}{0.1})^2}\,,
\\ \hueco
v_0&=&2 - e^{-(\frac{x-0.5}{0.01})^2}\,,
\ea\right.
\quad
[0,T]\,=\,[0,10^{-4}]\,,
\quad
\chi\,=\,100
\quad
\mbox{ and }
\quad \mu=1\,.
\eeq
The initial configuration corresponds with two symmetric smooth functions with the physical parameters chosen such that the attraction effects dominates over the consumption ones, in order to produce a dynamic with the variable $u$ tending to zero in the central region, and with high gradients on it, in order to test how well the schemes maintain at a discrete level the positivity of the unknown $u$. The exact dynamics of this example is presented in Figure~\ref{fig:Dynamic_UV_B} and it has been computed using scheme \textbf{UV} with parameters $h=10^{-5}$ and $\Delta t=10^{-9}$. 
\begin{figure}
\begin{center}
\includegraphics[width=0.32\textwidth]{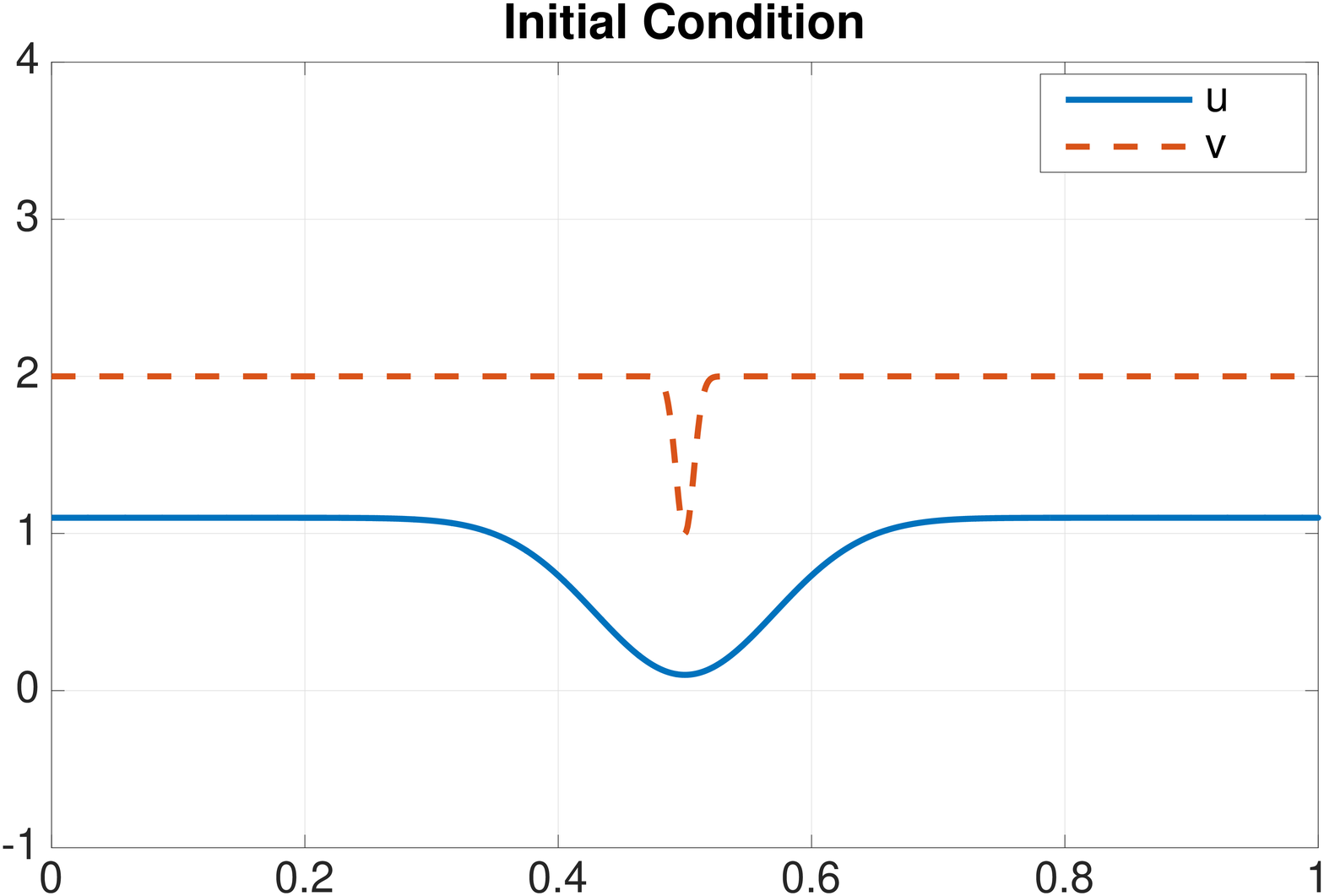}
\includegraphics[width=0.32\textwidth]{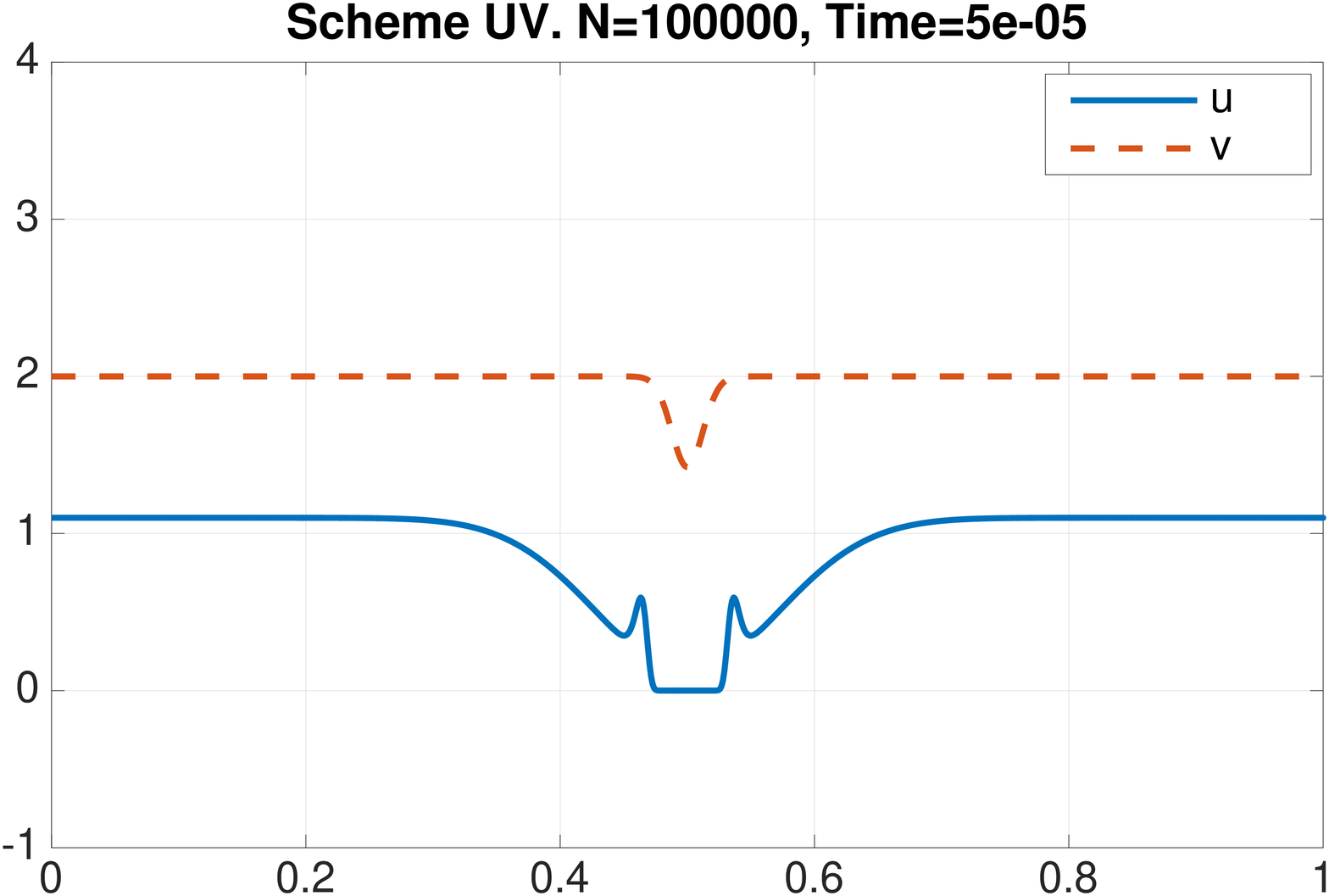}
\includegraphics[width=0.32\textwidth]{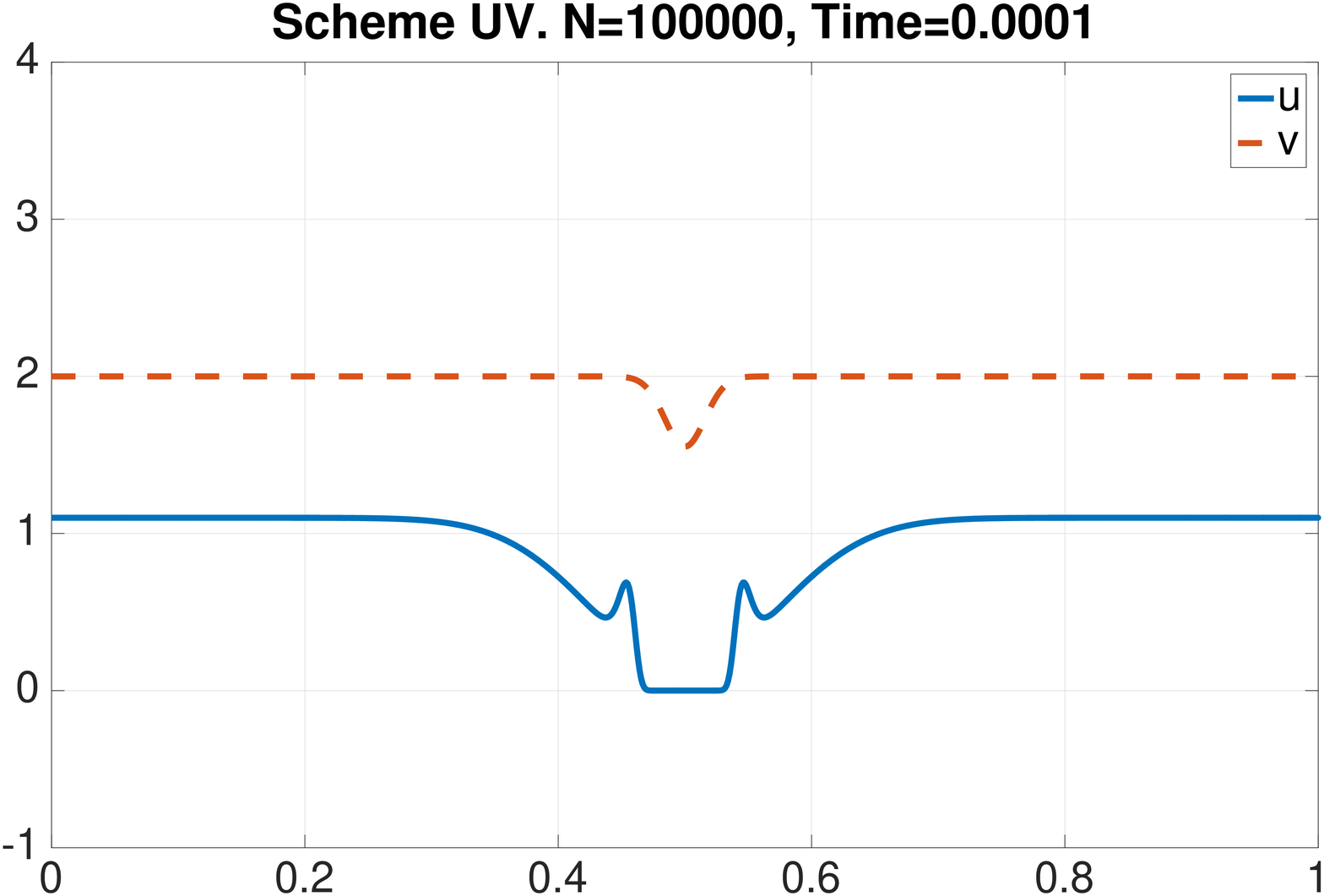}
\caption{Example II: Dynamic of scheme \textbf{UV} using data in \eqref{eq:B_initialdynconst} for the time interval $[0,10^{-4}]$, computed using $h=10^{-5}$ and $\Delta t=10^{-9}$.} \label{fig:Dynamic_UV_B}
\end{center}
\end{figure}
In Table~\ref{tab:Ex2_min} we present the minimum values achieved by $u$ in the whole domain $\Omega$ during the complete time interval $[0,T]$ (i.e. $\min_{t\in[0,T]}\min_{x\in\Omega} u$), for each of the schemes using different values of the discretization parameters $h$ and $\Delta t$.  Scheme \textbf{UV} is able to maintain $u$ positive using the two time steps considered ($\Delta t=10^{-7}, 10^{-8}$) but only for $h$ small enough, in fact when $h\leq 1/5000$. This fact can be explained by taking into account that this discrete formulation when $h$ is small enough is the closest one to the continuous in space problem, which satisfies a maximum principle. 
Schemes \textbf{UV-ND} and \textbf{UV-NS} do not work well when $\Delta t=10^{-7}$, because the iterative methods \eqref{eq:iterative1} and \eqref{eq:iterative2} are not convergent for that choice of the time step. The reason why we obtain some values for scheme \textbf{UV-NS} is because we have added an additional condition to the iterative loop, which states that if after $100$ iterations the stopping criterium is not satisfied, the algorithm has to take the last  obtained result and move forward (hoping that this idea might help the scheme in case that the desired tolerance is not achieved but the error is getting close to it). This idea does not prevent scheme \textbf{UV-ND} to crash, but it helps somehow scheme \textbf{UV-AD} to obtain some intermediate results, not satisfying a good approximation of scheme \textbf{UV-AD}. On the other hand, if we consider $\Delta t=10^{-8}$, then the iterative schemes associated with both schemes, \textbf{UV-ND} and \textbf{UV-NS}, are convergent and both schemes are able to achieve the expected approximated positivity (the obtained values are not strictly positive but they are very small in absolute value). The obtained results are in agreement with the results in Remarks \ref{RemarkPosUV-ND} and \ref{RemarkPosUV-NS}, which state that schemes \textbf{UV-ND} and \textbf{UV-NS} will satisfy  the positivity constraint approximatively.
For scheme \textbf{UV-AD} we can observe that it is always capture the positivity of the unknown $u$ as expected, given that \textit{upwind} schemes were designed with that purpose in mind.
Finally, scheme \textbf{UVS} is not reliable for $\Delta t=10^{-7}$ due to the fact that for small values of $h$ the simulations crashes. But it performs very well when $\Delta t=10^{-8}$ achieving approximated positivity when $h=1/1000$ and strictly positivity when $h\leq 1/5000$. The obtained results are in agreement with  Remark \ref{RemarkPosUS}, which states that scheme \textbf{UVS} will satisfy  the positivity constraint approximatively.
\begin{table}
\begin{center}
\begin{tabular}{|c|c|c|c|c|c|}
\hline                                                       
$\dis\min_{t\in[0,T]}\min_{x\in\Omega} u$ & $h=1/100$  & $h=1/500$  & $h=1/1000$ & $h=1/5000$ & $h=1/10000$ 
\\ \hline    
\multicolumn{6}{|c|}{Scheme \textbf{UV}}                                      
\\ \hline 
$\Delta t=10^{-7}$ & $ -0.3414 $ & $ -0.0674 $ & $ -3.24\times 10^{-4} $ & $ 6.45\times 10^{-22} $ & $ 7.96\times 10^{-22} $ 
\\ \hline    
$\Delta t=10^{-8}$ & $ -0.3450 $ & $ -0.0838 $ & $ -4.50\times 10^{-4} $ & $ 2.61\times 10^{-22} $ & $ 3.28\times 10^{-22} $ 
\\ \hline                                              
\multicolumn{6}{|c|}{Scheme \textbf{UV-ND}}                                      
\\ \hline    
$\Delta t=10^{-7}$ & $ \times $ & $ \times $ & $ \times $ & $ \times $ & $ \times $ 
\\ \hline    
$\Delta t=10^{-8}$ & $ -3.13\times 10^{-4} $ & $ -6.90\times 10^{-6} $ & $ -9.20\times 10^{-7} $ & $ -1.43\times 10^{-8} $ & $ -2.57\times 10^{-9} $ 
\\ \hline                                            
\multicolumn{6}{|c|}{Scheme \textbf{UV-NS}}                                      
\\ \hline    
$\Delta t=10^{-7}$ & $ -0.0011 $ & $ -1.44\times 10^{-5} $ & $ -2.25\times 10^{-6} $ & $ -0.1514 $ & $ -0.1622 $ 
\\ \hline    
$\Delta t=10^{-8}$ & $ -0.0011 $ & $ -1.40\times 10^{-5} $ & $ -1.86\times 10^{-6} $ & $ -2.68\times 10^{-8} $ & $ -4.98\times 10^{-9} $ 
\\ \hline    
\multicolumn{6}{|c|}{Scheme \textbf{UV-AD}}                                      
\\ \hline    
$\Delta t=10^{-7}$ & $ 4.84\times 10^{-5} $ & $ 1.55\times 10^{-10} $ & $ 3.15\times 10^{-14} $ & $ 5.98\times 10^{-20} $ & $ 7.33\times 10^{-21} $ 
\\ \hline    
$\Delta t=10^{-8}$ & $ 4.84\times 10^{-5} $ & $ 1.54\times 10^{-10} $ & $ 2.82\times 10^{-14} $ & $ 3.16\times 10^{-20} $ & $ 3.44\times 10^{-21} $ 
\\ \hline    
\multicolumn{6}{|c|}{Scheme \textbf{UVS}}                                      
\\ \hline    
$\Delta t=10^{-7}$ & $ -0.1425 $ & $ -0.0065 $ & $ -1.50\times 10^{-4} $ & $ \times $ & $ \times $ 
\\ \hline    
$\Delta t=10^{-8}$ & $ -0.1454 $ & $ -0.0118 $ & $ -1.63\times 10^{-4} $ & $ 2.32\times 10^{-22}$ & $ 2.90\times 10^{-22} $ 
\\ \hline    
\end{tabular}
\end{center}
\caption{Example II: Evolution in time of the minimum values achieved by $u$ in the domain $\Omega$ in the whole interval $[0, T ]$ for different values of the discretization parameters $h$  and $\Delta t$.}\label{tab:Ex2_min}
\end{table}

\subsection{Example III: Energy stability test}

The aim of this test is to compare the approximation of the evolution of the energy for the different schemes. 
The initial configuration are two smooth functions with different amplitude and the physical parameters are:
\beq\label{eq:C_initialdynconst}
\left\{\ba{rcl}
u_0&=&4(2.0001 + \cos(7\pi x))\,,
\\ \hueco
v_0&=&3(2.0001 + \cos(12\pi x))\,,
\ea\right.
\quad
[0,T]\,=\,[0,10^{-4}]\,,
\quad
\chi\,=\,30
\quad
\mbox{ and }
\quad \mu=10000\,.
\eeq
The parameters are chosen in order that consumption effects are stronger than the attraction ones. We present the exact dynamics of the system in Figure~\ref{fig:Dynamic_UV_C}, which has been computed using scheme \textbf{UV} with discretization parameters $h=10^{-5}$ and $\Delta t=10^{-8}$. 
The evolution of the energy for schemes \textbf{UV}, \textbf{UV-ND}, \textbf{UV-NS} and \textbf{UV-AD} is presented in Figure~\ref{fig:Energy_UVs_C},
and for scheme \textbf{UVS} in Figure~\ref{fig:Energy_UVS_C}. We can observe how schemes \textbf{UV, UV-ND, UV-NS} and \textbf{UV-AD} approximates well the evolution of the energy once the discretization parameters are small enough (for this example the requirements are $h\leq 10^{-3}$ and $\Delta t \leq 10^{-7}$). In fact, it is interesting to mention that scheme \textbf{UV} is able to produce a reasonable approximation even when $\Delta t=10^{-6}$ while schemes \textbf{UV-ND} and \textbf{UV-NS} crashes. Moreover, it is interesting to check Figure~\ref{fig:UVAD_oscillations} to understand why scheme \textbf{UV-AD} does not perform well with the choice $\Delta t=10^{-6}$. The reason is that this scheme has been designed as an \textit{upwind} scheme, which are a type of schemes famous for behaving well when approximating transport effects due to some (incompresible) velocity, while maintaining the positivity of the solution variable. But, 
in the chemotaxis model considered in this work, the transport of $u$ is due to the term $\nabla\cdot(u\nabla v)$ so that the transport velocity is represented by $\nabla v$ which in general is not incompresible (because $\nabla\cdot(\nabla v)\not=0$ in general). Then, it  produces nonphysical oscillations in the solution if the time step is not small enough (although the scheme always preserve the positivity of the solution as expected).
On the other hand, scheme \textbf{UVS} does not seem to produce as good approximations of the evolution of the exact energy $E(u,v)$ as the other schemes are producing with the considered time steps. In fact, what it is known by Corollary~\ref{cor:energy} is that this scheme satisfy the energy law \eqref{eq:sigschenelaw} for $E(u,\bsigma)$, 
but it is not clear that this property implies that the evolution of energy $E(u,\bsigma)$ has to exactly match the evolution of energy $E(u,v)$. As a matter of fact, we can observe in Figure~\ref{fig:Energy_UVS_C}, that those energies do not match in this example for the considered time steps.

\begin{figure}
\begin{center}
\includegraphics[width=0.32\textwidth]{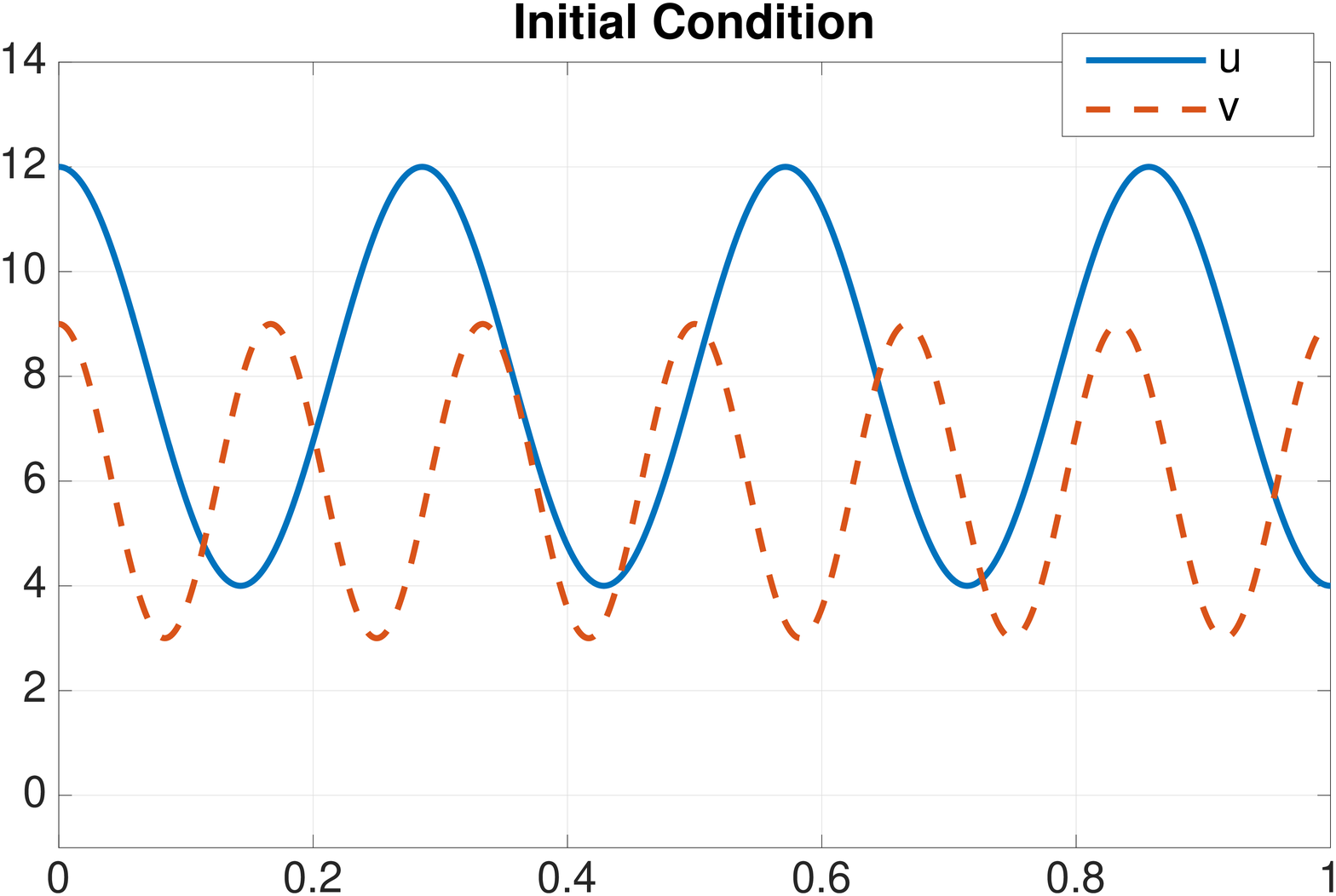}
\includegraphics[width=0.32\textwidth]{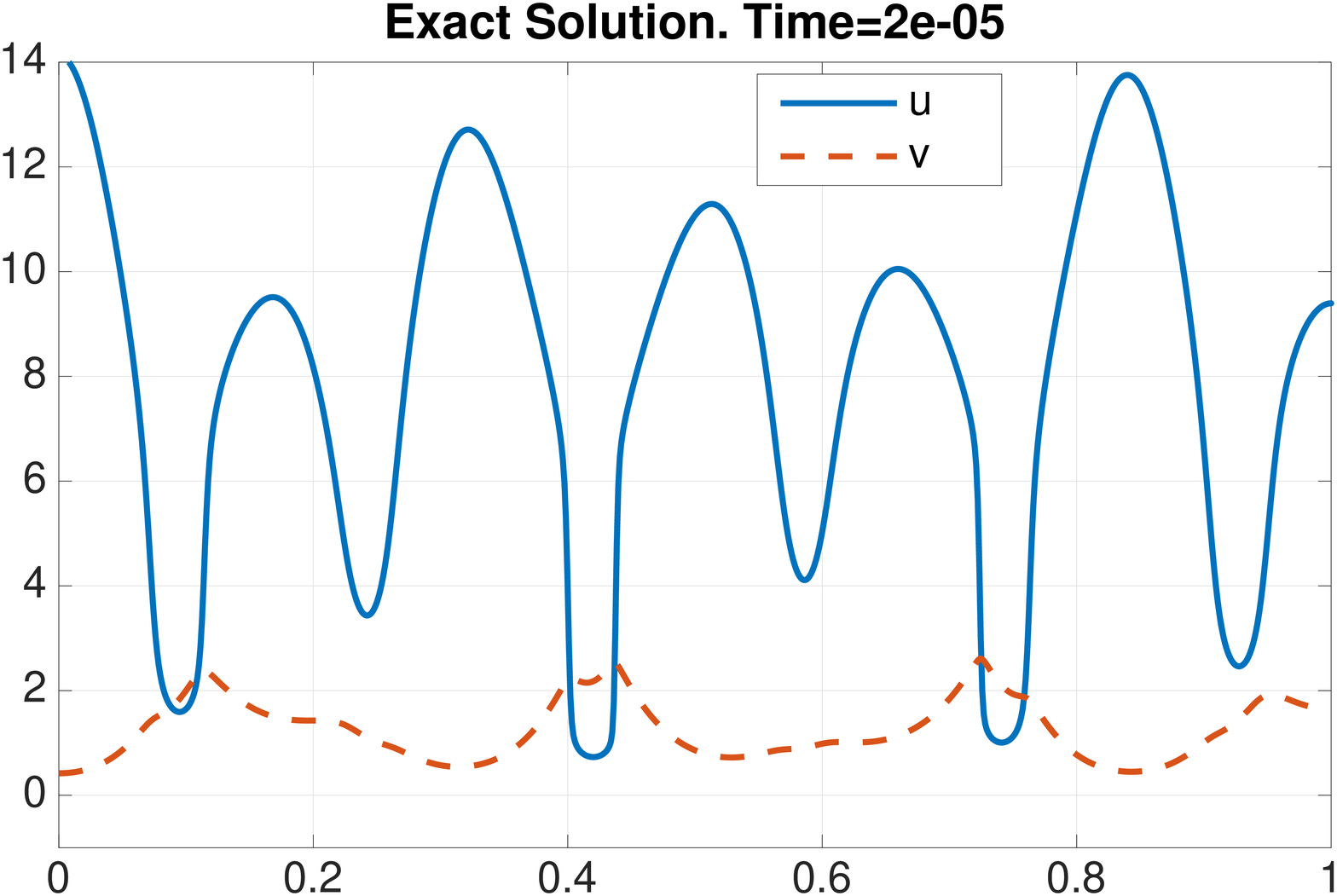}
\includegraphics[width=0.32\textwidth]{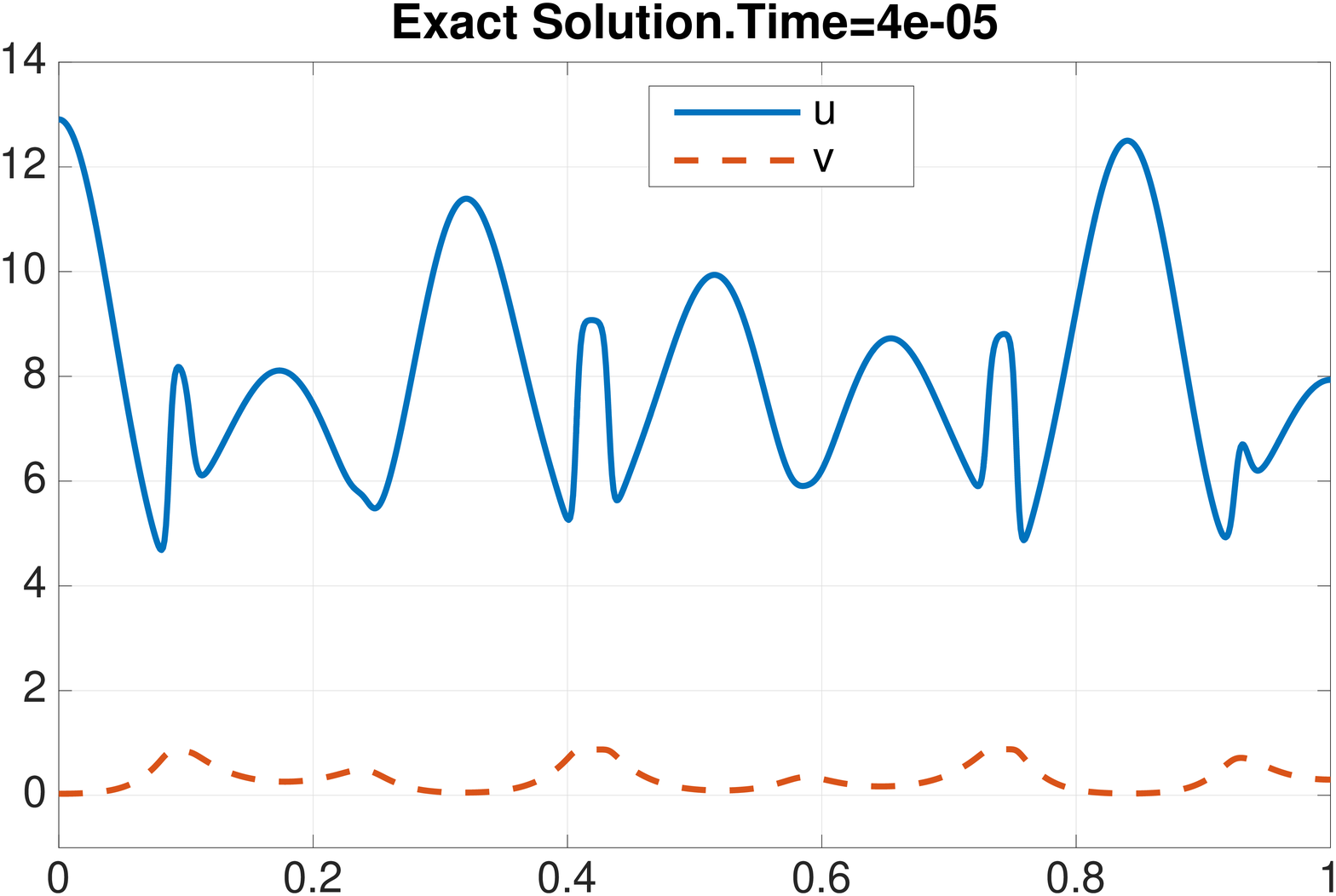}
\\
\includegraphics[width=0.32\textwidth]{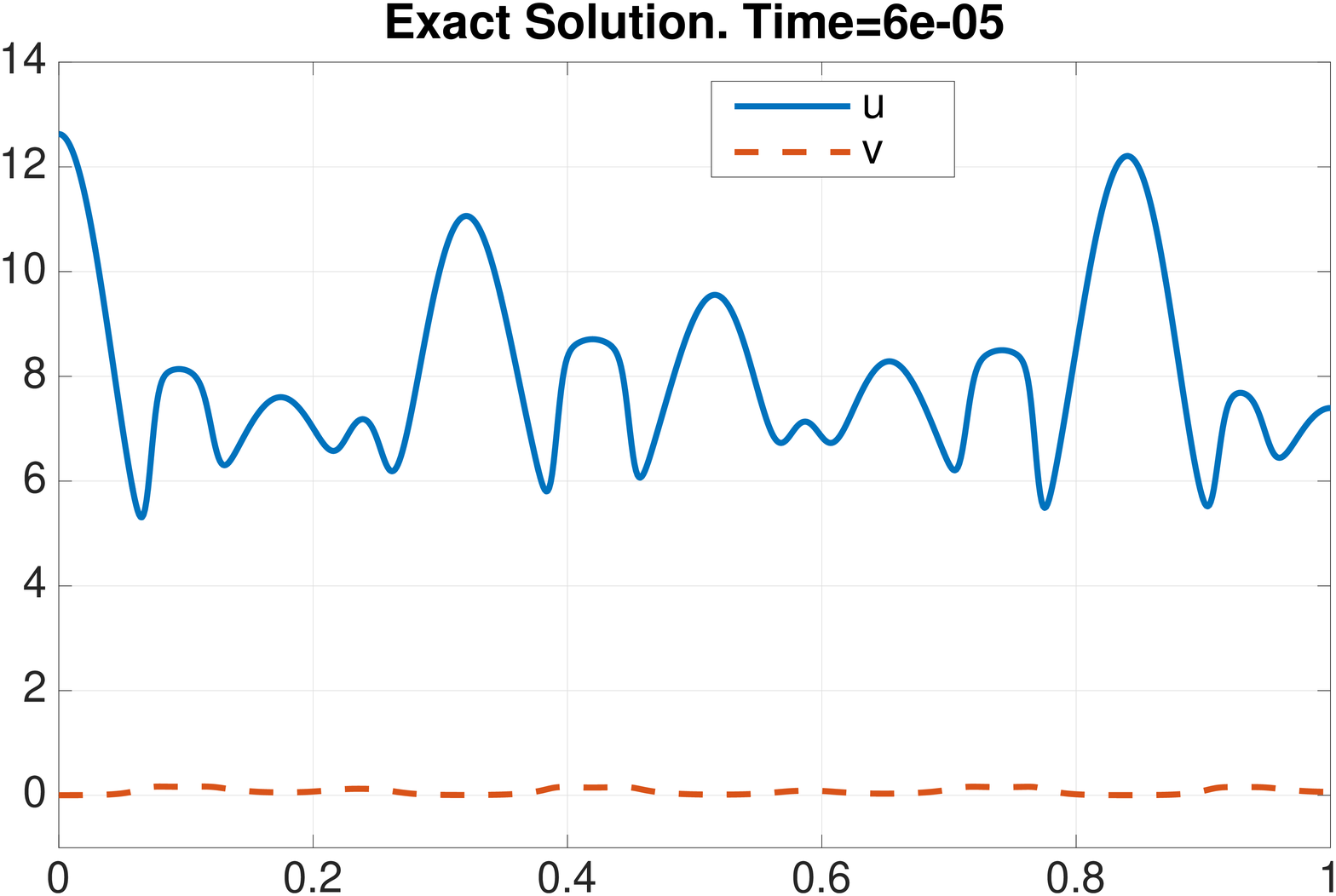}
\includegraphics[width=0.32\textwidth]{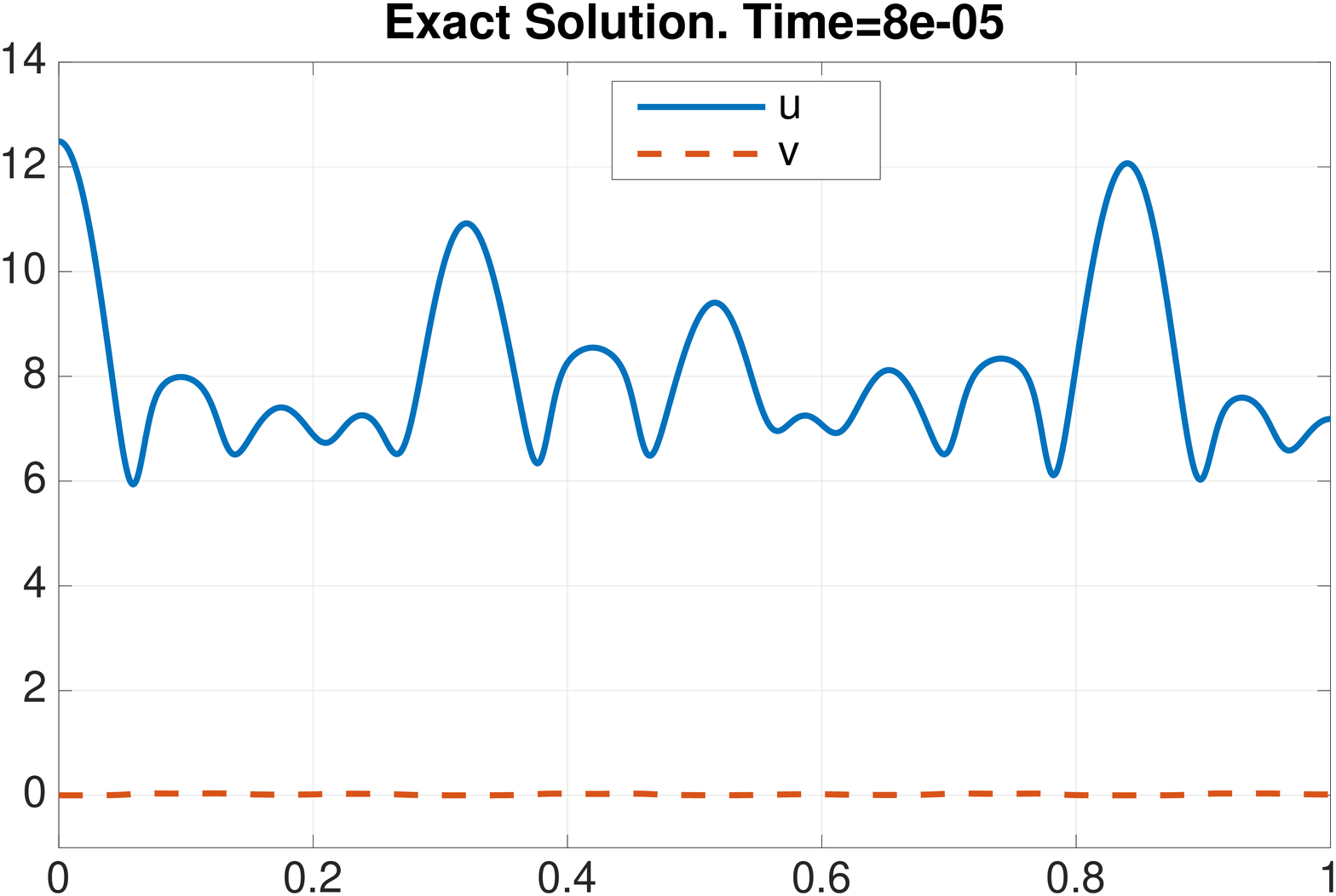}
\includegraphics[width=0.32\textwidth]{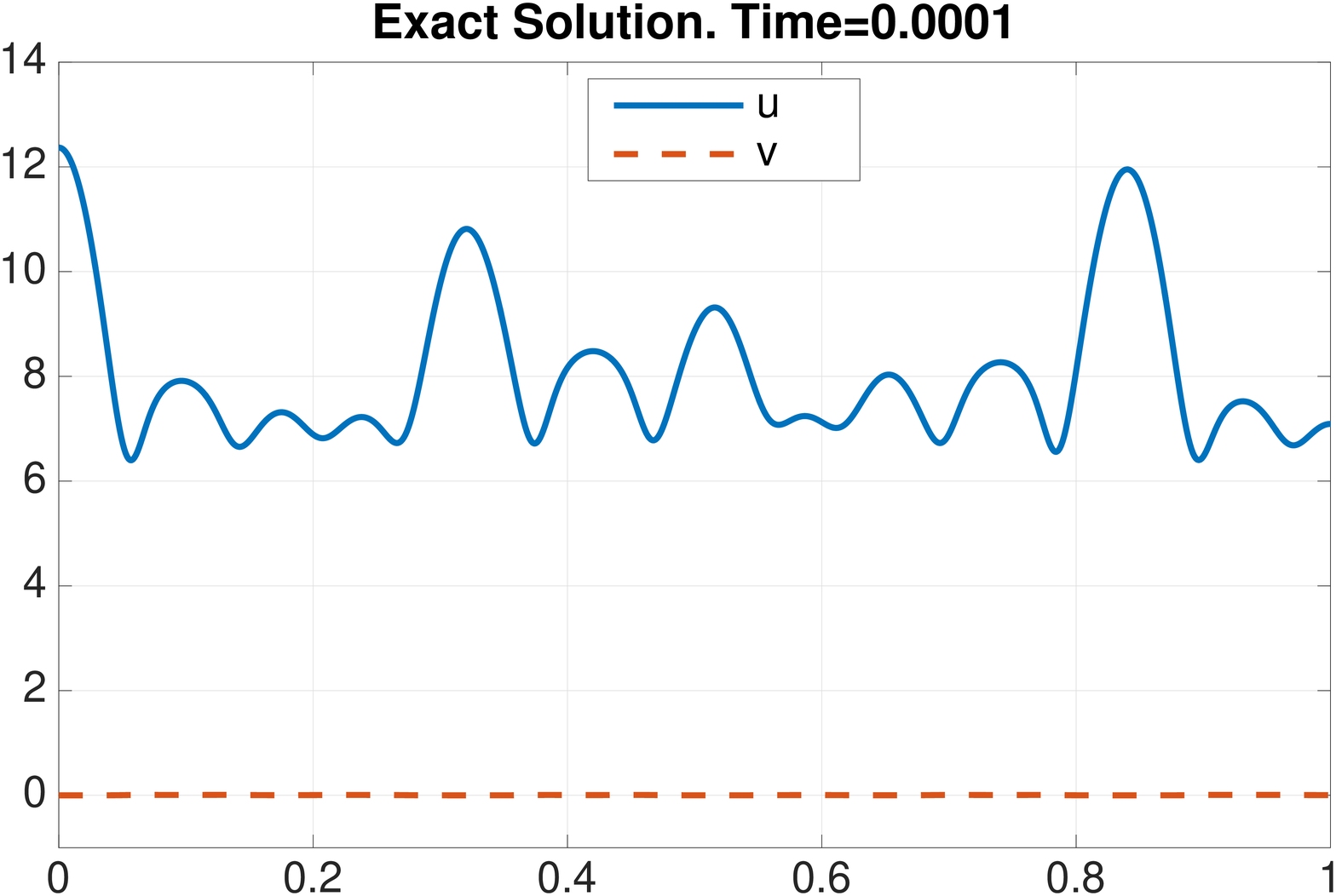}
\caption{Example III: Dynamic of scheme \textbf{UV} using data in \eqref{eq:C_initialdynconst}.} \label{fig:Dynamic_UV_C}
\end{center}
\end{figure}

\begin{figure}
\begin{center}
\includegraphics[width=0.43\textwidth]{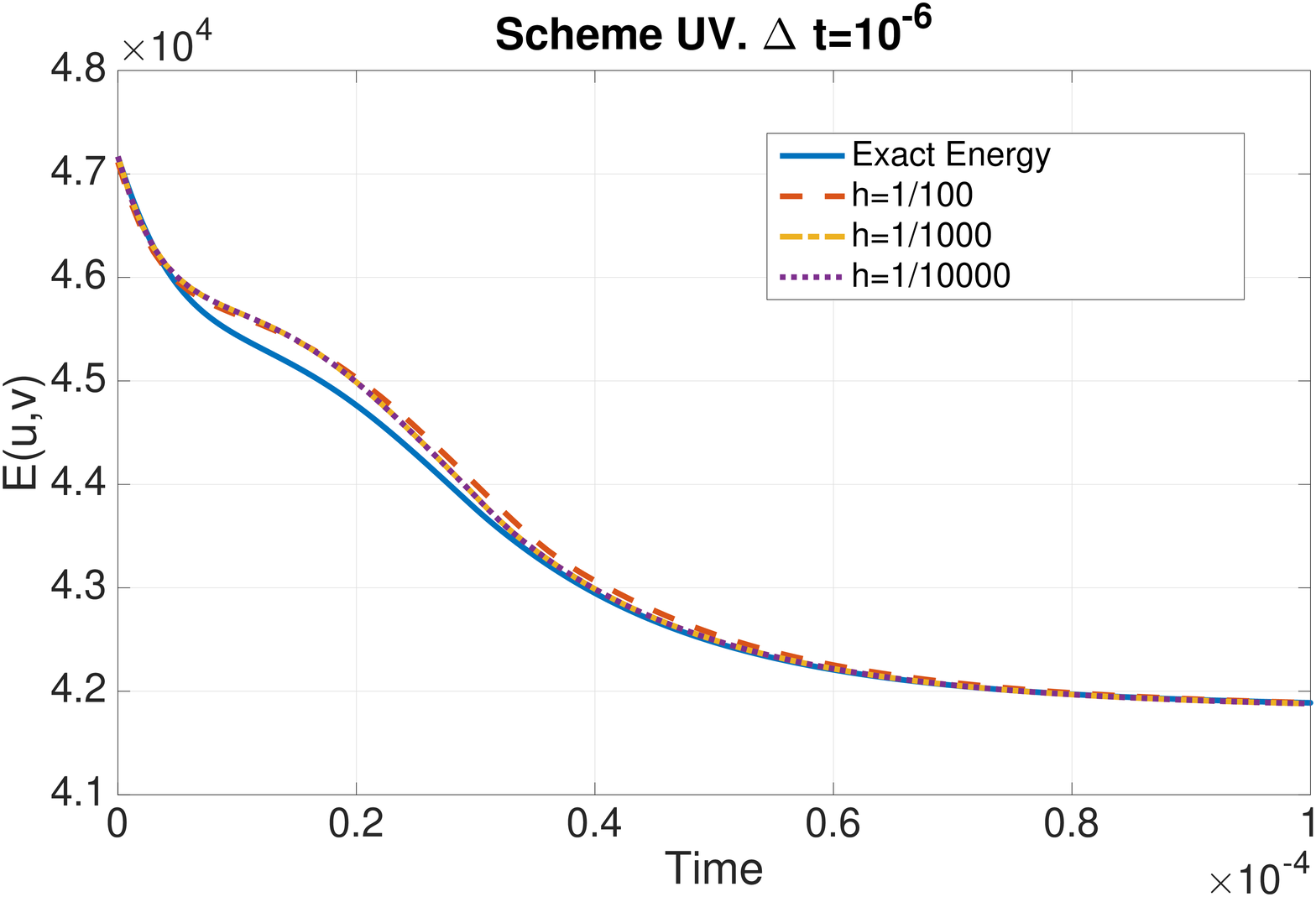}
\includegraphics[width=0.43\textwidth]{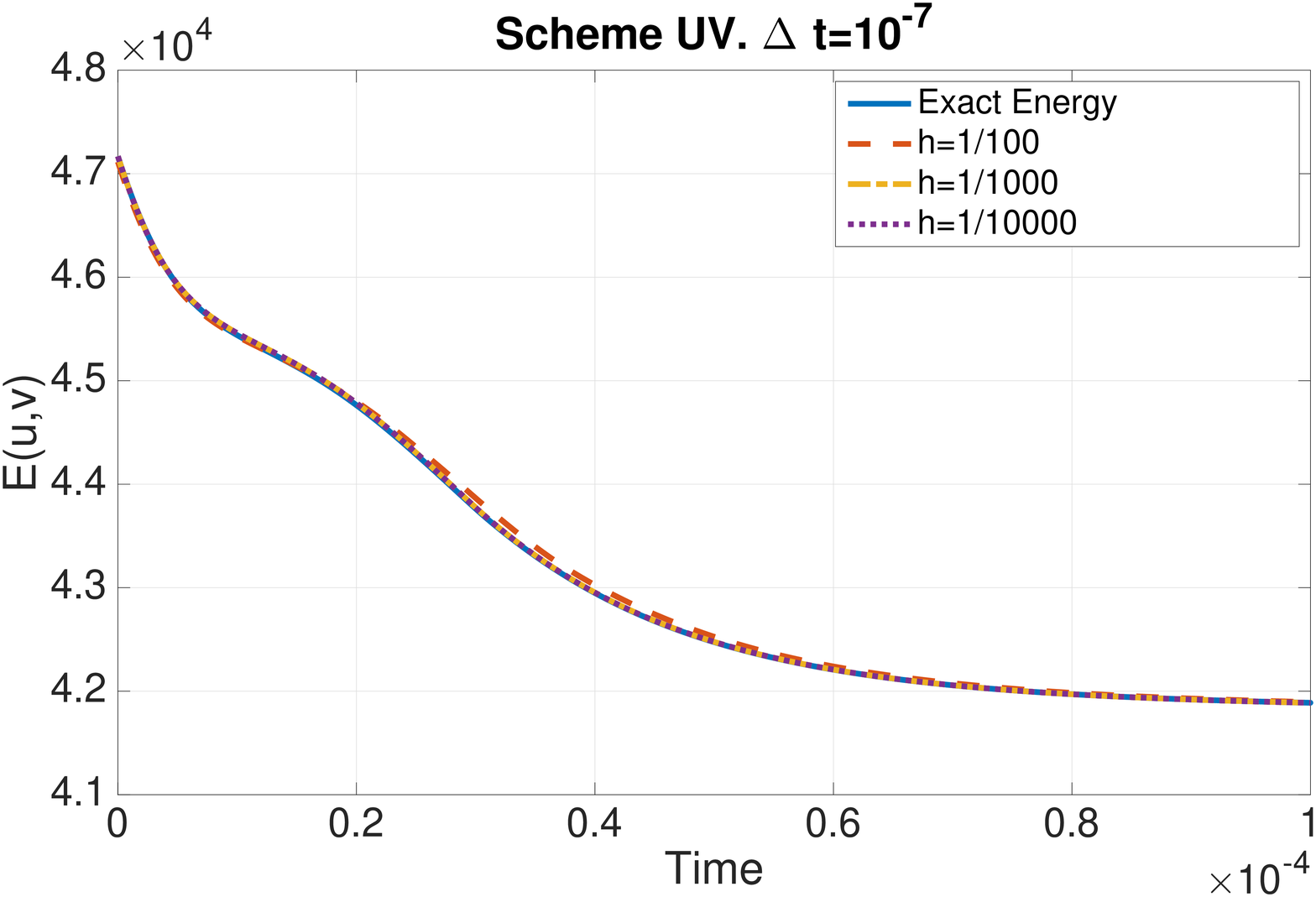}
\\
\includegraphics[width=0.43\textwidth]{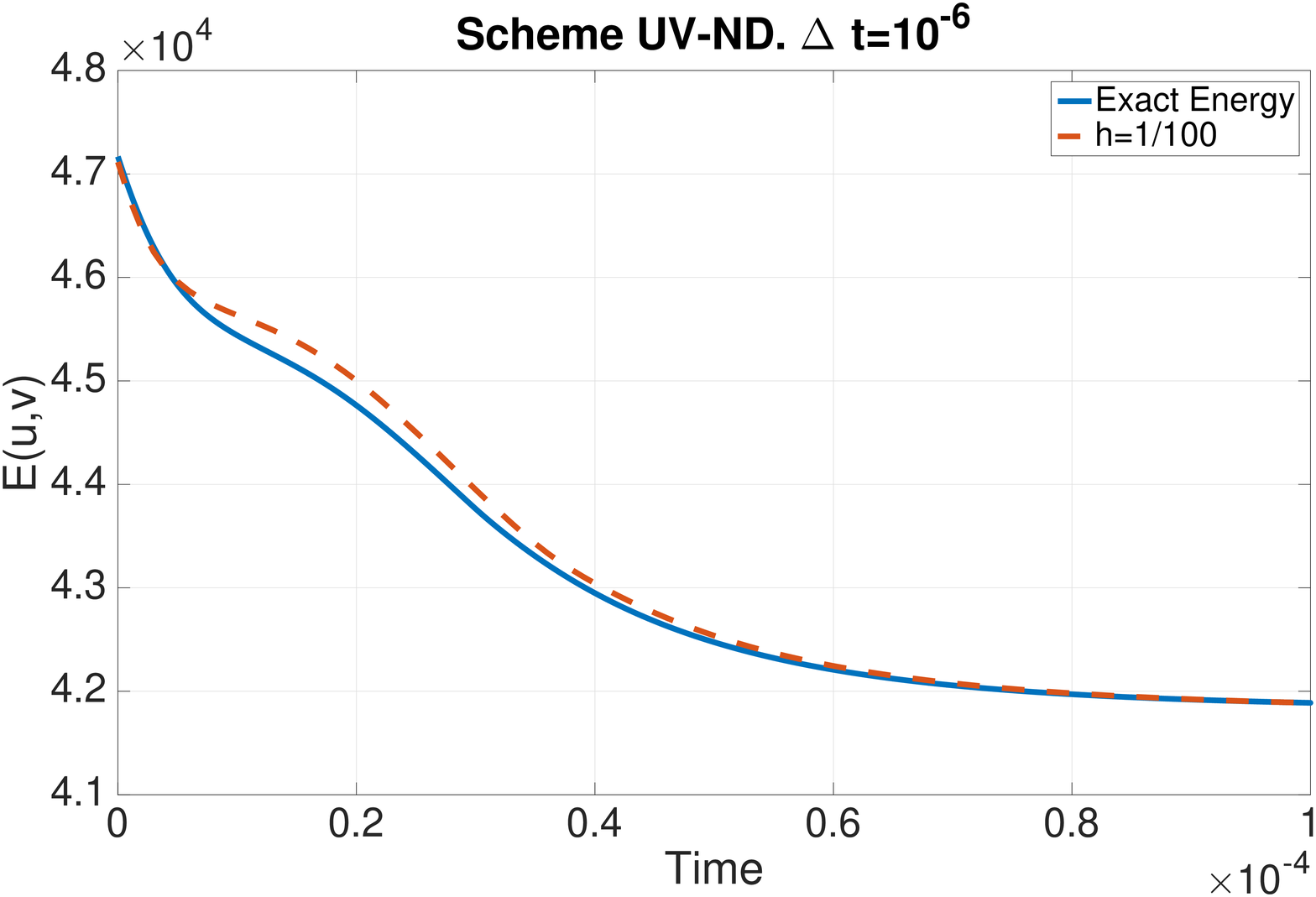}
\includegraphics[width=0.43\textwidth]{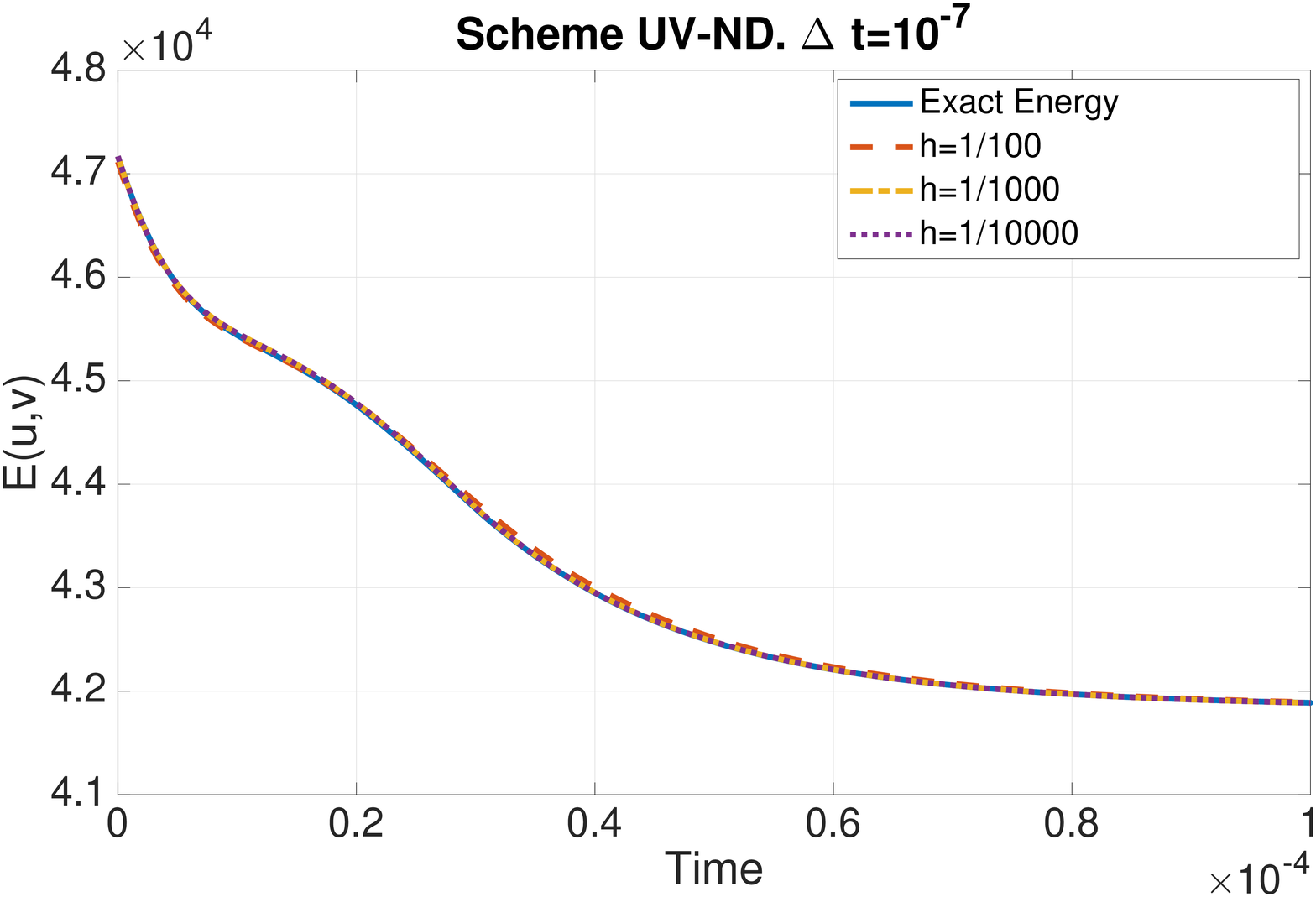}
\\
\includegraphics[width=0.43\textwidth]{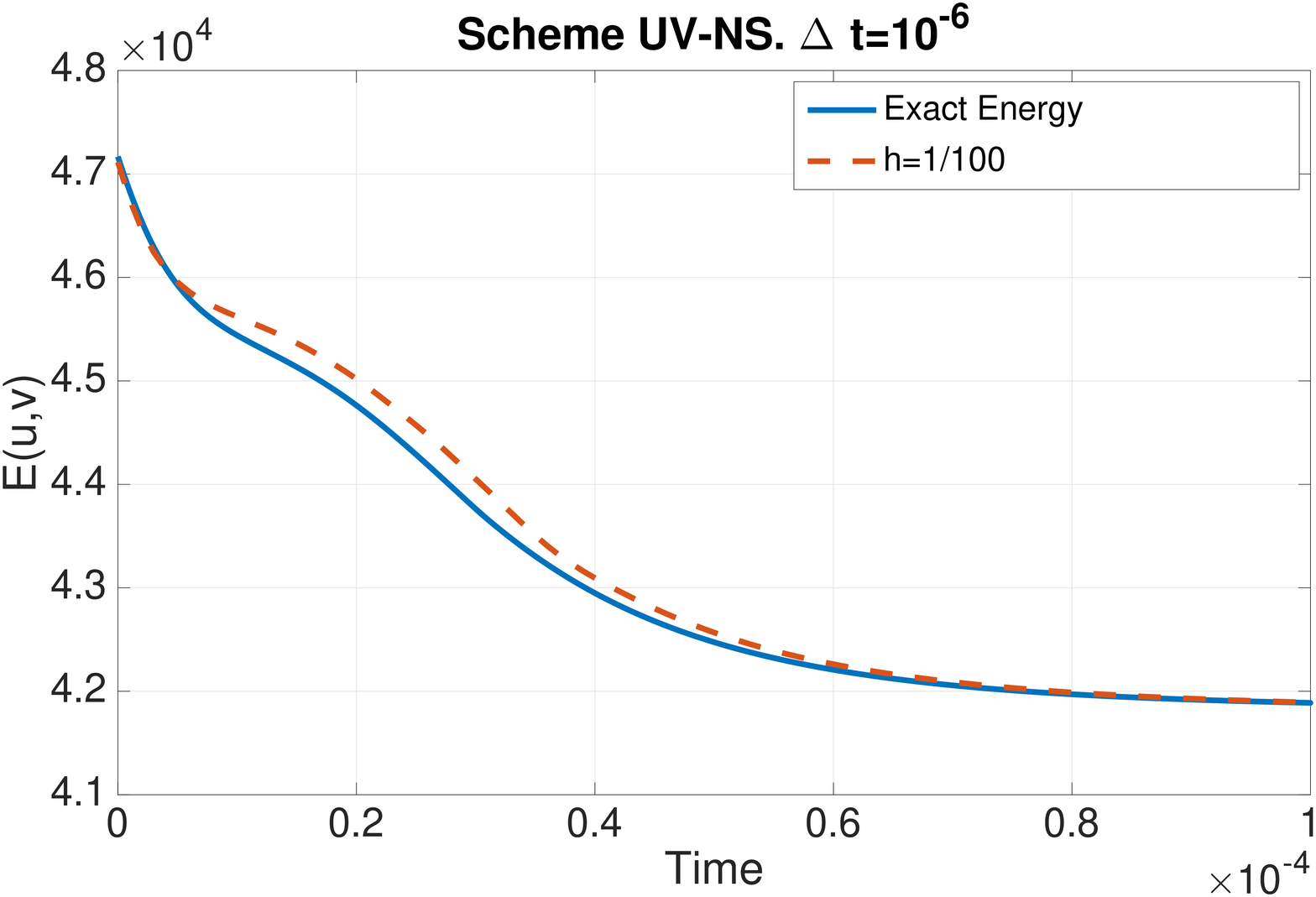}
\includegraphics[width=0.43\textwidth]{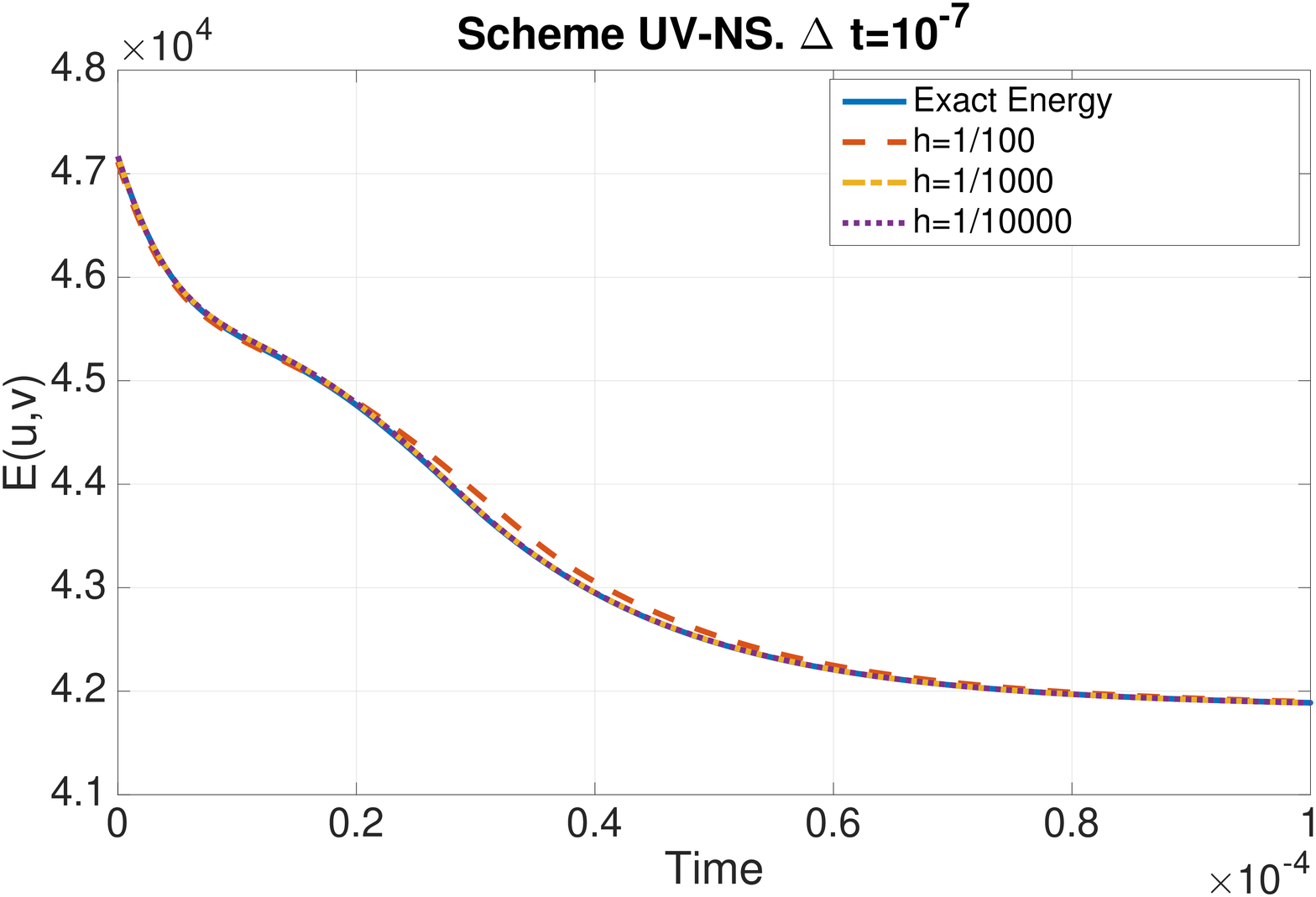}
\\
\includegraphics[width=0.43\textwidth]{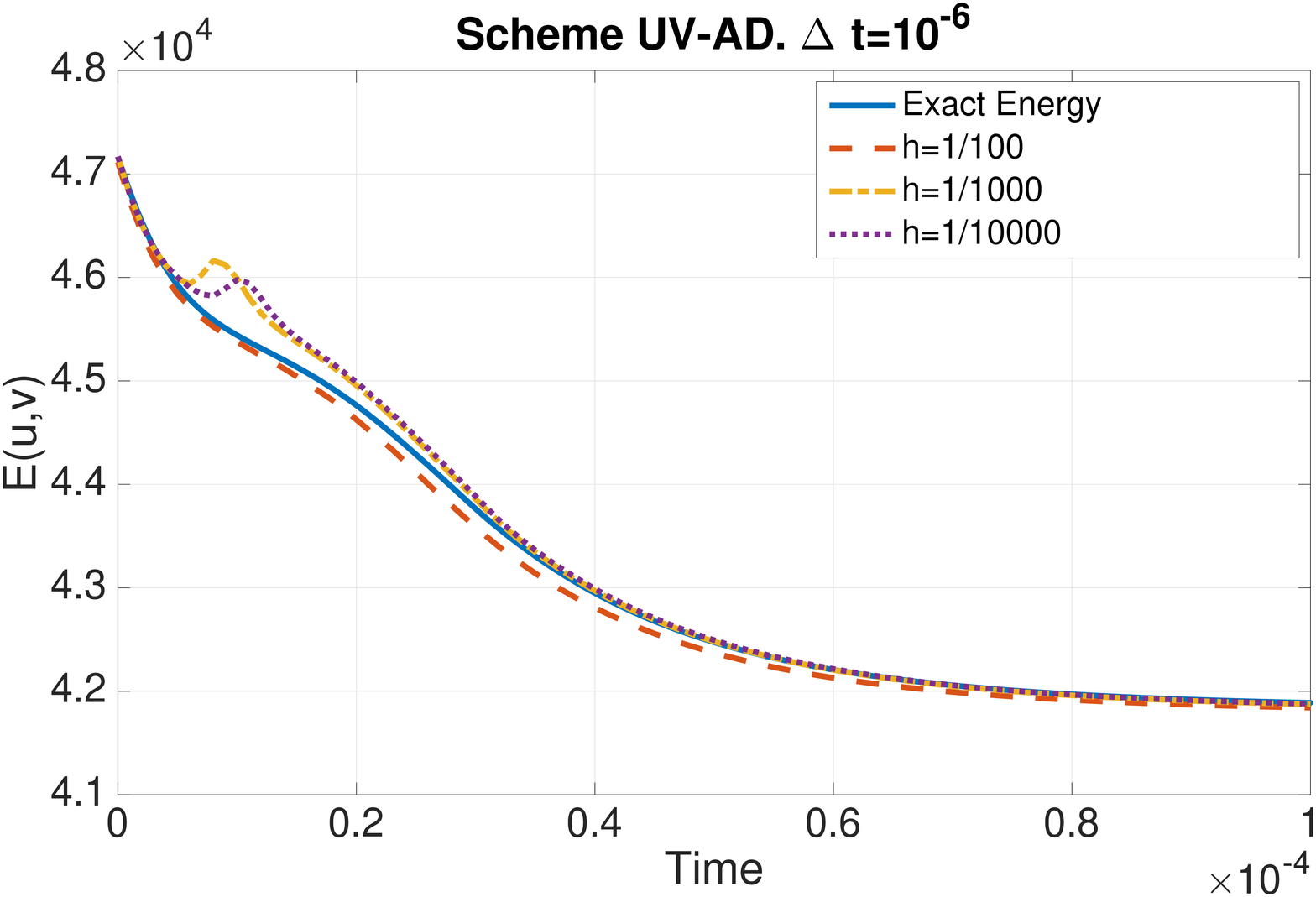}
\includegraphics[width=0.43\textwidth]{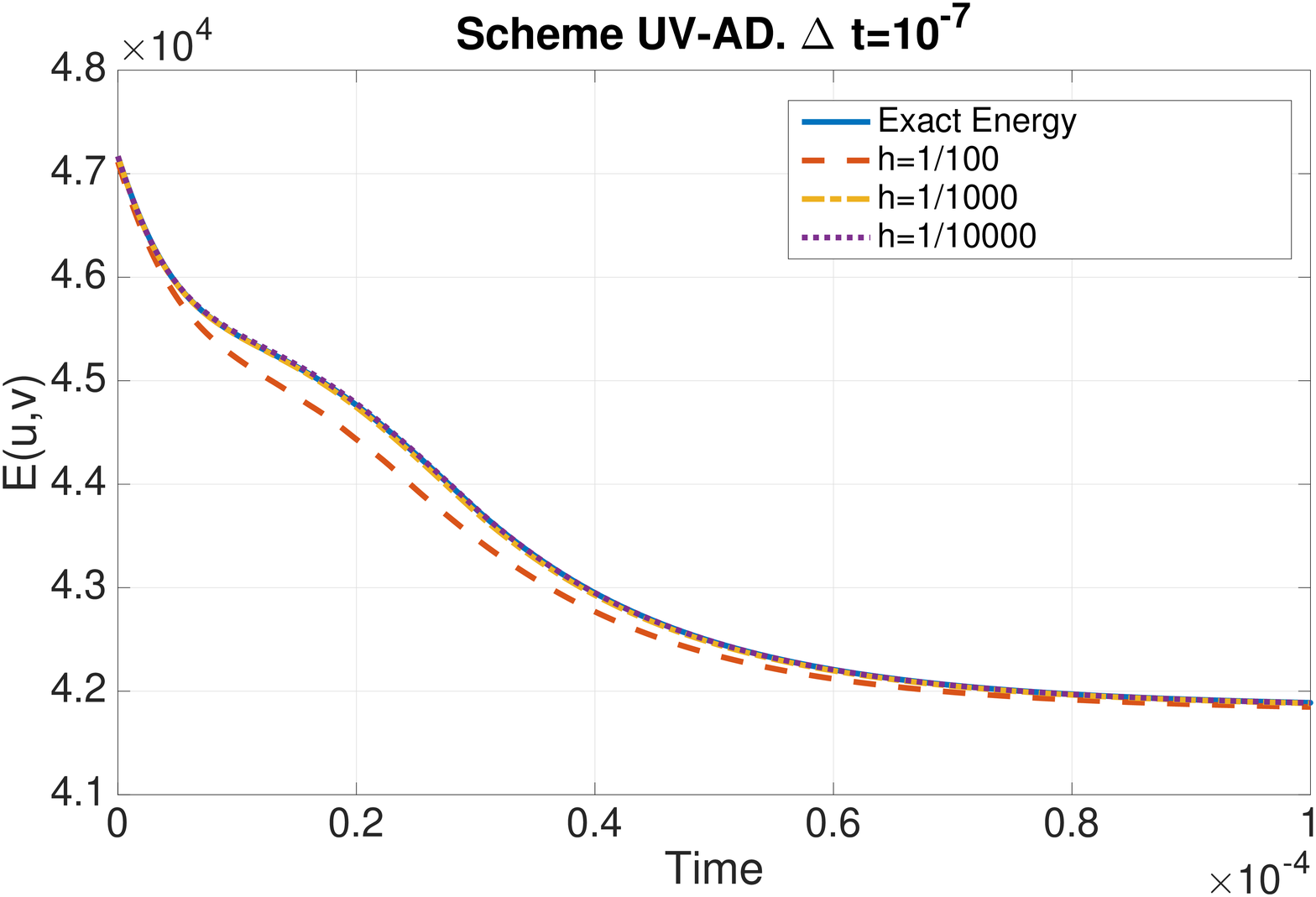}
\caption{Example III: Comparison of the evolution of the energy $E(u,v)$ for Schemes \textbf{UV}, \textbf{UV-ND}, \textbf{UV-NS} and \textbf{UV-AD} (from first to fourth row, respectively)  using different spatial meshes with $\Delta t=10^{-6}$ (Left) and $\Delta t=10^{-7}$ (Right) (schemes \textbf{UV-ND} and \textbf{UV-NS} do not converge for $\Delta t=10^{-6}$ with $h=10^{-3}$ and $h=10^{-4}$). } \label{fig:Energy_UVs_C}
\end{center}
\end{figure}

\begin{figure}
\begin{center}
\includegraphics[width=0.43\textwidth]{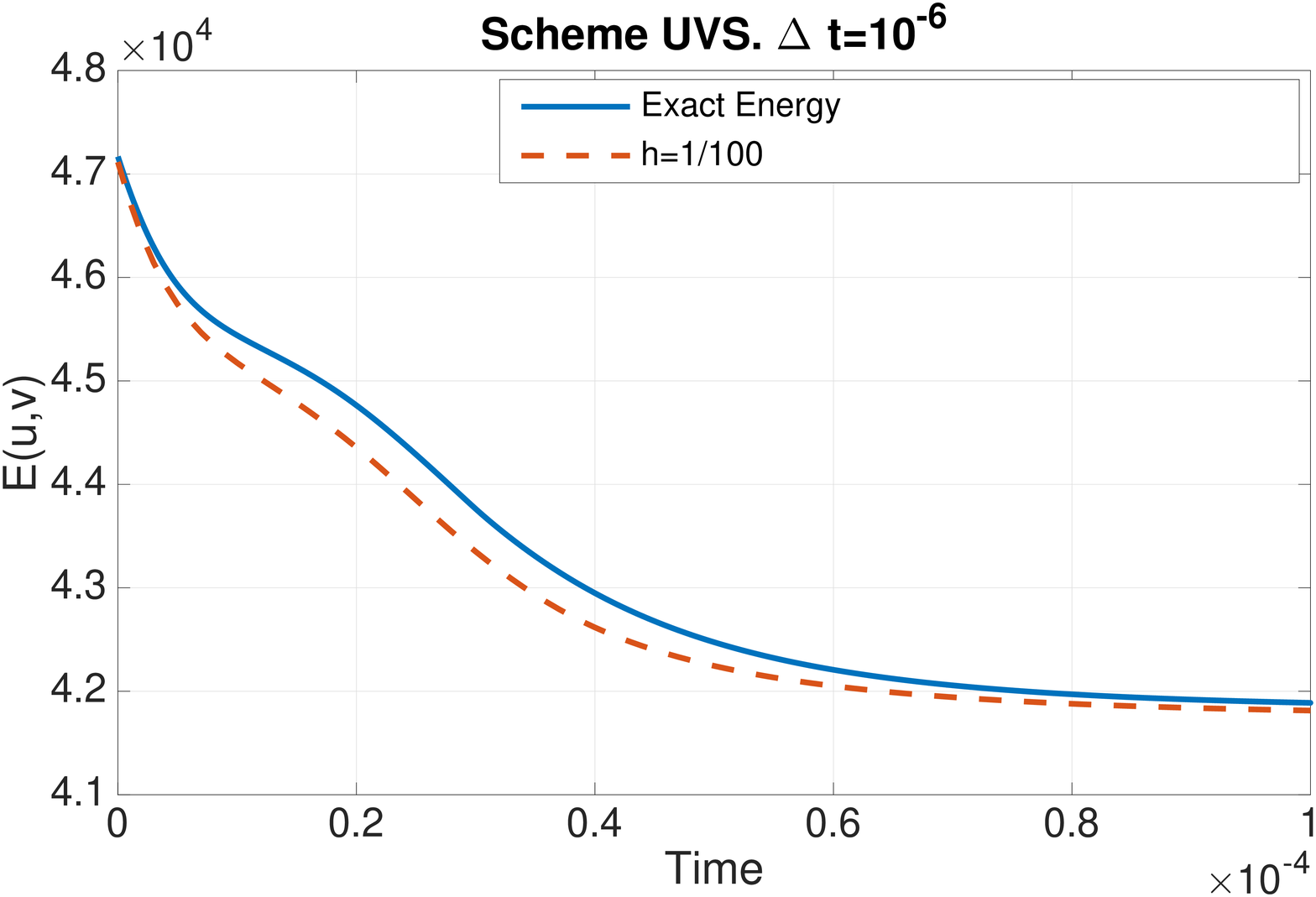}
\includegraphics[width=0.43\textwidth]{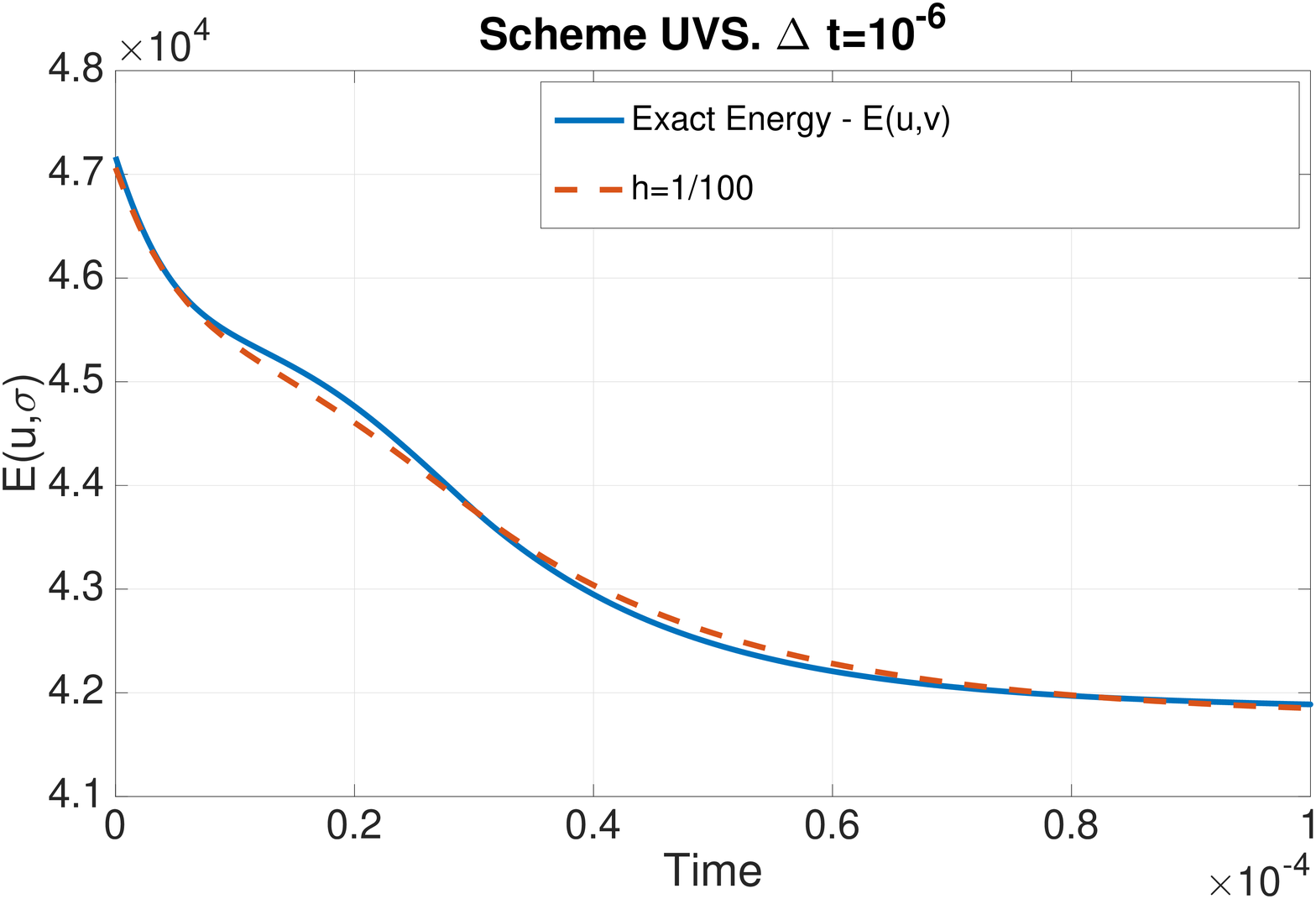}
\\
\includegraphics[width=0.43\textwidth]{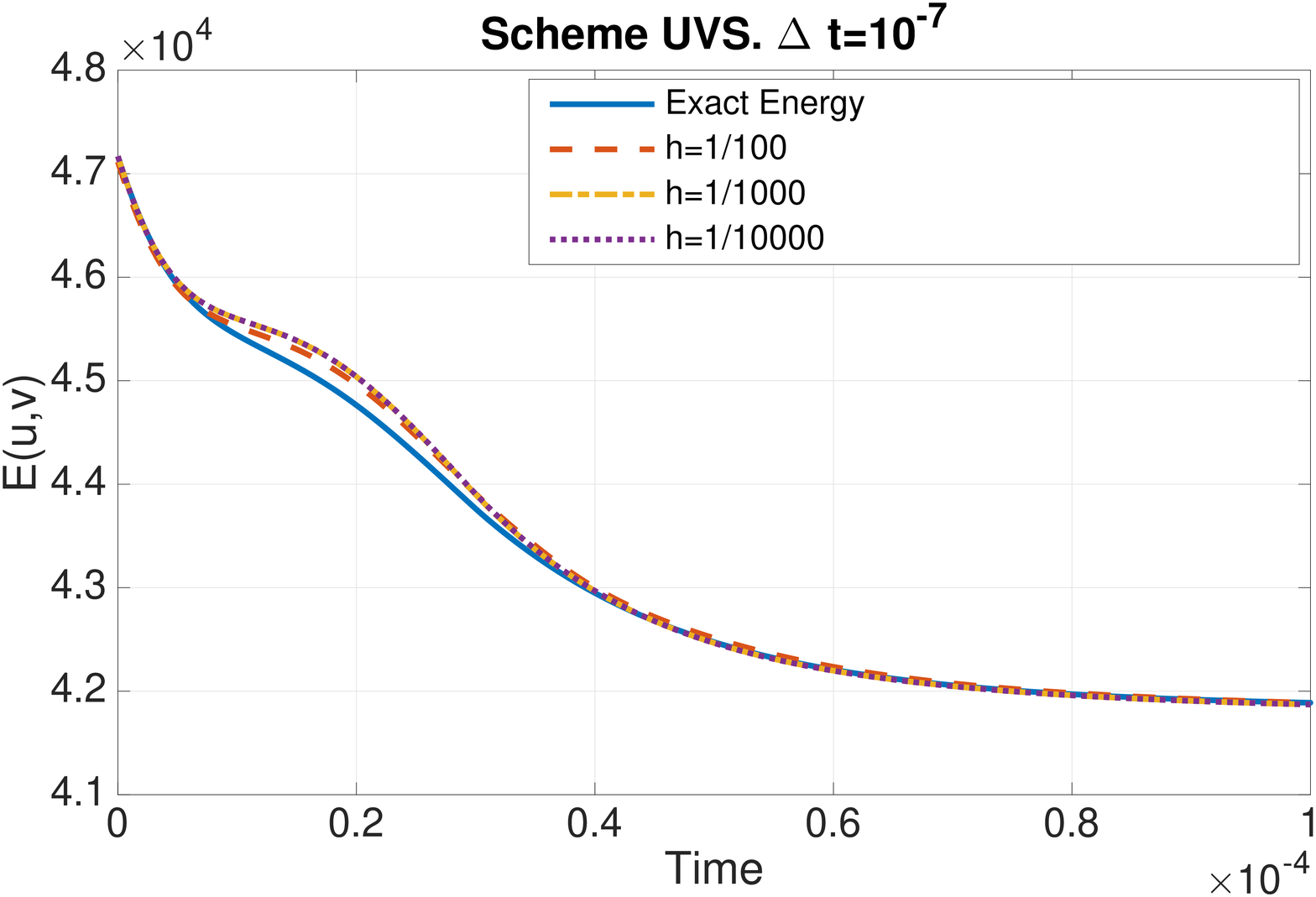}
\includegraphics[width=0.43\textwidth]{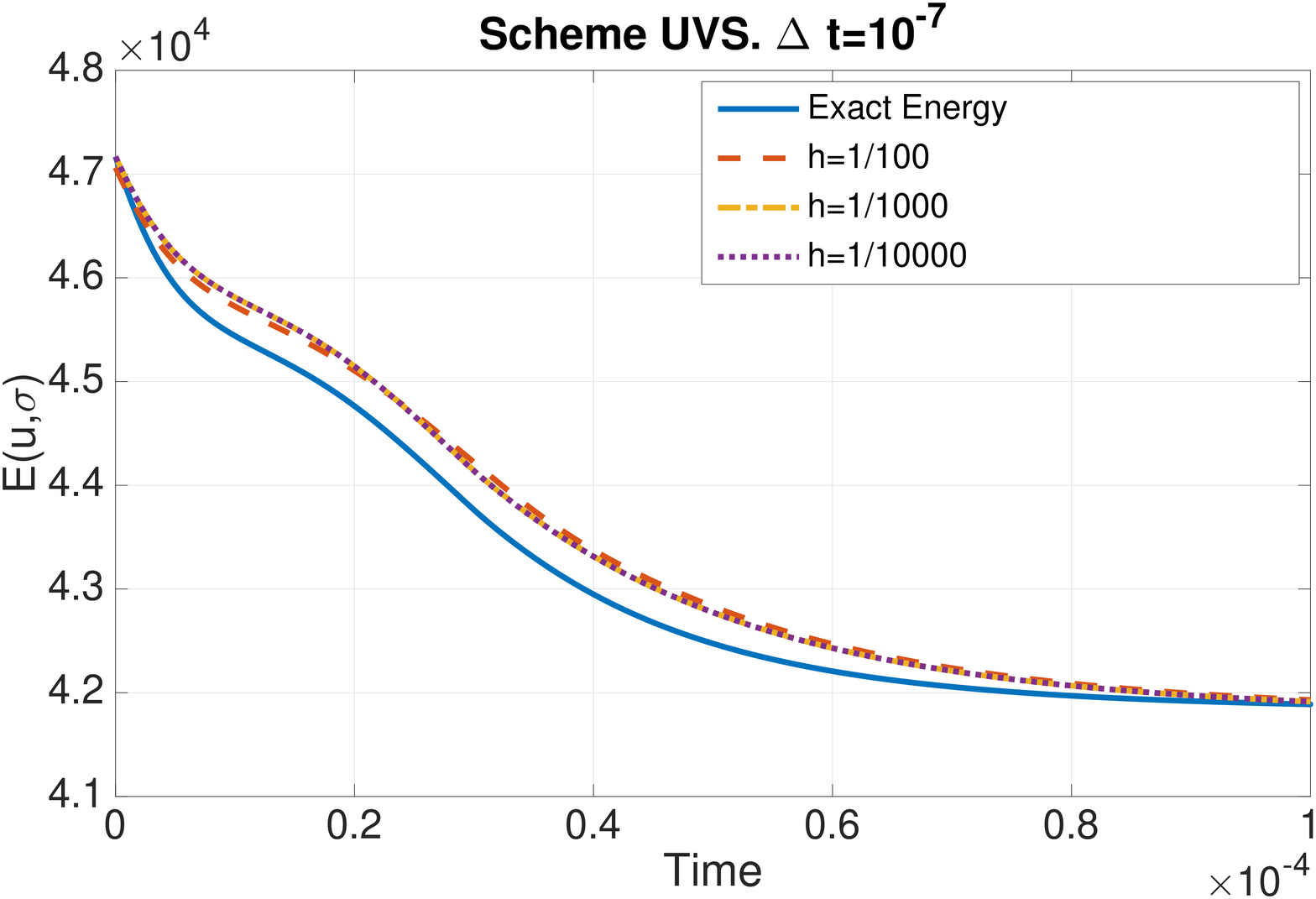}
\caption{Example III: Comparison of the evolution of the energies $E(u,v)$ (Left) and $E(u,\sigma)$ (Right) for Scheme \textbf{UVS} using different spatial meshes with $\Delta t=10^{-6}$ (Top) and $\Delta t=10^{-7}$ (Bottom) (schemes do not converge for $\Delta t=10^{-6}$ with $h=10^{-3}$ and $h=10^{-4}$). } \label{fig:Energy_UVS_C}
\end{center}
\end{figure}

\begin{figure}
\begin{center}
\includegraphics[width=0.32\textwidth]{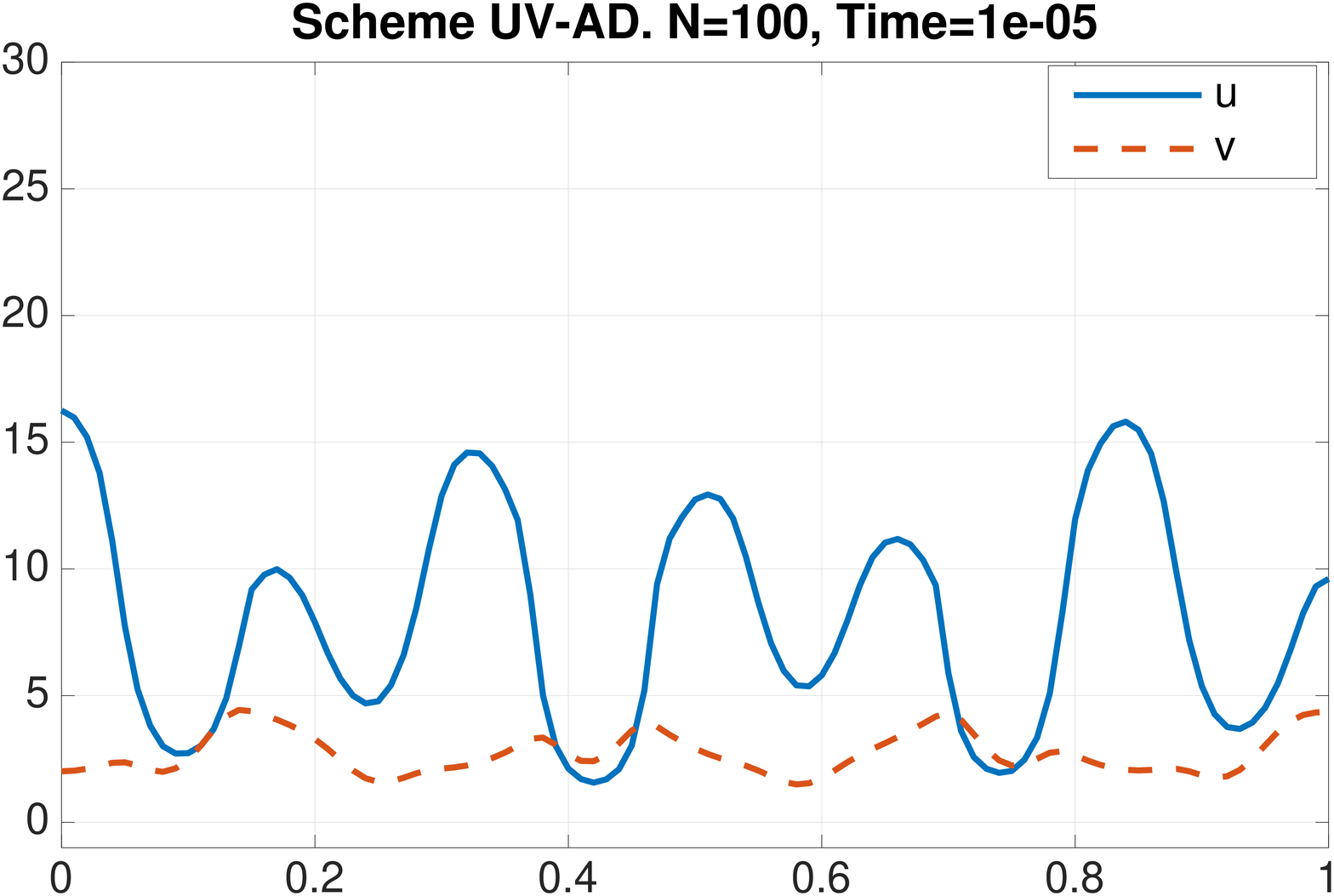}
\includegraphics[width=0.32\textwidth]{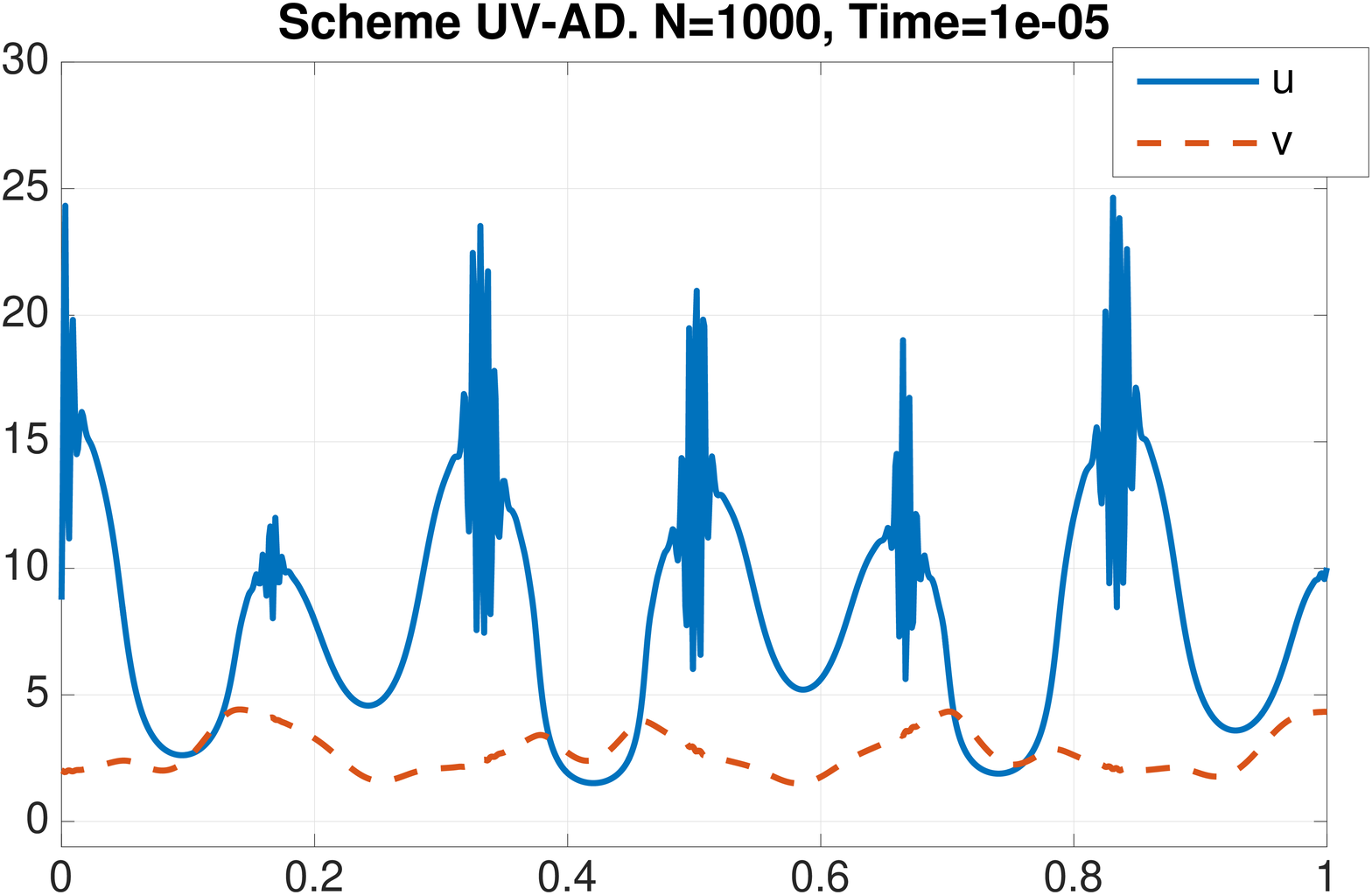}
\includegraphics[width=0.32\textwidth]{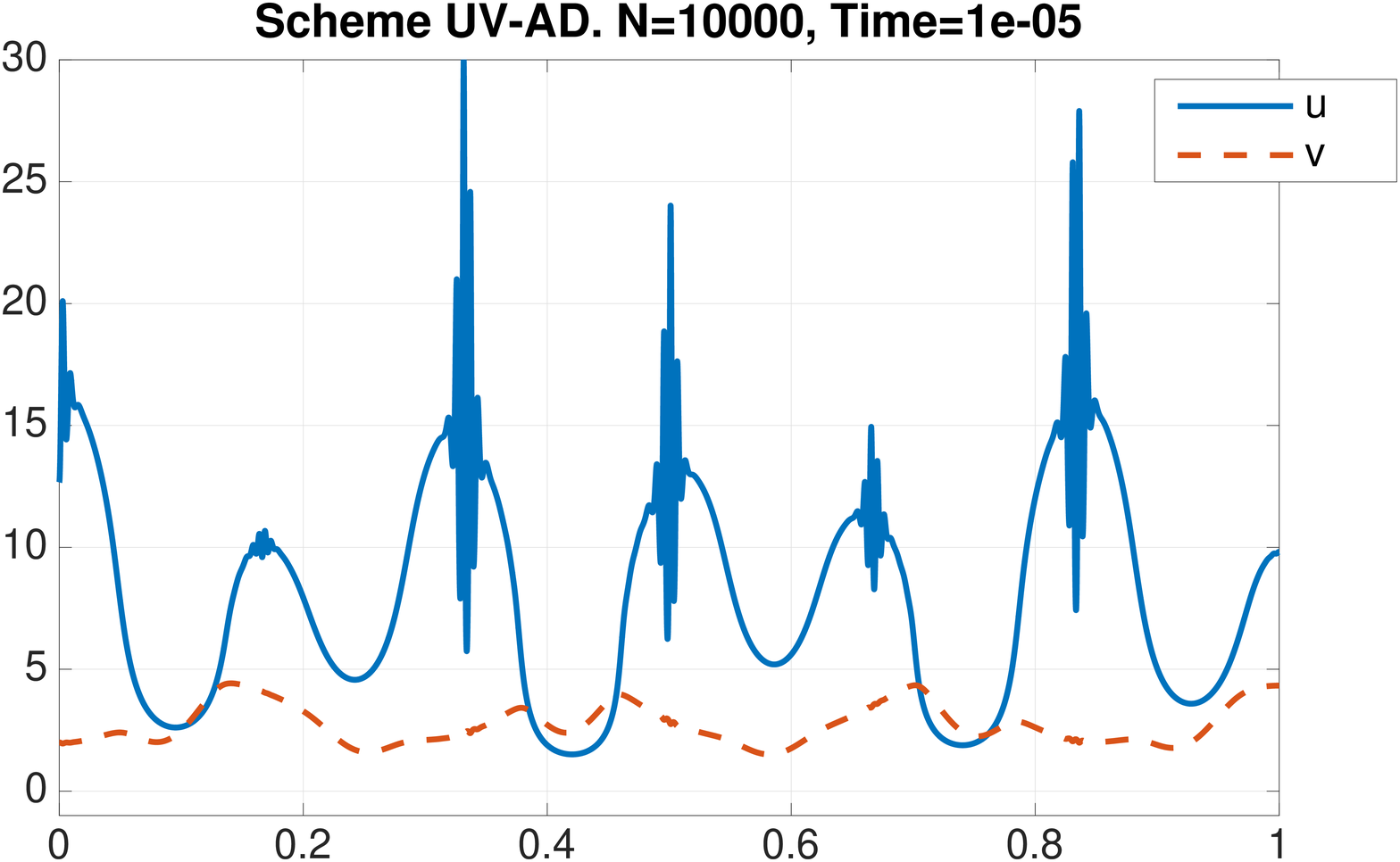}
\caption{Example III: Scheme \textbf{UV-AD} at time $t=10^{-5}$ for $\Delta t =10^{-6}$ with $h=1/100$ (left), $h=1/1000$ (center) and $h=1/10000$ (right).} \label{fig:UVAD_oscillations}
\end{center}
\end{figure}

\subsection{Example IV: Approximation and error estimates test}

In this example we perform a numerical error estimate study in space for each of the schemes presented in the manuscript. The initial configuration and the physical parameters considered for this test are: 
\beq\label{eq:D_initialdynconst}
\left\{\ba{rcl}
u_0&=&3(1.0001 + \cos(8\pi x))\,,
\\ \hueco
v_0&=&5(1.0001 + \cos(7\pi x))\,,
\ea\right.
\quad
[0,T]\,=\,[0,10^{-4}]\,,
\quad
\chi\,=\,10
\quad
\mbox{ and }
\quad \mu=1500\,.
\eeq
%

We will compute the EOC (Experimental Order of Convergence) taking $\Delta t=10^{-9}$ and using as reference (or exact) solution the one obtained by solving the system using scheme \textbf{UV} with spatial discretization parameter $h=10^{-5}$ at final time $T=10^{-4}$. The very non-trivial dynamics of this system are presented in Figure~\ref{fig:Dynamic_UV_D}. We observe in the dynamics that the cell population density $u$ moves towards the regions of high concentration of chemical substance, which is consumed at a very high rate, producing that the location of high concentration regions of chemical substance $v$ changes fast in time, so the cell population density needs to rearrange itself also in a very fast way in order to move towards the new locations with high concentration of substance $v$.

\begin{figure}
\begin{center}
\includegraphics[width=0.32\textwidth]{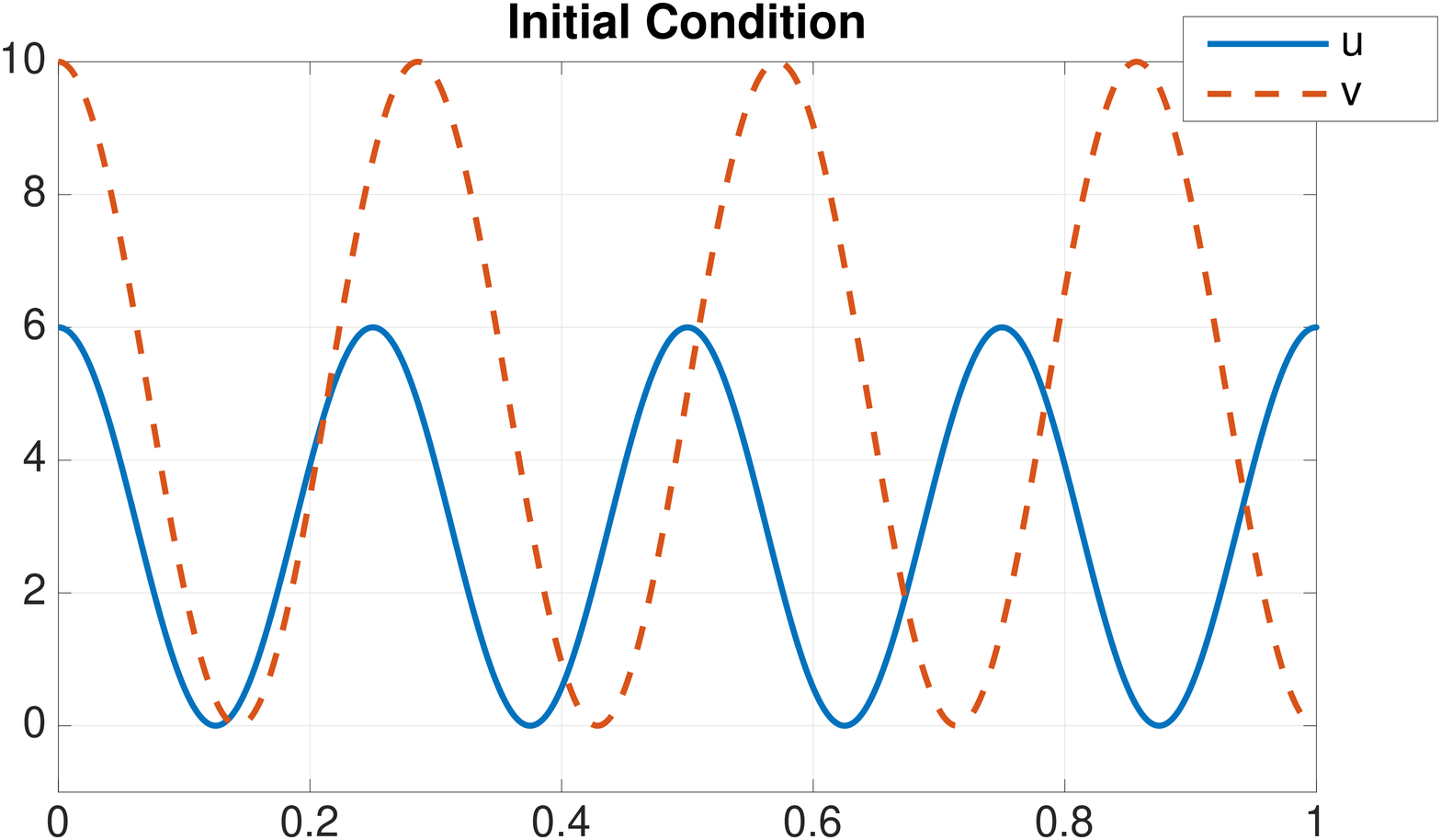}
\includegraphics[width=0.32\textwidth]{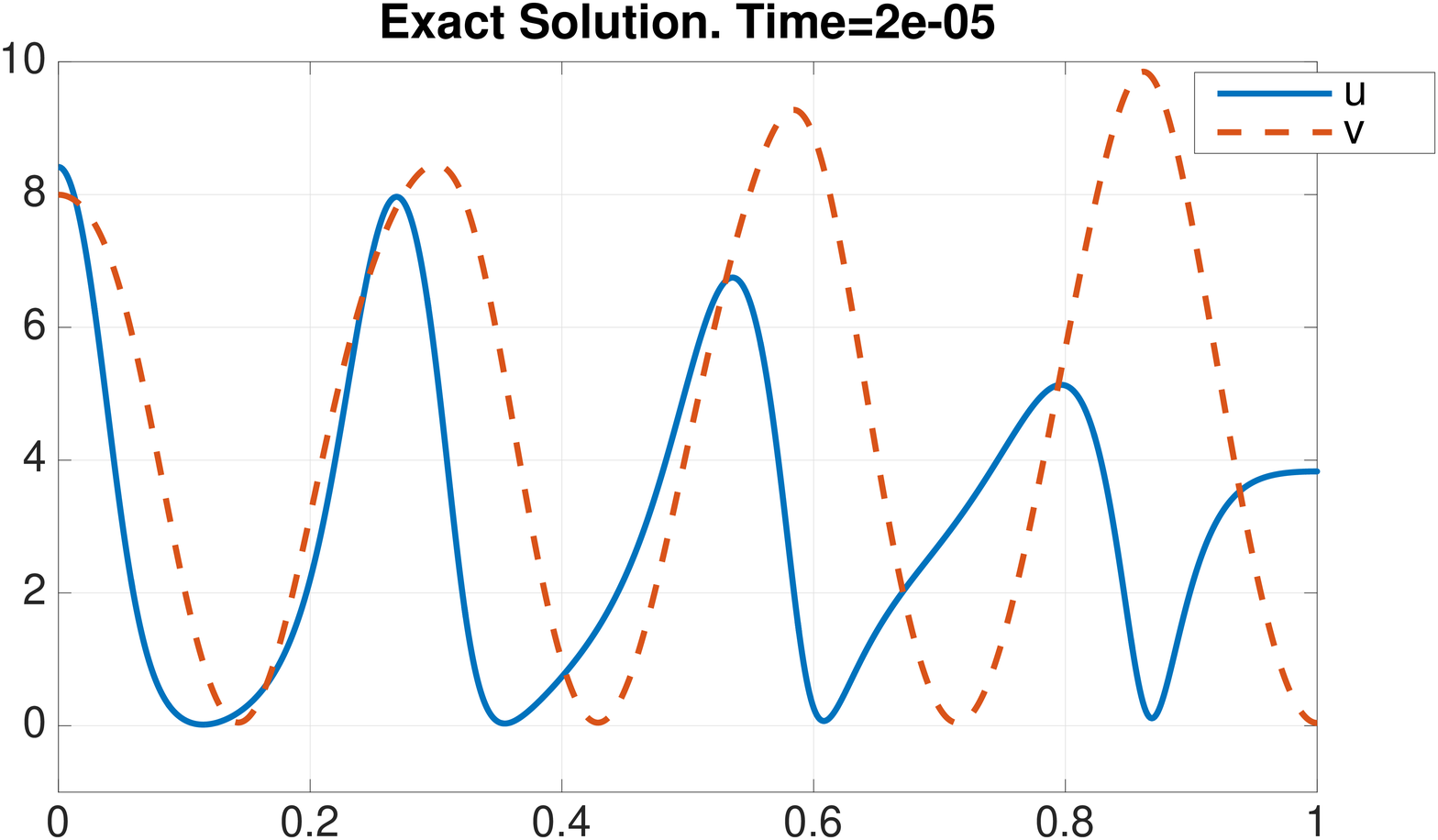}
\includegraphics[width=0.32\textwidth]{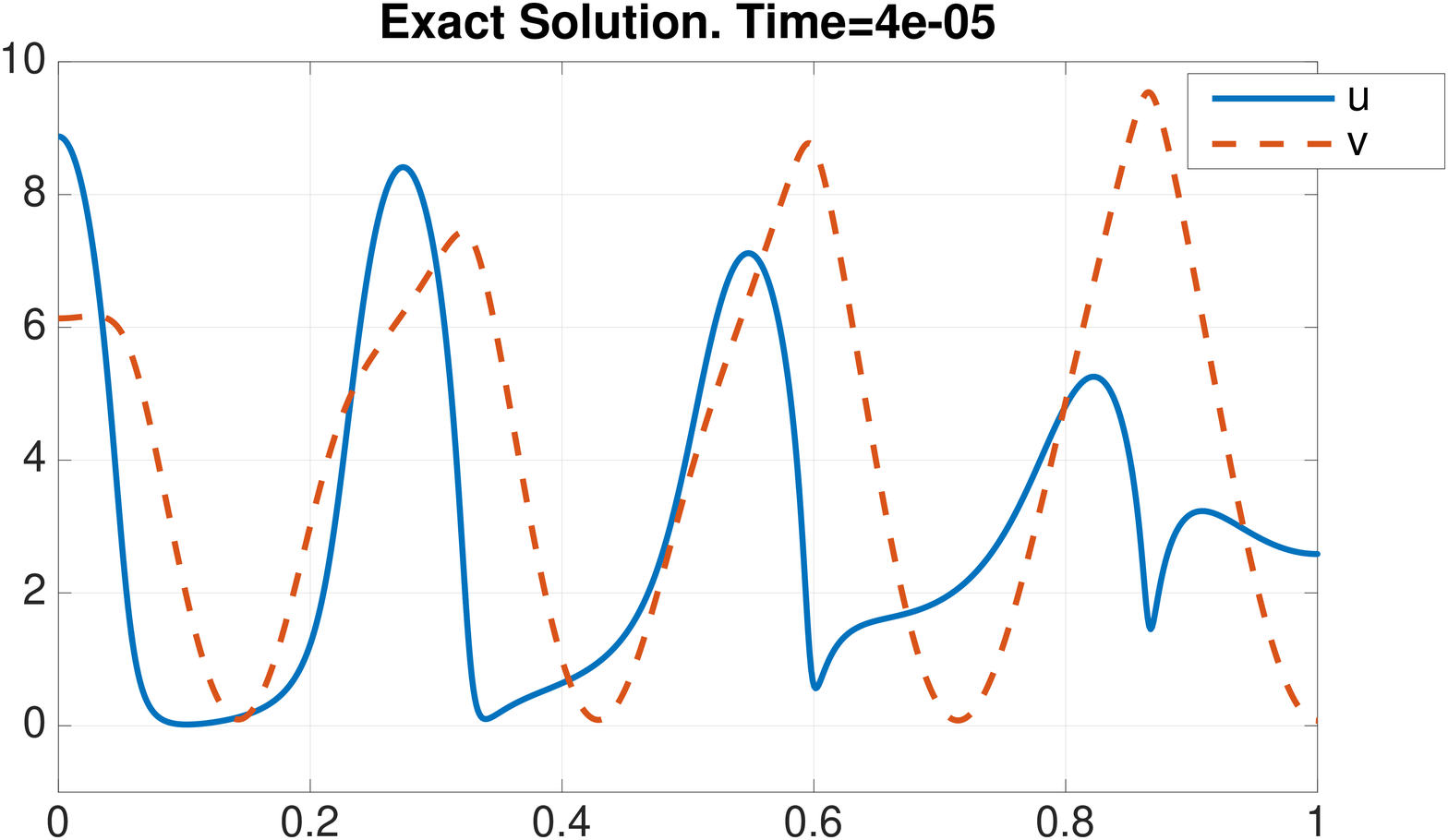}
\\
\includegraphics[width=0.32\textwidth]{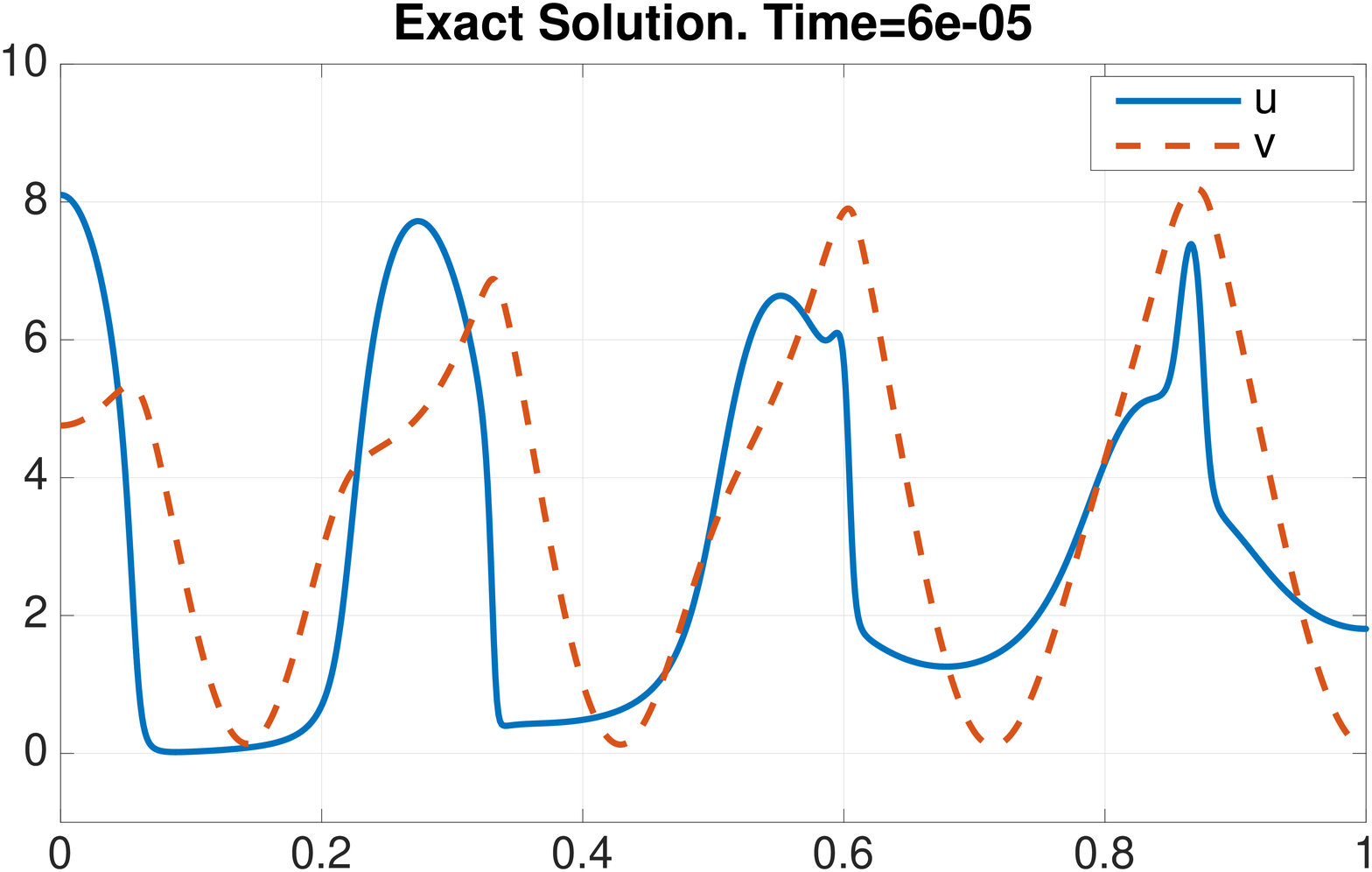}
\includegraphics[width=0.32\textwidth]{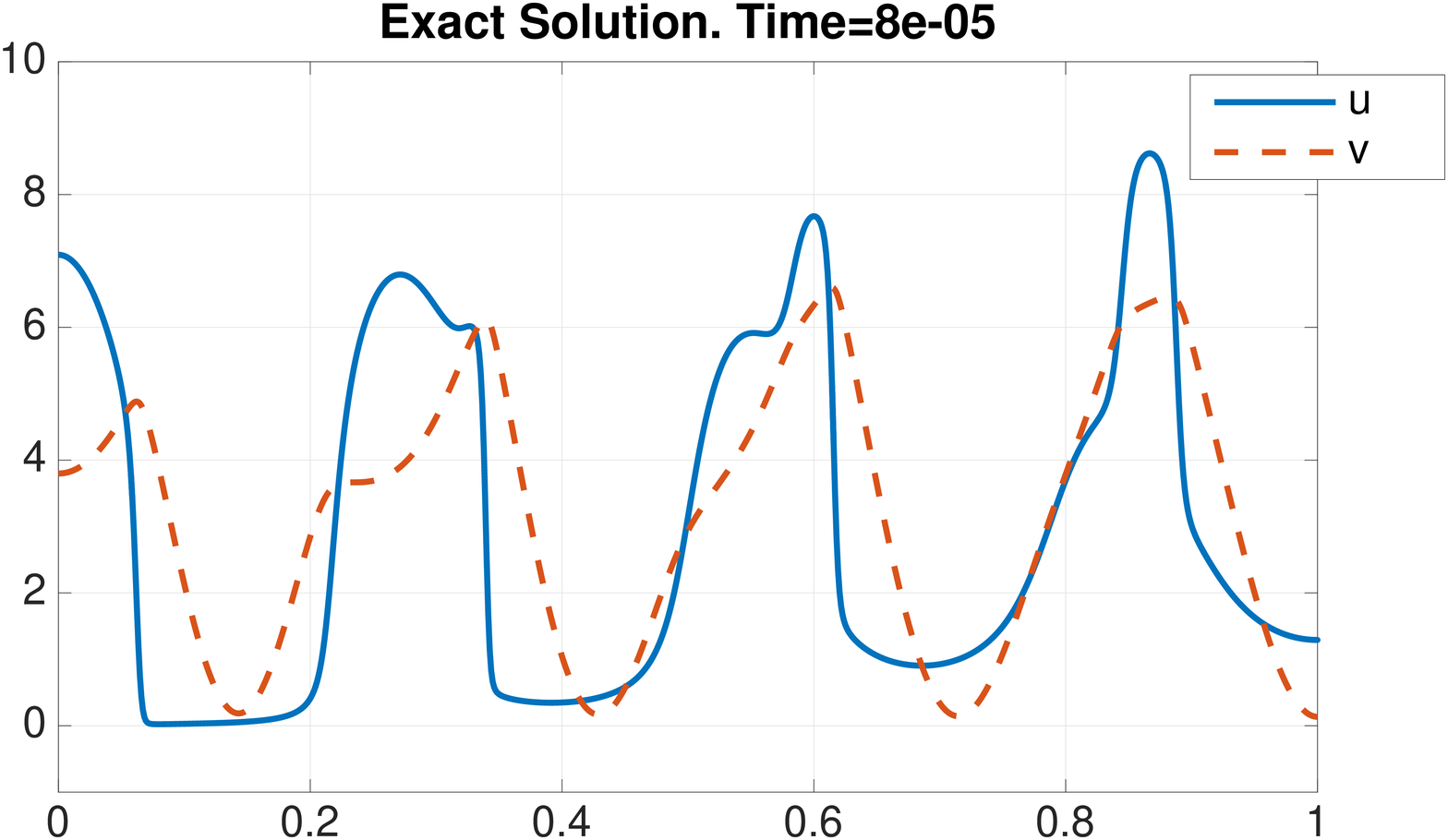}
\includegraphics[width=0.32\textwidth]{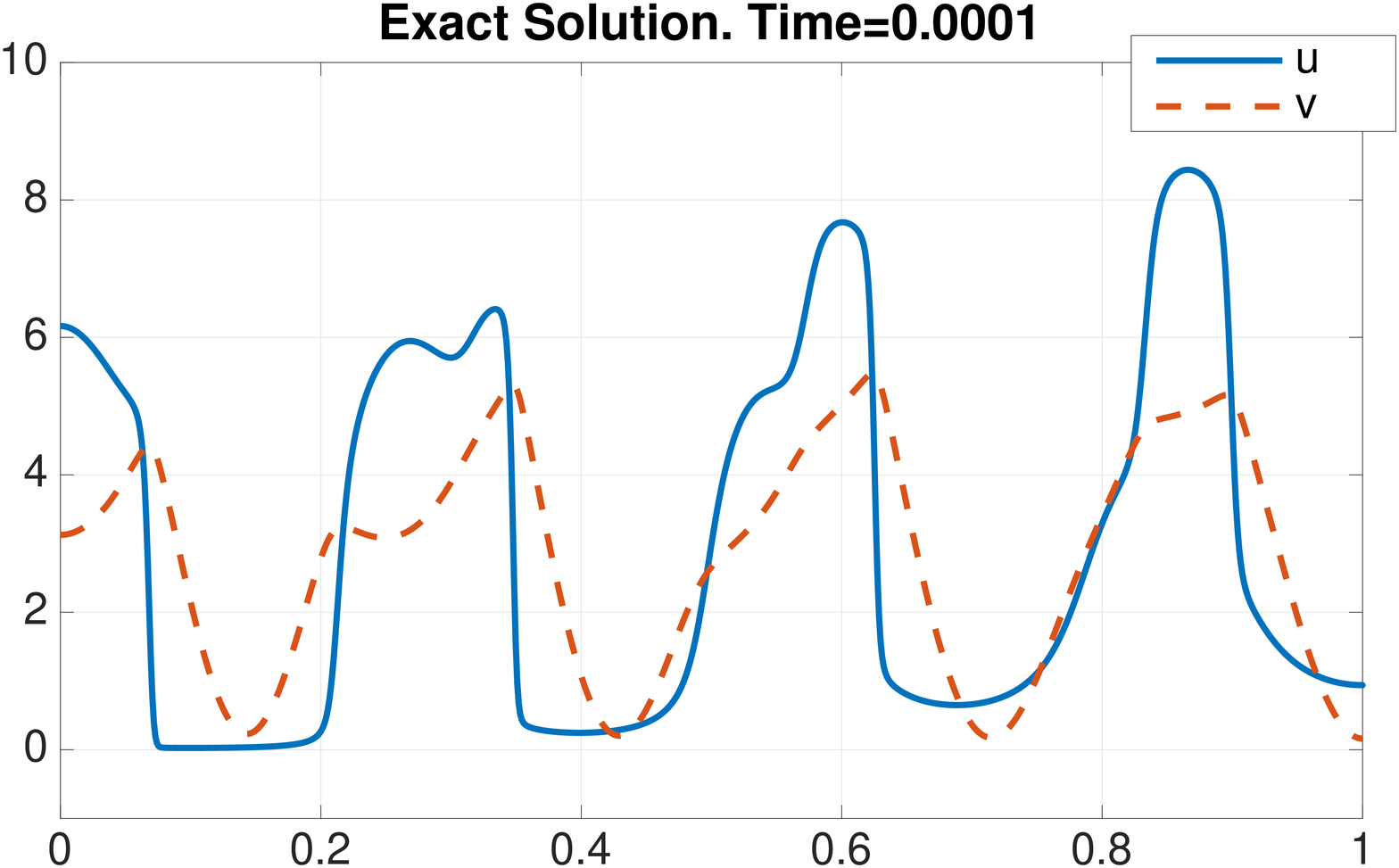}
\caption{Example IV: Dynamic of scheme \textbf{UV} using data in \eqref{eq:D_initialdynconst}.} \label{fig:Dynamic_UV_D}
\end{center}
\end{figure}

%
%
%
%

We now introduce some additional notation. The individual errors using discrete norms and the convergence rate between two consecutive meshes of size $h$ and $\tilde{h}$ are defined as:
$$
e(\phi):=\|\phi_{exact} - \phi_h\|_{L^2(\Omega)}
\quad
\mbox{ and }
\quad
r(\phi):=\left[\log\left(\frac{e(\phi)}{\tilde{e}(\phi)}\right)\right]\left[\log\left(\frac{h}{\tilde{h}}\right)\right]^{-1}\,.
$$

The convergence history for a sequence of triangulations 
for all the schemes is presented in Table~\ref{tab:Orders_UVs}. 
We can observe that schemes \textbf{UV}, \textbf{UV-ND}, \textbf{UV-NS}  and \textbf{UV-AD} show order of convergence, and between these four schemes it is clear that \textbf{UV}, \textbf{UV-ND} and \textbf{UV-NS} performs better than \textbf{UV-AD}, due to the fact that this last scheme introduces extra numerical dissipation to help the scheme to achieve the positivity of $u$, but at the same time this numerical dissipation prevents the scheme to achieve a higher order of convergence. 
On the other hand, scheme \textbf{UVS} does not seem to achieve order of convergence when compared with a reference solution computed using scheme \textbf{UV} (although the errors are small, for instance smaller than the ones obtained with scheme \textbf{UV-AD}). But it achieves optimal orders when the study is performed using  a reference solution computed using scheme  \textbf{UVS}  itself with $h=10^{-5}$. These results make evident that the scheme \textbf{UVS}  is well behaved although it needs very small discretization parameters to achieve exactly the same solution than the one obtained by the other schemes.

\begin{table}
\begin{center}
\begin{tabular}{|c|cc|cc|cc|}
\hline                                               
h & $e(u)$  & $r(u)$  & $e(v)$ & $r(v)$ & $e(v_x)$ & $r(v_x)$
\\ \hline    
\multicolumn{7}{|c|}{Scheme \textbf{UV}}
\\ \hline    
1/200 & $0.0344$ & $ - $ & $0.0046$ & $ - $ & $0.7893$ & $ - $
\\
1/400 & $0.0079$ & $2.1171$ & $0.0012$  & $1.9859$ & $0.2038$ & $1.9536$ 
\\
1/600 & $0.0035$ & $2.0093$ & $ 0.0005$ & $1.9943$ & $0.0904$ & $2.0052$  
\\
1/800 & $0.0020$ & $1.9916$ & $ 0.0003$  & $1.9968$ & $ 0.0508$ & $2.0037$
\\
1/1000  & $0.0013$ & $2.0000$ & $ 0.0002$ & $1.9980$ & $0.0325$ & $2.0025$
\\ \hline
\multicolumn{7}{|c|}{Scheme \textbf{UV-ND}}
\\ \hline  
1/200 & $0.0355$ & $ - $ & $0.0032$ & $ - $ & $0.6240$ & $ - $
\\
1/400 & $0.0119$ & $1.5785$ & $0.0009$  & $1.8814$ & $0.1772$ & $1.8162$ 
\\
1/600 & $0.0060$ & $1.6656$ & $0.0004$ & $1.8811$ & $0.0852$ & $1.8049$  
\\
1/800 & $0.0037$ & $1.7198$ & $0.0002$  & $1.9069$ & $ 0.0500$ & $1.8564$
\\
1/1000  & $0.0025$ & $1.8046$ & $0.0001$ & $1.9353$ & $0.0327$ & $  1.8958 $
\\ \hline
\multicolumn{7}{|c|}{Scheme \textbf{UV-NS}}
\\ \hline
1/200 & $0.0298$ & $ - $ & $ 0.00400$ & $ - $ & $0.7671$ & $ - $
\\
1/400 & $0.0074$ & $2.0132$ & $ 0.00099$  & $2.0069$ & $0.1946$ & $1.9787$ 
\\
1/600 & $0.0032$ & $2.0473$ & $  0.00044$ & $1.9965$ & $0.0859$ & $2.0163$  
\\
1/800 & $0.0018$ & $1.9928$ & $ 0.00024$  & $1.9981$ & $0.0482$ & $2.0071$
\\
1/1000  & $0.0012$ & $1.9987$ & $0.00015$ & $1.9991$ & $0.0308$ & $2.0044$
\\ \hline
\multicolumn{7}{|c|}{Scheme \textbf{UV-AD}}
\\ \hline
1/200 & $0.1032$ & $ - $ & $0.0232$ & $ - $ & $1.4694$ & $ - $
\\
1/400 & $0.0558$ & $0.8865 $ & $0.0124$  & $0.9029$ & $0.8023$ & $0.8730$ 
\\
1/600 & $0.0384$ & $0.9197$ & $0.0085$ & $0.9405 $ & $0.5568$ & $0.9010$  
\\
1/800 & $0.0294 $ & $0.9321$ & $0.0064$  & $0.9569$ & $0.4275$ & $0.9185$
\\
1/1000  & $0.0238$ & $0.9388$ & $0.0052$ & $0.9662$ & $0.3472$ & $0.9328$
\\ \hline
\multicolumn{7}{|c|}{Scheme \textbf{UVS} }
\\
\multicolumn{7}{|c|}{Reference solution computed with scheme \textbf{UV}}\\ \hline                           
1/200 & $0.0782$ & $ - $ & $ 0.0071$ & $ - $ & $0.9256$ & $ - $
\\
1/400 & $0.0311$ & $1.3284$ & $0.0037$  & $0.9605$ & $0.2546$ & $1.8621$ 
\\
1/600 & $0.0258$ & $0.4587$ & $  0.0032$ & $0.3135$ & $0.1353$ & $1.5588$  
\\
1/800 & $0.0247$ & $0.1545$ & $0.0031$  & $0.1274$ & $0.0998$ & $1.0568$
\\
1/1000  & $ 0.0244$ & $0.0587$ & $0.0031$ & $0.0580$ & $0.0868$ & $0.6286$
\\ \hline
\multicolumn{7}{|c|}{Scheme \textbf{UVS} }
\\
\multicolumn{7}{|c|}{Reference solution computed with scheme \textbf{UVS}}
\\ \hline                                                          
1/200 & $0.0757$ & $ - $ & $ 0.0063$ & $ - $ & $0.9193$ & $ - $
\\
1/400 & $0.0195$ & $1.9564$ & $0.0018$  & $1.8221$ & $0.2368$ & $1.9565$ 
\\
1/600 & $0.0087$ & $1.9990$ & $  0.0008$ & $1.9479$ & $0.1059$ & $1.9859$  
\\
1/800 & $0.0048$ & $2.0317$ & $0.0005$  & $2.0183$ & $0.0599$ & $1.9800$
\\
1/1000  & $ 0.0031$ & $1.9533$ & $0.0003$ & $1.9576$ & $0.0384$ & $1.9959$
\\ \hline
\end{tabular}
\caption{Example IV: Experimental absolute errors and order of convergences for schemes \textbf{UV}, \textbf{UV-ND}, \textbf{UV-NS}, \textbf{UV-AD} and \textbf{UVS}
}\label{tab:Orders_UVs}
\end{center}
\end{table}

\section{Conclusions}\label{sec:conclusions}

In this work we have proposed and studied numerical schemes to approximate a chemo-attraction and consumption model, and we have compared them numerically in $1D$ domains.  
%
We have focused on designing schemes in such a way that they maintain  the main properties of the continuous problem at the discrete level. In fact, the three more challenging properties that we have identified and focused on are: $(a)$ positivity, $(b)$ dissipative energy law \eqref{eq:EL-uv} and $(c)$ an estimate of a singular functional \eqref{eq:regresult}. 
We have developed several schemes: 
\begin{itemize}
\item \textbf{Schemes UV-ND and UV-NS.} Satisfying discrete versions of $(a)$ and $(c)$. 
\item \textbf{Scheme UVS.} Satisfying a discrete version of $(a)$ and $(b)$. 
\end{itemize}
We have compared these schemes with a straightforward \textbf{Scheme UV} 
and with \textbf{Scheme UV-AD} (an upwind Finite Volume scheme known to satisfy $(a)$ and that it can be reinterpreted as a Finite Element scheme with artificial numerical dissipation). 

\

The numerical tests reported in this work have been designed to check how well the schemes behave with respect to positivity approximation, the approximation of the evolution in time of the energy and the behavior of the error as the size of the spatial mesh is reduced (error estimates).

\

Scheme \textbf{UVS} has been designed to satisfy the energy property $(b)$, having to rewrite the original system introducing regularization terms that to be properly approximated they might
need small choices of the discrete parameters. The numerical results illustrate that the scheme perform well in all the tests if the discretization parameters are small enough (as small as the one needed for the rest of the schemes), although it is not able to exactly capture the decreasing evolution of the energy $E(u,v)$. This fact it is not a surprise because the scheme \textbf{UVS} is designed to satisfy a relation for a modified energy $E(u,\bsigma)$, which will only approximates well to $E(u,v)$ for very small discretization parameters.
The other four schemes (\textbf{UV}, \textbf{UV-ND}, \textbf{UV-NS}  and \textbf{UV-AD}) performs well in the three proposed tests but we need to remark that schemes 
\textbf{UV}, \textbf{UV-ND} and \textbf{UV-NS} performs much better than \textbf{UV-AD} in the approximation test, because 
schemes \textbf{UV}, \textbf{UV-ND} and \textbf{UV-NS} show second order convergence for all the unknowns, while \textbf{UV-AD} achieves only first order 
(due to the artificial dissipation introduced to maintain the positivity). Moreover, we have detected that some spurious  oscillations might appear when using scheme \textbf{UV-AD} with a not very small time step, due to the fact that \textit{upwind} schemes can have problems when the transport velocity is not incompressible. 
\
Although both schemes \textbf{UV-ND} and \textbf{UV-NS} perform equally well in the numerical tests, it is worth to mention that scheme \textbf{UV-ND} is much more flexible than scheme \textbf{UV-NS}, because the latter needs the requirement of considering structured meshes in order to satisfy the derived properties.

\

Finally, this work emphasizes two interesting points that can be extended to other problems where there are dissipative energy laws and maximum/minimum principles involved: 
\\
(I) For numerical schemes that satisfy a energy stability property 
with respect to a modified energy, it is important to keep in mind that even if in the continuous case the two energies are equivalent, the discrete versions will coincide only if the discretization parameters are really small. 
\\
(II) Finite Volume schemes achieve positivity by introducing artificial numerical dissipation in the system, but this dissipation can interfere with the accuracy of the scheme. Moreover, when a transport term appears with no incompressible velocity and the time step is not carefully chosen, the solutions can exhibit nonphysical oscillations.

\end{document}